\documentclass{article}

\usepackage{arxiv}

\usepackage[utf8]{inputenc} 
\usepackage[T1]{fontenc}    
\usepackage[hypertexnames=true]{hyperref}       
\usepackage{url}            
\usepackage{booktabs}       
\usepackage{amsmath, amsfonts}       
\usepackage{amsthm}
\usepackage{nicefrac}       
\usepackage{microtype}      
\usepackage{cleveref}       
\usepackage{graphicx}
\usepackage{doi}

\usepackage{amssymb}
\usepackage{algorithmic}
\usepackage[ruled,linesnumbered,vlined]{algorithm2e}
\usepackage{array}
\usepackage[caption=false,font=normalsize,labelfont=sf,textfont=sf]{subfig}
\usepackage{textcomp}
\usepackage{stfloats}
\usepackage{verbatim}
\usepackage{cite}
\usepackage{multirow} 
\usepackage{scalerel}
\usepackage{tikz}
\usepackage[title,titletoc]{appendix}

\DeclareMathOperator*{\argmin}{\arg\min}

\theoremstyle{plain}

\newtheorem{corollary}{Corollary}
\newtheorem{lemma}{Lemma}
\newtheorem{proposition}{Proposition}
\theoremstyle{definition}
\newtheorem{assumption}{Assumption}
\newtheorem{definition}{Definition}
\theoremstyle{remark}
\newtheorem{remark}{Remark}

\title{A Memetic Procedure for \\ Global Multi-Objective Optimization \thanks{Submitted to \textit{Mathematical Programming Computation (MPC).}}}

\author{ \href{https://orcid.org/0000-0002-2488-5486}{\includegraphics[scale=0.06]{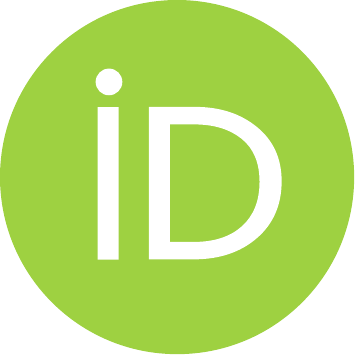}\hspace{1mm}Matteo Lapucci}\\
	Global Optimization Laboratory (GOL) \\
	Department of Information Engineering \\
	University of Florence \\
	Via di Santa Marta, 3, 50139, Florence, Italy \\
	\texttt{matteo.lapucci@unifi.it} \\
	\And
	\href{https://orcid.org/0000-0002-1394-0937}{\includegraphics[scale=0.06]{orcid.pdf}\hspace{1mm}Pierluigi Mansueto} \\
	Global Optimization Laboratory (GOL) \\
	Department of Information Engineering \\
	University of Florence \\
	Via di Santa Marta, 3, 50139, Florence, Italy \\
	\texttt{pierluigi.mansueto@unifi.it} \\
	\And
	\href{https://orcid.org/0000-0003-1160-7572}{\includegraphics[scale=0.06]{orcid.pdf}\hspace{1mm}Fabio Schoen} \\
	Global Optimization Laboratory (GOL) \\
	Department of Information Engineering \\
	University of Florence \\
	Via di Santa Marta, 3, 50139, Florence, Italy \\
	\texttt{fabio.schoen@unifi.it} \\
}

\renewcommand{\headeright}{Matteo Lapucci et al.}
\renewcommand{\undertitle}{}
\renewcommand{\shorttitle}{A Memetic Procedure for Global Multi-Objective Optimization}

\hypersetup{
pdftitle={A Memetic Procedure for Global Multi-Objective Optimization},
pdfsubject={},
pdfauthor={Matteo Lapucci, Pierluigi Mansueto, Fabio Schoen},
pdfkeywords={Multi-objective optimization, Memetic algorithm, NSGA-II, Descent method, Pareto front approximation},
}

\begin{document}
\maketitle

\begin{abstract}
	In this paper we consider multi-objective optimization problems over a box. The problem is very relevant and several computational approaches have been proposed in the literature. They  broadly fall into two main classes: evolutionary methods, which are usually very good at exploring the feasible region and retrieving good solutions even in the nonconvex case, and descent methods, which excel in efficiently approximating good quality solutions. In this paper, first we confirm, through numerical experiments, the advantages and disadvantages of these approaches. Then we propose a new method which combines the good features of both. The resulting algorithm, which we call Non-dominated Sorting Memetic Algorithm (\texttt{NSMA}), besides enjoying interesting theoretical properties, excels in all of the numerical tests we performed on several, widely employed, test functions. 
\end{abstract}

\keywords{Multi-objective optimization \and Memetic algorithm \and NSGA-II \and Descent method \and Pareto front approximation}

\MSCs{90C29 \and 90C30 \and 68W20}

\section{Introduction}
\label{sec::introduction}

Multi-objective optimization problems have a significant relevance in many applications of various fields, such as engineering \cite{campana2018multi,pellegrini2014application,sun2016multi}, management \cite{gravel1992multicriterion,white1998epsilon}, statistics \cite{carrizosa1998dominating}, space exploration \cite{palermo2003system,tavana2004subjective}, etc.. It is thus not surprising that many research streams have flourished around this topic for the last 25 years.

Multi-objective optimization (MOO) exhibits two major complexities, that coupled together make problems particularly difficult to handle. The first complexity element is the general absence of a solution minimizing all the objective functions simultaneously; as a consequence, the definitions of optimality (global, local and stationarity), based on Pareto's theory, are not trivial and make optimization processes not obvious, both in terms of aims and tools. The second one, on the other hand, traces one classical issue typical of scalar optimization: in absence of convexity assumptions, there is no equivalence between local and global (Pareto) optimality.  
	
The combination of the two aforementioned features of MOO problems makes evolutionary-type algorithms (EAs) particularly well suited to be used and, indeed, they have been the most widely studied class of algorithms for decades in this context \cite{laumanns2002combining,mostaghim2007multi}. The \texttt{NSGA-II} algorithm \cite{deb2002fast} is arguably the most popular among these methods; basically, it is a population-based procedure exploiting a cheaply computable score to efficiently rank solutions w.r.t.\ the objectives, and performing the classical genetic crossover, mutation and selection operations to create the new generation of solutions.
Basically, \texttt{NSGA-II} represents the de facto standard, at least popularity-wise, for unconstrained and bound-constrained MOO.
	
In fact, alongside the EA stream, two different classes of approaches have been studied for MOO. The first one concerns scalarization approaches \cite{drummond2008choice,eichfelder2009adaptive,pascoletti1984scalarizing}. However, solving MOO problems by scalarization has some drawbacks: firstly, unfortunate choices of weights may lead to unbounded scalar problems, even under strong regularity assumptions \cite[Section 7]{fliege2009newton}; moreover, scalarization is designed to produce a single solution and, in order to generate an approximation of the whole Pareto front, the problem has to be solved with different choices of weights; unfortunately, it is not known a priori how weights should be selected to obtain a wide and uniform Pareto front.
	
The other family of methods is that of MO descent methods (either first-order, second-order and derivative-free) \cite{carrizo2016trust,cocchi2020augmented,drummond2004projected,fliege2009newton,fliege2000steepest,fukuda2019barrier,gonccalves2020globally}. These methods mimic classical iterative scalar optimization algorithms. 

Originally, these methods were designed to generate a single Pareto-stationary solution, but in recent years specific strategies have been proposed to handle lists of points and to generate approximations of the entire Pareto front \cite{COCCHI2021100008, cocchi2020convergence,cocchi2018implicit,custodio2011direct,fliege2016method,liuzzi2016derivative}.  
Numerical results show that these methods, when used on problems with reasonable regularity assumptions, are effective and much more efficient than evolutionary methods, especially as the problem size grows \cite{COCCHI2021100008, cocchi2020convergence,custodio2011direct}. 

On the other hand, descent algorithms have convergence properties that are theoretically relevant but, in practice, guarantee only Pareto stationarity of the retrieved solutions. This is a significant limitation with  highly non-convex problems, similarly as for gradient-based algorithms in scalar optimization that may converge to stationary points which are not even local optima. 

In scalar global optimization, particularly successful strategies are memetic ones. Memetic algorithms combine population-based techniques (either heuristic and/or genetic ones) and local search steps \cite{CABASSIDE,GRIBELHG,grosso2007population,LOCATELLIDE,locatelli2013global,MANSUETO2021107849}.
In the case of MOO, this idea has only superficially been considered. In fact, each class of MOO algorithms has practical drawbacks. EAs have no theoretical convergence property \cite{fliege2009newton} and are usually expensive \cite{hu2003hybridization,liu2017improved,Sindhya2013495}. On the other hand, descent algorithms often produce suboptimal solutions when starting from non carefully chosen points and are thus not suitable for highly non-convex problems.
For these reasons, some MO memetic approaches have been proposed in the literature. However, we can find approaches that are mostly application-specific  \cite{wang2011multi} or that employ heuristic \cite{Kim2014290,Lara2010,liu2007multiobjective,mandal_memetic_2015}, meta-heuristic \cite{Bandyopadhyay2008,villalobos2018memetic}, stochastic \cite{drugan2012,Filatovas2017859} or scalarization-based \cite{hu2003hybridization,Sindhya2013495,Tiwari2009} local search steps. Even the few proposed strategies employing gradient information for the local search steps do not exploit the concept of common descent directions. Rather, convex combinations of gradients are generated and exploited in various ways \cite{Bhuvana2015,brown2005directed,Lara2010,liu2017improved,shukla2007gradient}.

In this paper, we show by computational experiments the benefits and limitations of evolutionary and MO descent algorithms in different settings (high/low dimensionality, convex/non-convex objectives). Then, we propose a memetic algorithm for bound-constrained MOO problems, combining an evolutionary approach (namely, the popular \texttt{NSGA-II}) with MO descent methods, similarly to what is done in the scalar case in \cite{grosso2007population}. We finally show that the proposed method, that inherits the good features of both EA and MO descent families, outperforms the state-of-the-art MOO solvers in any setting.
	
The rest of the manuscript is organized as follows. In Section 2, we recall the main concepts of both the descent methods and the \texttt{NSGA-II} algorithm. In Section 3, we provide a description of our memetic algorithm along with a theoretical analysis. In Section 4, we first compare the two families of algorithms in some specific problems, in order to show the benefits and the shortcomings of both. Then, we show the results of computational experiments highlighting the good performance of our approach w.r.t.\ main state-of-the-art methods. Finally, in Section 5 we provide some concluding remarks.

\section{Preliminaries}
\label{sec::preliminaries}
In this work, we consider multi-objective optimization problems of the form
\begin{equation}
	\label{eq::mo-prob}
	\begin{aligned}
	\min_{x\in\mathbb{R}^n}\;&F(x)=(f_1(x),\ldots,f_m(x))^T\\\text{s.t. }&x \in [l, u],
	\end{aligned}
\end{equation}
where $F:\mathbb{R}^n\to\mathbb{R}^m$ is a continuously differentiable function and $l,u\in\mathbb{R}^n$ with $l_i\le u_i \hspace{0.15cm} \forall i \in \{1,\ldots, n\}$. The values of $l$ and $u$
may possibly be infinite. Given the boundary constraints, we denote the feasible set, which is closed, convex and non empty, by $\Omega = \{x \in \mathbb{R}^n \mid x \in [l, u]\}$. We denote by $J_F$ the Jacobian matrix associated with $F$.

In order to introduce some preliminary concepts of multi-objective optimization, we define a partial ordering of the points in $\mathbb{R}^m$. Considering two points $u, v \in \mathbb{R}^m$, we define

\begin{equation*}
	\begin{array}{lcr}
		u < v \iff u_i < v_i \hspace{0.9cm} \forall i = 1,\ldots, m, \\
		u \le v \iff u_i \le v_i \hspace{0.9cm} \forall i= 1,\ldots, m.
	\end{array}
\end{equation*}
If $u \le v$ and $u \not = v$, we can say that $u$ dominates $v$ and we use the following notation: $u \lneqq v$. Finally, we say that the point $x \in \mathbb{R}^n$ dominates $y \in \mathbb{R}^n$ w.r.t.\ $F$ if  $F(x) \lneqq F(y)$.

In multi-objective optimization problems, we ideally would like to obtain a point which simultaneously minimizes all the objectives $f_1,\ldots, f_m$. However, such a solution is unlikely to exist. For this reason, the Pareto optimality concepts have been introduced.

\begin{definition}
	A point $\bar{x}\in\Omega$ is \textit{Pareto optimal} for Problem \eqref{eq::mo-prob} if a point $y\in\Omega$ such that $F(y)\lneqq F(\bar{x})$ does not exist.
	If there exists a neighborhood $\mathcal{N}(\bar{x}) \supset \{\bar{x}\}$ such that the previous property is satisfied in $\Omega\cap \mathcal{N}(\bar{x})$, then $\bar{x}$ is \textit{locally Pareto optimal}. 
\end{definition}
In practice, it is difficult to attain solutions characterized by the Pareto optimality property. A slightly weaker property is weak Pareto optimality.

\begin{definition}
	A point $\bar{x}\in\Omega$ is \textit{weakly Pareto optimal} for Problem \eqref{eq::mo-prob} if a point $y\in\Omega$ such that $F(y)< F(\bar{x})$ does not exist. If there exists a neighborhood $\mathcal{N}(\bar{x}) \supset \{\bar{x}\}$ such that the previous property is satisfied in $\Omega\cap \mathcal{N}(\bar{x})$, then $\bar{x}$ is \textit{locally weakly Pareto optimal}.
\end{definition}
We refer to  the set of all Pareto optimal solutions of the problem as the \textit{Pareto set}, while by \textit{Pareto front} we refer to the image of the Pareto set through $F$.

We can now introduce the concept of Pareto stationarity.
\begin{definition}
	\label{def::par-stat}
	A point $\bar{x}\in\Omega$ is \textit{Pareto-stationary} for Problem \eqref{eq::mo-prob} if we have that 
	\begin{equation*}
		\max_{j=1,\ldots,m}\nabla f_j(\bar{x})^Td\ge 0,
	\end{equation*}
	for all feasible directions $d\in\mathcal{D}(\bar{x})=\{v\in\mathbb{R}^n\mid \exists\bar{t}>0:\bar{x}+tv\in \Omega\; \forall\, t\in[0,\bar{t}\,]\}$.
\end{definition}
Under differentiability assumptions, Pareto stationarity is a necessary condition for all types of Pareto optimality. Further assuming the convexity of the objectives in Problem \eqref{eq::mo-prob}, the condition is also sufficient for Pareto optimality. The property can be compactly re-written as 
\begin{equation*}
	\min_{d \in \mathcal{D}(\bar{x})} \max_{j=1,\ldots,m} \nabla f_j(\bar{x})^Td = 0.
\end{equation*}

Finally, we introduce a relaxation of Pareto stationarity, recalling the $\varepsilon$-Pareto-stationarity concept introduced in \cite{cocchi2020augmented}.

\begin{definition}
	Let $\varepsilon \ge 0$. A point $\bar{x} \in \Omega$ is $\varepsilon$-Pareto-stationary for Problem \eqref{eq::mo-prob} if 
	\begin{equation*}
		\min_{\substack{d \in \mathcal{D}(\bar{x})\\\|d\|\le 1}} \max_{j=1,\ldots,m} \nabla f_j(\bar{x})^Td \ge -\varepsilon.
	\end{equation*}
\end{definition}

In the following, we briefly review evolutionary and descent algorithms for MOO, with particular emphasis on the \texttt{NSGA-II} algorithm and steepest/projected gradient descent methods respectively.

\subsection{Multi-Objective Descent Methods}
\label{subsec::descent-methods}


First of all, let us address the following unconstrained optimization problem.
\begin{equation}
	\label{eq::unc-mo-prob}
	\begin{aligned}
		\min_{x\in\mathbb{R}^n}\;& F(x)=(f_1(x),\ldots,f_m(x))^T.
	\end{aligned}
\end{equation}

If a point $\bar{x} \in \mathbb{R}^n$ is not Pareto-stationary, then there exists a descent direction w.r.t.\ all the objectives. 
Therefore, following \cite[Section 3.1]{fliege2000steepest}, we can introduce the \textit{steepest common descent direction} as the solution of problem
\begin{equation}
	\label{eq::ste-com-desc}
	\min_{\substack{d\in\mathbb{R}^n \\ \left\| d \right\| \le 1}}\max_{j=1,\ldots,m} \nabla f_j(\bar{x})^Td,
\end{equation}
which, if $\ell_\infty$ norm is employed, can be re-written as an LP one:

\begin{equation}
	\begin{aligned}
		\min_{\beta\in\mathbb{R},\,d \in \mathbb{R}^n}\;& \beta\\
		\text{s.t. }&-1\le d_i\le 1\quad \hspace{0.15cm} \forall\,i=1,\ldots,n,\\
		&\nabla f_j(\bar{x})^Td\le \beta\quad \forall\,j=1,\ldots,m.
	\end{aligned}
\end{equation}
A slightly different problem formulation is the $\ell_2$-regularized one, which is again proposed in \cite{fliege2000steepest}. However, we preferred to use Formulation \eqref{eq::ste-com-desc} because of the simplicity of the LP problem.
We define the continuous function $\theta: \mathbb{R}^n \rightarrow \mathbb{R}$ such that $\theta(\bar{x})$ indicates the optimal value of Problem \eqref{eq::ste-com-desc} at $\bar{x}$. If $\bar{x}$ is Pareto-stationary, $\theta(\bar{x}) = 0$, otherwise $\theta(\bar{x}) < 0$. We also denote by $v(\bar{x}) \subseteq \mathbb{R}^n$ the set of optimal solutions to Problem \eqref{eq::ste-com-desc}. Indeed, the solution may not be unique, although this fact is not a real technical issue.

Based on the concept of steepest common descent direction, the standard \textit{Multi-Objective Steepest Descent} (\texttt{MOSD}) algorithm was proposed in \cite{fliege2000steepest}. In \texttt{MOSD}, a back-tracking \textit{Armijo-type Line Search} (\texttt{ALS}) is used. The idea of this latter one is to reduce the step size until we get a sufficient decrease for all the objective functions. We report \texttt{ALS} in Algorithm \ref{alg::ALS}.

We now recall the main theoretical results of the two algorithms, starting from the finite termination property of the line search.

\begin{lemma}{\cite[Lemma 4]{fliege2000steepest}}
	If $F$ is continuously differentiable and $J_F(x)d<0$ (i.e.\, $\theta(x)<0$), then there exists some $\varepsilon > 0$, which may depend on $x$, $d$ and $\beta$, such that 
	\begin{equation*}
		F(x+td) < F(x) + \beta t J_F(x)d
	\end{equation*}
	for all $t\in(0,\varepsilon]$.
\end{lemma}

\begin{algorithm}[h]
	\caption{Armijo-type Line Search} \label{alg::ALS}
	Input: $F:\mathbb{R}^n\to \mathbb{R}^m$, $x_k\in\mathbb{R}^n$, $d_k \in v(x_k)$,  $\alpha_0>0$, $\delta\in(0,1)$, $\beta\in(0,1)$. \\
	$\alpha=\alpha_0$\\
	\While{$F(x_k+\alpha d_k)\nleq F(x_k) + \beta \alpha J_F(x_k)d_k$}{
		$\alpha=\delta\alpha$
	}
	\Return $\alpha$
\end{algorithm}
\noindent Regarding the \texttt{MOSD} procedure, the following convergence property holds.
\begin{lemma}{\cite[Theorem 1, Section 9.1]{fliege2000steepest}}
	\label{lem::conv-MOSD}
	Every accumulation point of the sequence $\{x_k\}$ produced by the \texttt{MOSD} algorithm is a Pareto-stationary point. If the function $F$ has bounded level sets, in the sense that $\{x\in\mathbb{R}^n\mid F(x)\le F(x_0)\}$ is bounded, then the sequence $\{x_k\}$ stays bounded and has at least one accumulation point.
\end{lemma}

Through the years, the \texttt{MOSD} procedure was extended to handle a sequence of sets of non-dominated points, rather than a sequence of points, aiming to approximate the Pareto front of the optimization problems. Indeed, in the multi-objective context, an approximation of the Pareto front could be much more useful than a single solution: the user is free to choose, a posteriori, the solution providing the most appropriate trade-off among many. An algorithm representing an extension of the \texttt{MOSD} procedure is the one introduced in \cite{cocchi2020convergence}, which is called \textit{Front Steepest Descent Algorithm} (\texttt{FSDA}). 

In the next lemma, we report the convergence property of the algorithm, where the authors use the concept of \textit{linked sequence} introduced in \cite{liuzzi2016derivative}.

\begin{definition}
	\label{def::link-seq}
	A sequence $\{x_k\}$ is a linked sequence if, for all $k$, $x_k \in X_k$ and $x_k$ is generated at iteration $k - 1$ starting the search procedure from $x_{k - 1}$.
\end{definition}

\begin{lemma}{\cite[Proposition 5]{cocchi2020convergence}}
	Let $\{X_k\}$ be the sequence of sets of non-dominated points produced by \texttt{FSDA}. Let us assume that there exists a point $x_0 \in X_0$ such that:
	\begin{itemize}
		\item $x_0$ is not Pareto-stationary;
		\item the set $\mathcal{L}(x_0) = \bigcup_{j=1}^{m}\{x \in \mathbb{R}^n: f_j(x) \le f_j(x_0)\}$ is compact.
	\end{itemize}
	Let $\{x_k\}$ be a linked sequence, then it admits limit points and every limit point is Pareto-stationary for Problem \eqref{eq::unc-mo-prob}.
\end{lemma}

In \cite{cocchi2020convergence}, two concepts are introduced. The first one is the \textit{steepest partial descent direction} at $\bar{x}$ w.r.t.\ a subset of indices of objectives $I \subseteq \{1,\ldots, m\}$. This type of direction at a point $\bar{x}$ can be found solving this optimization problem:

\begin{equation}
	\label{eq::ste-par-desc}
	\min_{\substack{d\in\mathbb{R}^n \\ \left\| d \right\| \le 1}} \max_{j \in I} \nabla f_j(\bar{x})^Td.
\end{equation}
As for the steepest common descent direction, we define the continuous function $\theta^I: \mathbb{R}^n \rightarrow \mathbb{R}$, where $\theta^I(\bar{x})$ indicates the optimal value of Problem \eqref{eq::ste-par-desc} at $\bar{x}$. Accordingly, we denote by $v^I(\bar{x}) \subseteq \mathbb{R}^n$ the set of optimal solutions to the Problem \eqref{eq::ste-par-desc}.
If appropriately used, the steepest partial descent direction may be useful to spread the search in the objectives space and to reach the extreme regions of the Pareto front. 

For our purposes, given a subset $I$, we also define $F_I(\bar{x})$ as the $|I|$-dimensional vector with components $f_j(\bar{x})$, with $j \in I$. In addition, given a set of points $\bar{X}$, we introduce the set $\bar{X}^I \subseteq \bar{X}$ as the set of points that are mutually non-dominated w.r.t.\ $F_I$, i.e.\
\begin{equation*}
	\bar{X}^I = \{x \in \bar{X} \mid \nexists y \in \bar{X} \text{ s.t. } F_I(y) \lneqq F_I(x)\}.
\end{equation*}
 
The second concept introduced in \cite{cocchi2020convergence} is a weaker front-based variant of \texttt{ALS}: we call it the \textit{Front Armijo-Type Line Search} (\texttt{FALS}). We report it in Algorithm \ref{alg::FALS}.  In \texttt{FALS}, the step size is reduced until a sufficient decrease is reached w.r.t.\ all the points in $\bar{X}^I$ for at least one of the objective functions $f_j$, with $j \in I$. \texttt{FALS} can be considered as a weak extension of \texttt{ALS} to the multi-objective case. As it is not required to obtain a sufficient decrease for all the objective functions, employing \texttt{FALS} leads to two consequences: less required computational time and bigger values for the step size. These features again may be very useful to obtain good and spread Pareto front approximations in a short time.
\begin{algorithm}[h]
	\caption{Front Armijo-type Line Search} \label{alg::FALS}
	Input: $F:\mathbb{R}^n\to \mathbb{R}^m$, $I \subseteq \{1,\ldots, m\}$, $X_k^I$ set of mutually non-dominated points w.r.t.\ $F_I$, $x_c\in X_k^I$, $d_c^I \in v^I(x_c)$, $\theta^I(x_c) \in \mathbb{R}$, $\alpha_0>0$, $\delta\in(0,1)$, $\beta\in(0,1)$. \\
	$\alpha=\alpha_0$\\
	\While{$\exists y \in X_k^I \text{ s.t. } F_I(y) + \mathbf{1}\beta\alpha\theta^I(x_c) < F_I(x_c + \alpha d_c^I)$}{
		$\alpha=\delta\alpha$
	}
	\Return $\alpha$
\end{algorithm}

\texttt{FALS} has a finite termination property, which we recall in the following lemma.

\begin{lemma}{\cite[Proposition 4]{cocchi2020convergence}}
	Let $I \subseteq \{1,\ldots, m\}$, $x_c \in X_k^I$ be such that $\theta^I(x_c) < 0$, i.e.\ there exists a direction $d_c^I \in v^I(x_c)$ such that
	\begin{equation*}
		\nabla f_j(x_c)^Td^I_c < 0
	\end{equation*}
	$\forall j \in I$. Then $\exists\bar{\alpha} > 0$, sufficiently small, such that
	\begin{equation*}
		F_I(y) + \mathbf{1}\beta\bar{\alpha}\theta^I(x_c) \not < F_I(x_c + \bar{\alpha} d_c^I), \hspace{0.3cm} \forall y \in X_k^I,
	\end{equation*}
	i.e., the while loop of \texttt{FALS} terminates in a finite number $\bar{h}$ of iterations, returning a value $\bar{\alpha} = \delta^{\bar{h}}\alpha_0$. Furthermore, the produced point $x_c + \bar{\alpha}d_c^I$ is not dominated with respect to the set $X_k$.
\end{lemma}

Since we consider bound constraints in Problem \eqref{eq::mo-prob}, we need to recall the (single-point) \textit{Multi-Objective Projected Gradient} (\texttt{MOPG}) method. This latter algorithm was firstly introduced in \cite{drummond2004projected} and then developed and analyzed in \cite{fukuda2011convergence, fukuda2013inexact}. In addition, the \texttt{MOPG} main results were summarized in \cite{Fukuda2014}. The method is an extension of the \texttt{MOSD} procedure dealing with constrained problems. In particular, it deals with optimization problems characterized by a feasible closed and convex set $\Omega$. We report the \texttt{MOPG} procedure in Algorithm \ref{alg::MOPG}. 

\begin{algorithm}[h]
	\caption{Multi-Objective Projected Gradient} \label{alg::MOPG}
	Input: $F:\mathbb{R}^n\to\mathbb{R}^m$, $\Omega$ feasible closed and convex set,  $x_0\in\Omega$.\\
	$k=0$\\
	\While{$x_k$ is not Pareto-stationary}{
		Compute 
		\begin{equation*}
			z_{\Omega k} \in \argmin_{\substack{z \in \Omega \\ \| z - x_k \| \le 1}} \max_{j=1,\ldots,m} \nabla f_j(x_k)^T(z - x_k)
		\end{equation*}\\
		$d_{\Omega k} = z_{\Omega k} - x_k$\\
		$\alpha_k=\texttt{ALS}(F(\cdot),x_k,d_{\Omega k})$\\
		$x_{k+1} = x_k+\alpha_k d_{\Omega k}$\\
		$k= k+1$
	}
	\Return $x_k$
\end{algorithm}

The first difference w.r.t.\ the \texttt{MOSD} procedure is the way the direction is retrieved, since now the problem constraints have to be considered. The steepest descent direction is now defined as the solution of
\begin{equation}
	\label{eq::proj-desc}
	\min_{\substack{z \in \Omega \\ \| z - \bar{x} \| \le 1}} \max_{j=1,\ldots,m} \nabla f_j(\bar{x})^T(z - \bar{x}).
\end{equation}
As in the unconstrained case, we define the continuous function $\theta_\Omega: \mathbb{R}^n \rightarrow \mathbb{R}$ such that $\theta_\Omega(\bar{x})$ indicates the optimal value of Problem \eqref{eq::proj-desc} at $\bar{x}$. We also denote by $\mathcal{Z}_\Omega(\bar{x}) \subseteq \Omega$ the set of optimal solutions to Problem \eqref{eq::proj-desc}, and by $v_\Omega(\bar{x}) = \{z - \bar{x} \mid z \in \mathcal{Z}_\Omega(\bar{x})\} \subseteq \Omega$ the set of optimal directions. We denote these latter ones as \textit{constrained steepest common descent directions}; if a subset $I \subseteq \{1,\ldots, m\}$ is considered, they are referred to \textit{constrained steepest partial descent directions}.
If $\theta_\Omega(\bar{x}) < 0$, the point is not Pareto-stationary and we proceed to find a step size through \texttt{ALS}. As opposed to the unconstrained case, where $\alpha_0$ can be any positive real number, in the \texttt{MOPG} procedure $\alpha_0 = 1$. Since the set $\Omega$ is convex, $d$ is a feasible direction by construction and $\alpha_0 = 1$, every produced point will be feasible. 

We report here two theoretical results of the \texttt{MOPG} method.

\begin{lemma}
	\cite[Lemma 4.3]{Fukuda2014} Let $\{x_k\} \subset \mathbb{R}^n$ be a sequence generated by \texttt{MOPG}. Then, we have $\{x_k\} \subset \Omega$.
\end{lemma}

\begin{lemma}
	\cite[Theorem 4.4]{Fukuda2014} Every accumulation point, if any, of a sequence $\{x_k\}$ generated by \texttt{MOPG} is a feasible Pareto-stationary point.
\end{lemma}

In order to deal with box-constrained optimization problems, we adapted the \texttt{FSDA} algorithm \cite{cocchi2020convergence}. We call the adaptation \textit{Front Projected Gradient Algorithm} (\texttt{FPGA}) and the differences w.r.t.\ the original algorithm are the following:
\begin{itemize}
	\item the initial set $X_0$ is composed by \textit{feasible} non-dominated points w.r.t.\ $F$;
	\item the direction is found solving Problem \eqref{eq::proj-desc} instead of Problem \eqref{eq::ste-com-desc};
	\item we employ the \textit{Bound-constrained Front Armijo-Type Line Search} (\texttt{B-FALS}), which is a modified version of \texttt{FALS} that we introduced to take into account bound constraints.
\end{itemize}

\texttt{B-FALS}, which we report in Algorithm \ref{alg::B-FALS}, is similar to \texttt{FALS}: the only added requirement is that the step size must lead to a point that is feasible. 

\begin{algorithm}[h]
	\caption{Bound-constrained Front Armijo-type Line Search} \label{alg::B-FALS}
	Input: $F:\mathbb{R}^n\to \mathbb{R}^m$, $\Omega$ feasible convex set, $I \subseteq \{1,\ldots, m\}$, $X_k^I$ set of feasible mutually non-dominated points w.r.t.\ $F_I$, $x_c\in X_k^I$, $d_{\Omega c}^I \in v_\Omega^I(x_c)$, $\theta_\Omega^I(x_c) \in \mathbb{R}$, $\alpha_0>0$, $\delta\in(0,1)$, $\beta\in(0,1)$. \\
	$\alpha=\alpha_0$\\
	\While{$x_c + \alpha d_{\Omega c}^I \not \in \Omega$ \textbf{or} $\exists y \in X_k^I \text{ s.t. } F_I(y) + \mathbf{1}\beta\alpha\theta_\Omega^I(x_c) < F_I(x_c + \alpha d_{\Omega c}^I)$}{
		$\alpha=\delta\alpha$
	}
	\Return $\alpha$
\end{algorithm}

Through the following proposition, we show that \texttt{B-FALS} terminates in a finite number of iterations.

\begin{proposition}
	\label{prop::fin-ter-B-FALS}
	Let $I \subseteq \{1,\ldots, m\}$, $x_c \in X_k^I$ be such that $\theta_\Omega^I(x_c) < 0$, i.e.\ there exists a direction $d_{\Omega c}^I  \in v_\Omega^I(x_c)$ such that
	\begin{equation*}
		\nabla f_j(x_c)^Td_{\Omega c}^I < 0
	\end{equation*}
	$\forall j \in I$. Then $\exists \bar{\alpha} > 0$, sufficiently small, such that
	\begin{equation*}
		x_c + \bar{\alpha} d_{\Omega c}^I \in \Omega
	\end{equation*}
	and
	\begin{equation*}
		F_I(y) + \mathbf{1}\beta\bar{\alpha}\theta_\Omega^I(x_c) \not< F_I(x_c + \bar{\alpha} d_{\Omega c}^I) \hspace{0.15cm} \forall y \in X_k^I,
	\end{equation*}
	i.e., the while loop of \texttt{B-FALS} terminates in a finite number $\bar{h}$ of iterations, returning a value $\bar{\alpha} = \delta^{\bar{h}}\alpha_0$. Furthermore, the produced point $x_c + \bar{\alpha}d_{\Omega c}^I$ is not dominated with respect to the set $X_k$.
\end{proposition}
\begin{proof}
	Assume by contradiction that the thesis is false. Then the algorithm produces an infinite sequence $\{\delta^h\alpha_0\}$ such that, for all $h$, either
	\begin{equation*}
		x_c + \delta^h\alpha_0 d_{\Omega c}^I \not \in \Omega
	\end{equation*}
	or a point $y_h \in X_k^I$ exists such that 
	\begin{equation}
		\label{eq::2-B-FALS-while-condition}
		F_I(y_h) + \mathbf{1}\beta\delta^h\alpha_0\theta_\Omega^I(x_c) < F_I(x_c + \delta^h\alpha_0 d_{\Omega c}^I).
	\end{equation}
	Since $\Omega$ is convex, $d_{\Omega c}^I$ is a feasible direction by construction and $\delta^h\alpha_0\to0$ as $h\to\infty$, for $h$ sufficiently large the point $x_c+\delta^h\alpha_0d_{\Omega c}^I$ is feasible and thus Condition \eqref{eq::2-B-FALS-while-condition} holds. Then, following the proof of Proposition 4 in \cite{cocchi2020convergence}, we can prove the thesis.
\end{proof}

We provide a full description of \texttt{FPGA} in Appendix \ref{app::FPGA}, along with feasibility and convergence properties. In this work, we consider the \texttt{FPGA} algorithm as a representative of gradient-based methods designed to produce Pareto front approximations for bound-constrained optimization problems. 

In \cite{cocchi2020convergence}, a variant of the \texttt{FSDA} algorithm exploiting an Armijo-type extrapolation technique is also introduced. The authors claim that this variant outperforms the original algorithm. For this reason, in the initial stage of our work, we decided to also adapt this variant for box-constrained optimization problems. However, some preliminary computational experiments led us to the somewhat surprising conclusion that vanilla \texttt{FPGA} is better than using extrapolation. A possible justification of this result may lie in the presence of constraints. Anyhow, we decided for the sake of brevity to only consider \texttt{FPGA} in the remainder of the paper.

\subsection{\texttt{NSGA-II}}
\label{subsec::NSGA-II}

\texttt{NSGA-II} is a non-dominated sorting-based multi-objective evolutionary algorithm that was proposed in \cite{deb2002fast}. In particular,
\texttt{NSGA-II} is a genetic algorithm that creates a mating pool by combining the parent and offspring populations and selecting the best $N$ solutions. In this section, we review the main characteristics of \texttt{NSGA-II}. For a deeper understanding of the algorithm mechanisms, the reader is referred to the original work \cite{deb2002fast}.

We report the main steps of \texttt{NSGA-II} in Algorithm \ref{alg::NSGA-II}.

\begin{algorithm}[h]
	\caption{NSGA-II} \label{alg::NSGA-II}
	Input: $F:\mathbb{R}^n \rightarrow \mathbb{R}^m$, $\Omega = \{x \in \mathbb{R}^n \mid x \in [l, u]\}$ feasible set, $X_0\subset\Omega$, $N$ population size.\\
	$k=0$\\
	$\hat{X}_0 = X_0$ \\
	$\hat{R}_0, \hat{C}_0$ = \texttt{getMetrics}($\hat{X}_0$) \\
	$X_0, R_0, C_0$ = \texttt{getSurvivors}($\hat{X}_0, \hat{R}_0, \hat{C}_0, N$) \\
	\While{a stopping criterion is not satisfied}{
		$P_k$ = \texttt{getParents}($X_k, R_k, C_k$)\\
		$O_k$ = \texttt{crossover}($P_k, l, u$)\\
		$\tilde{O}_k$ = \texttt{mutation}($O_k, l, u$)\\
		$\hat{X}_{k + 1} = X_k \cup \tilde{O}_k$\\
		$\hat{R}_{k + 1}, \hat{C}_{k + 1}$ = \texttt{getMetrics}($\hat{X}_{k + 1}$) \\
		$X_{k + 1}, R_{k + 1}, C_{k + 1}$ = \texttt{getSurvivors}($\hat{X}_{k + 1}, \hat{R}_{k + 1}, \hat{C}_{k + 1}, N$) \\
		$k = k + 1$
	}
	\Return $X_k$
\end{algorithm}

\texttt{NSGA-II} deals with a fixed size population ($N$ solutions) and takes as input an initial population $X_0$. For the sake of clarity, from now on we consider $X_0$ as a set composed by $N$ feasible solutions. However, we want to remark two facts.

\begin{itemize}
	\item Starting with a population $X_0$ composed by $N$ points is not necessary: if the population is smaller/bigger, after the first iteration it is increased/reduced in order to get exactly $N$ solutions in it.
	\item \texttt{NSGA-II} can also manage unfeasible points. However, since in our work we address bound constrained problems, and the genetic operators ensure that after the first iteration no point in the population violates the bound constraints, we assume that $X_0$ is only composed by feasible points. 
\end{itemize}
The core idea of the algorithm is that during an iteration:

\begin{itemize}
	\item the parents are chosen among the current solutions;
	\item $N$ offsprings are created from the parents through the \texttt{crossover} operator;
	\item the offsprings are mutated using the \texttt{mutation} function;
	\item a new population of $2N$ solutions is created merging the current population with the offsprings;
	\item only the best $N$ points (survivors) are selected and maintained.
\end{itemize}
The \texttt{crossover} and \texttt{mutation} operators have a crucial role in the \texttt{NSGA-II} mechanisms. The \texttt{crossover} operator aim is the creation of offsprings that inherit (hopefully the best) features of the parents. The \texttt{mutation} operator introduces some random changes in the offsprings. This latter one could be useful when we want to spread our search in the objectives space as much as possible. For a more detailed and technical explanation about these two operators, we again refer the reader to \cite{deb2002fast}. We want to remark here that the \texttt{NSGA-II} mechanisms ensure that there are no duplicates among the offsprings and any offspring is not a duplicate of any point in the current population.
At the end of the algorithm execution, the current population $X_k$ is returned.

In the next subsections, we provide other details of the algorithm that are useful for our purposes.

\subsubsection{Metrics}
\label{subsubsec::getmetrics}

In this section, we explain the metrics used in the \texttt{NSGA-II} mechanisms (computed in the \texttt{getMetrics} function). In particular, these scores are used to select the parents and the survivors.

The first one is the \textit{ranking}, which leads to a splitting of the population in different domination levels. Briefly, if a point has rank $0$, it means that it is not dominated by any point in $X_k$ w.r.t.\ $F$. If it has rank $1$, it is dominated by some of the points with rank $0$, but it is not w.r.t.\ any other point with rank equal to or greater than $1$. In order to obtain the ranking values, a fast sorting approach is employed, which is one of the strength elements of the \texttt{NSGA-II} algorithm.

The second considered metric is \textit{crowding distance}. It is useful to get an estimate of the density of solutions surrounding a particular point in the population. Having a high crowding distance indicates that the point is in a poorly populated area of the objectives space, and maintaining it in the population may likely lead to a spread Pareto front. Note that for each point this metric is calculated with respect to the solutions with the same ranking value.

We again refer the reader to the original paper \cite{deb2002fast} for the rigorous definition of the metrics.

\subsubsection{Parents selection}
\label{subsubsec::getparents}

In the \texttt{getParents} function, pairs of parents are randomly chosen among the solutions in $X_k$. Then, considering a pair, only one of the two points is selected by \textit{binary tournament}.
In this latter one, the solutions are compared in the following way.

\begin{itemize}
	\item The point with the lowest rank is preferred.
	\item If the ranking value is the same for both points, the one with the highest crowding distance is chosen.
	\item In the unlikely case in which the crowding distance values are equal too, a random choice is done.
\end{itemize}
The selected point will be used with a parent chosen from another pair in the \texttt{crossover} function in order to create offsprings.

This approach of comparing the solutions is also used in the \texttt{getSurvivors} function.

\subsubsection{Selection operation}
\label{subsubsec::getsurvivors}

After getting the offsprings through the \texttt{crossover} and \texttt{mutation} operators, the new population is composed by $2N$ solutions. The aim of the \texttt{getSurvivors} function is to select and maintain the best $N$ solutions. As in the \texttt{getParents} function, the selection is based on the ranking and the crowding distance. 
\begin{itemize}
	\item The set composed by the $2N$ points is initially sorted based on the ranking.
	\item The solutions with the same rank are sorted based on the crowding distance.
	\item The first $N$ points are chosen as the best ones.
\end{itemize}
\section{Non-dominated Sorting Memetic Algorithm}
\label{sec::NSMA}

In this section, we introduce a novel memetic algorithm for bound-constrained MOO problems, which we call \textit{Non-dominated Sorting Memetic Algorithm} (\texttt{NSMA}). We first show and describe the algorithmic scheme. Then, we formally introduce the \textit{Front Multi-Objective Projected Gradient} (\texttt{FMOPG}) algorithm, which is the descent method used within the \texttt{NSMA} and for which we also provide a rigorous theoretical analysis.

\subsection{Algorithmic scheme}
\label{subsec::algorithmic-scheme}

The scheme of \texttt{NSMA} is reported in Algorithm \ref{alg::NSMA}.
Basically, the structure of the proposed algorithm is similar to that of \texttt{NSGA-II}, from which we also inherit all the genetic operators. The main differences between the two methods are constituted by three new operations: 

\begin{itemize}
	\item \texttt{getSurrogateBounds} (Line \ref{line::getSurrogateBounds});
	\item \texttt{getCrowdingDistanceThreshold} (Line \ref{line::getCrowdingDistanceThreshold});
	\item \texttt{optimizePopulation} (Line \ref{line::optimizePopulation}).
\end{itemize}

\begin{algorithm}[h]
	\caption{Non-dominated Sorting Memetic Algorithm} \label{alg::NSMA}
	Input: $F:\mathbb{R}^n \rightarrow \mathbb{R}^m$, $\Omega = \{x \in \mathbb{R}^n | x \in [l, u]\}$ feasible set, $X_0\subset \Omega$, $N$ population size, $s_h \in \mathbb{R}^+_0$, $q \in [0, 1]$, $n_{opt} \in \mathbb{N}^+$, $\{\varepsilon_t\} \subset \mathbb{R}^+_0$ decreasing sequence.\\
	$k = 0$ \\
	$t = 0$ \\
	$\hat{X}_0 = X_0$ \\
	$\hat{R}_0, \hat{C}_0$ = \texttt{getMetrics}($\hat{X}_0$) \\
	$X_0, R_0, C_0$ = \texttt{getSurvivors}($\hat{X}_0, \hat{R}_0, \hat{C}_0, N$) \\
	\While{a stopping criterion is not satisfied}{
		$P_k$ = \texttt{getParents}($X_k, R_k, C_k$)\\
		$l_k^s, u_k^s$ = \texttt{getSurrogateBounds}($X_k, l, u, s_h$) \label{line::getSurrogateBounds}\\
		$O_k$ = \texttt{crossover}($P_k, l_k^s, u_k^s$)\\
		$\tilde{O}_k$ = \texttt{mutation}($O_k, l_k^s, u_k^s$)\\
		$\hat{X}_{k + 1} = X_k \cup \tilde{O}_k$\\
		$\hat{R}_{k + 1}, \hat{C}_{k + 1}$ = \texttt{getMetrics}($\hat{X}_{k + 1}$) \\
		\label{line::getCrowdingDistanceThreshold}$\bar{c}_{k + 1}$ =  \texttt{getCrowdingDistanceThreshold}($\hat{X}_{k + 1}, \hat{R}_{k + 1}, \hat{C}_{k + 1}, q$) \\
		$X_{k + 1}, R_{k + 1}, C_{k + 1}$ = \texttt{getSurvivors}($\hat{X}_{k + 1}, \hat{R}_{k + 1}, \hat{C}_{k + 1}, N$) \\
		\If{$k \textbf{ mod } n_{opt} = 0$}{
			$X_{k + 1}, R_{k + 1}, C_{k + 1}$ = \texttt{optimizePopulation}($F(\cdot), \Omega, X_{k + 1}, R_{k + 1},$$C_{k + 1}, \bar{c}_{k + 1}, \varepsilon_t, N$)\label{line::optimizePopulation}\\
			$t = t + 1$
		}
		$k = k + 1$
	}
	\Return $X_k$
\end{algorithm}

In the next subsections, we give a detailed description of these three new functions.

\subsubsection{Estimating Surrogate Bounds}
\label{subsubsec::getsurrogatebounds}

When the addressed problem is characterized by a particularly large feasible region ($l \ll u$), the \texttt{NSGA-II} algorithm turns out to be slow at obtaining a good approximation of the Pareto front. This issue occurs because of the \texttt{crossover} and, above all, of the \texttt{mutation} operator. Random mutations over a large search area lead from the very first iterations to a population which is overly disperse and far from optimality. In such a scenario, even the effectiveness of the \texttt{crossover} operator might be compromised: some parents might have extremely bad features. As a consequence, \texttt{NSGA-II} may exhibit a performance slowdown.

In \texttt{NSMA}, this issue is solved using surrogate bounds for the \texttt{crossover} and the \texttt{mutation} operators instead of the original ones. These bounds are obtained using the \texttt{getSurrogateBounds} function, which we report in Algorithm \ref{alg::getSurrogateBounds}. The surrogate bounds are computed using the current population and a shift value $s_h$. This latter parameter is  employed to progressively enlarge the region where the population can be distributed: a greater value leads to a bigger enlargement.

\begin{algorithm}[h]
	\caption{getSurrogateBounds} \label{alg::getSurrogateBounds}
	Input: $X_k\subset\Omega$, $l, u \in \mathbb{R}^n$ lower and upper bounds, $s_h \in \mathbb{R}^+_0$.\\
	\For{$i = 1,\ldots, n$}{
		$(l_{k}^s)_i = \max\left\{l_i, \displaystyle\min_{x \in X_k}\{x_i\} - s_h\right\}$\\
		$(u_{k}^s)_i = \min\left\{u_i, \displaystyle\max_{x \in X_k}\{x_i\} + s_h\right\}$
	}
	\Return $l_k^s, u_k^s$
\end{algorithm}

Ideally, the exploration starts considering only a small portion of the feasible area, which is defined by the initial population and the shift value $s_h$. In this way, the points cannot be moved by the \texttt{crossover} and the \texttt{mutation} operators too far away in the feasible set. At each following iteration, new surrogate bounds are computed to enlarge the search space. 

After a number of iterations, it may happen that the surrogate bounds cover a bigger region than the one defined by the original bounds. In such case, the search goes on over the entire feasible set.

\subsubsection{Identifying Exploration Candidates}
\label{subsubsec::getcrowdingdistancethreshold}
Similarly as in memetic approaches for scalar optimization, performing local searches starting from each point in a population usually turns out to be inefficient. In fact, a great computational effort is required to optimize many points that in the end do not lead to good solutions. 

%

In the case of \texttt{NSMA}, one may think of only performing local searches for the rank-$0$ points. However, this idea is inefficient too: during the last iterations, most, if not all, the points are likely to be associated with a ranking value equal to $0$. Furthermore, many of these points could be in a high density area of the Pareto front and, therefore, optimizing all of them could be a waste of computational time. 

The issue is solved by choosing to optimize the rank-$0$ points associated with an high crowding distance. As already remarked in Section \ref{subsubsec::getmetrics}, such points are in a poorly populated area of the objectives space. Therefore, optimizing them, we still contribute to obtain a better approximation of the Pareto front, since they are rank-$0$ points, and, at the same time, we have the possibility to populate a low density area, leading to a better spread Pareto front.

\begin{algorithm}[h]
	\caption{getCrowdingDistanceThreshold} \label{alg::getCrowdingDistanceThreshold}
	Input: $\hat{X}_{k + 1}\subset\Omega$, $\hat{R}_{k + 1}, \hat{C}_{k + 1} \in \mathbb{R}^{|\hat{X}_{k + 1}|}$ metrics vectors, $q \in [0, 1]$.\\
	$\bar{C}_{k + 1} = \{\hat{c}_p \in \hat{C}_{k + 1} | \hat{r}_p = 0 \land \hat{c}_p < +\infty\}$\\
	\If{$\bar{C}_{k + 1} \neq \emptyset$}{
		Let $\bar{c}_{k + 1}$ be the $q$--quantile of the set $\bar{C}_{k + 1}$
	}
	\Else{
		$\bar{c}_{k + 1} = +\infty$
	}
	\Return $\bar{c}_{k + 1}$
\end{algorithm}

Through the \texttt{getCrowdingDistanceThreshold} function, which we report in Algorithm \ref{alg::getCrowdingDistanceThreshold}, we retrieve the $q$--quantile of the crowding distances of the rank-$0$ points in $\hat{X}_{k + 1}$. We denote by $\bar{c}_{k + 1}$ this quantity: only the rank-$0$ points associated with a crowding distance greater than or equal to $\bar{c}_{k + 1}$ will be optimized through the \texttt{FMOPG} algorithm. Smaller values for the parameter $q$  lead to the optimization of a greater number of points.

As stated in \cite{deb2002fast}, some points will be associated to a crowding distance equal to $+\infty$. These points are considered the extreme solutions of the Pareto front w.r.t.\ a specific objective function. For this reason, they are always used as starting solutions for local searches, since they could lead to a wider Pareto front approximation. 

\subsubsection{Local searches by multi-objective descent}
\label{subsubsec::optimizepopulation}

In the \texttt{optimizePopulation} function, which we report in Algorithm \ref{alg::optimizePopulation}, the \texttt{FMOPG} method is employed to refine the population by performing local searches. This function is the core of our memetic approach: it allows to combine the typical features of descent methods with the genetic operators of \texttt{NSGA-II}.

\begin{algorithm}[h]
	\caption{optimizePopulation} \label{alg::optimizePopulation}
	Input: $F:\mathbb{R}^n \rightarrow \mathbb{R}^m$, $\Omega$ feasible closed and convex set, $X_{k + 1}\subset\Omega$, $R_{k + 1}, C_{k + 1} \in \mathbb{R}^{|X_{k + 1}|}$ metrics vectors, $\bar{c}_{k + 1}$ crowding distance threshold, $\varepsilon_t \in \mathbb{R}^+_0$, $N$ population size.\\
	$\hat{X}_{k + 1} = X_{k + 1}$ \\
	\For{$p = 1,\ldots, |X_{k + 1}|$}{
		\For{$I \in 2^{\{1,\ldots, m\}}$}{
			\If{$x_p$ is such that 
				\begin{itemize}
					\item $r_p = 0$, $c_p \ge \bar{c}_{k + 1}$
					\item $x_p \in \hat{X}_{k + 1}^I$
					\item $\theta_\Omega^I(x_p) < 0$
				\end{itemize}}{
				$\label{line::call-FMOPG}\tilde{X}_p$ = \texttt{FMOPG}($F(\cdot), \Omega, I, \hat{X}_{k + 1}^I, x_p, \varepsilon_t$)\\
				$\hat{X}_{k + 1} = \hat{X}_{k + 1} \cup \tilde{X}_p$
			}
		}
	}
	$\hat{R}_{k + 1}, \hat{C}_{k + 1}$ = \texttt{getMetrics}($\hat{X}_{k + 1}$) \\
	$X_{k + 1}, R_{k + 1}, C_{k + 1}$ = \texttt{getSurvivors}($\hat{X}_{k + 1}, \hat{R}_{k + 1}, \hat{C}_{k + 1}, N$) \\
	\Return $X_{k + 1}, R_{k + 1}, C_{k + 1}$
\end{algorithm}

In order to be optimized through \texttt{FMOPG} w.r.t.\ a subset of indices of objectives $I \subseteq \{1,\ldots, m\}$, a point $x_p$ must satisfy the following conditions:
\begin{itemize}
	\item Its rank must be 0 and its crowding distance must be greater than or equal to $\bar{c}_{k + 1}$ (these requirements are already discussed in Section \ref{subsubsec::getcrowdingdistancethreshold}).
	\item It must belong to $\hat{X}_{k + 1}^I$, which is the set of mutually non-dominated points w.r.t.\ $F_I$ contained in $\hat{X}_{k + 1}$ (a formal definition of this set can be found in Section \ref{subsec::descent-methods}). Trying to optimize points which are not contained in $\hat{X}_{k + 1}^I$ could be useless, since we have no guarantee to reach a non-dominated point w.r.t.\ $F_I$.
	\item It must not be Pareto-stationary w.r.t.\ $F_I$.
\end{itemize}
If the point satisfies all these requirements, it is used as starting solution in the \texttt{FMOPG} algorithm. Along with it, the set $\hat{X}_{k + 1}^I$ is used as input for the algorithm. \texttt{FMOPG} returns the set of produced solutions, which are collected in the set $\tilde{X}_p$ and inserted in the set $\hat{X}_{k + 1}$. 

Lastly, the new population $\hat{X}_{k + 1}$ is reduced in order to have exactly $N$ survivors. This last operation is performed through the \texttt{getMetrics} (Section \ref{subsubsec::getmetrics}) and the \texttt{getSurvivors} (Section \ref{subsubsec::getsurvivors}) functions of the \texttt{NSGA-II} algorithm.

\subsection{The Front Multi-Objective Projected Gradient algorithm}
\label{subsec::FMOPG}

The \textit{Front Multi-Objective Projected Gradient} (\texttt{FMOPG}) algorithm is the descent method used in our memetic approach. In particular, it is a variant of the \texttt{MOPG} method (Algorithm \ref{alg::MOPG}).

\subsubsection{Algorithmic scheme}
\label{subsubsec:algorithmic-scheme}

We report the scheme of \texttt{FMOPG} in Algorithm \ref{alg::FMOPG}.

\begin{algorithm}[h]
	\caption{Front Multi-Objective Projected Gradient} \label{alg::FMOPG}
	Input: $F: \mathbb{R}^n \rightarrow \mathbb{R}^m$, $\Omega$ feasible closed and convex set, $I \subseteq \{1,\ldots, m\}$, $X_0\subset\Omega$, $x_0 \in X_0$.\\
	$k = 0$ \\
	\While{$x_k$ is not Pareto-stationary w.r.t.\ $F_I$}{
		Compute 
		\begin{equation*}
			\theta_\Omega^I(x_k) = \min_{\substack{z \in \Omega \\ \| z - x_k \| \le 1}} \max_{j \in I} \nabla f_j(x_k)^T(z - x_k)
		\end{equation*}\\
		Let $d_{\Omega k}^I \in v_\Omega^I(x_k)$ be the direction associated with $\theta_\Omega^I(x_k)$ \\
		$\alpha_k$ = \texttt{B-FALS}($F(\cdot), \Omega, I, X_k^I, x_k, d_{\Omega k}^I, \theta_\Omega^I(x_k)$) \\
		$x_{k + 1} = x_k + \alpha_k d_{\Omega k}^I$ \\
		$X_{k + 1} = X_k \cup \{x_{k + 1}\}$ \\
		$k = k + 1$
	}
	\Return sequence $\{x_k\}$
\end{algorithm}

The main difference between \texttt{FMOPG} and \texttt{MOPG} is the following: while in the original algorithm the current point $x_k$ is only optimized w.r.t.\ itself, in \texttt{FMOPG} it is also w.r.t.\ the set of points in which it is contained. 
At each iteration, the constrained steepest descent direction $d_{\Omega k}^I$ at the current point $x_k$ w.r.t.\ the subset of indices of the objectives $I$ is found. Then, a step size $\alpha_k$ is calculated by the \texttt{B-FALS} procedure (Algorithm \ref{alg::B-FALS}). Given the direction and the step size, a new point $x_{k + 1}$ is obtained. This latter one is inserted in the set $X_k$, leading to a new set $X_{k + 1}$.

The \texttt{FMOPG} algorithm iterates until the current solution $x_k$ is Pareto-stationary w.r.t.\ $F_I$. At the end, the method returns the sequence of points $\{x_k\}$ generated during the iterations. Indeed, considering the stopping conditions of \texttt{B-FALS}, we have no guarantee that, for all $k$, the point $x_{k + 1}$ dominates $x_k$ w.r.t.\ $F_I$. So, every point produced by \texttt{FMOPG} could be useful to obtain good and spread Pareto front approximations.

Finally, note that the \texttt{FMOPG} algorithm is called by the \texttt{optimizePopulation} function with an additional parameter $\varepsilon_t$ (Line \ref{line::call-FMOPG} of Algorithm \ref{alg::optimizePopulation}). In fact, \texttt{FMOPG} is executed using $\varepsilon$-Pareto-stationarity as stopping condition. In \texttt{NSMA} (Algorithm \ref{alg::NSMA}), we consider a decreasing sequence $\{\varepsilon_t\} \subset \mathbb{R}_0^+$. So, during the iterations, we get closer and closer to the Pareto-stationarity.

\subsubsection{Algorithm analysis}
\label{subsubsec::algorithm-analysis}

In this section, we provide a rigorous analysis of the \texttt{FMOPG} algorithm from a theoretical perspective.
Before proceeding, we need to state an assumption.

\begin{assumption}
	\label{ass::non-dom-x0}
	Let $I \subseteq \{1,\ldots, m\}$, $X_0\subset\Omega$ be a set of feasible points and $x_0 \in X_0$. There does not exist a point $y_0 \in X_0$ that dominates $x_0$ w.r.t.\ $F_I$, i.e.\ $x_0 \in X_0^I$.
\end{assumption}

This assumption is reasonable since a point $x_p$ to be optimized through \texttt{FMOPG} must be non-dominated w.r.t.\ $F_I$ (Section \ref{subsubsec::optimizepopulation}).

We begin characterizing the points produced by the \texttt{FMOPG} algorithm.

\begin{proposition}
	\label{prop::non-dom-new-po}
	Consider a generic iteration $k$ of \texttt{FMOPG}. Let $I \subseteq \{1,\ldots, m\}$, $X_k$ be a set of feasible points and $x_k \in X_k$. Assume that $x_k$ is not dominated by any point in $X_k$ w.r.t.\ $F_I$. Then, \texttt{B-FALS} returns a step size $\alpha_k > 0$ such that the point $x_{k+1} = x_k+\alpha_kd_{\Omega k}^I$ is feasible and not dominated by any point in $X_{k+1}$ w.r.t.\ $F_I$.
\end{proposition}
\begin{proof}
	The \texttt{B-FALS} algorithm is performed from $x_k \in X_k^I$, with $\theta_\Omega^I(x_k) < 0$, along a constrained steepest descent direction $d_{\Omega k}^I$. Then, from Proposition \ref{prop::fin-ter-B-FALS}, \texttt{B-FALS} terminates in a finite number of steps and returns a step size $\alpha_k > 0$ such that the point $x_{k + 1} = x_k + \alpha_kd_{\Omega k}^I$ has the following properties: 
		\begin{itemize}
			\item $x_{k + 1} \in \Omega$;
			\item $x_{k + 1}$ is not dominated by any other point in $X_k$ w.r.t. $F_I$.
		\end{itemize}
	Since $X_{k+1} = X_{k} \cup \{x_{k+1}\}$, the assertion is finally proved.
\end{proof}

\begin{remark}
	\label{rem::non-dom-new-po}
	Since the point $x_{k + 1}$ induced by the step size produced by \texttt{B-FALS} is not dominated by any point in $X_{k + 1}$ w.r.t.\ $F_I$, we can easily conclude that the new point is also not dominated w.r.t.\ all the objectives.
\end{remark}

Given Proposition \ref{prop::non-dom-new-po}, we can state the following corollary.

\begin{corollary}
	\label{cor::non-dom-new-po}
	Let Assumption \ref{ass::non-dom-x0} hold with $I \subseteq \{1,\ldots, m\}$, the set $X_0$ and the point $x_0$. Then, the sequence of sets $\{X_k\}$ and the sequence of points $\{x_k\}$ generated by \texttt{FMOPG} are such that for all $k=0,1,\ldots$, $x_{k}$ is feasible and not dominated by any point in $X_k$ w.r.t.\ $F_I$.
\end{corollary}
\begin{proof}
	The assertion straightforwardly follows if the assumptions of Proposition \ref{prop::non-dom-new-po} are satisfied at every iteration $k$ of the algorithm. 
	
	When $k=0$, this is guaranteed by Assumption \ref{ass::non-dom-x0}. The case of a generic iteration $k$ simply follows by induction from Proposition \ref{prop::non-dom-new-po} itself.
\end{proof}

Before proceeding with the convergence analysis, we make a further reasonable assumption. This hypothesis is similar to Assumption 1 in \cite{cocchi2020convergence}. However, in this context, Assumption \ref{ass::non-dom-x0} and bound constraints must be also taken into account.

\begin{assumption}
	\label{ass::par-stat-bound}
	Assumption \ref{ass::non-dom-x0} holds and $x_0$ is also such that:
	\begin{itemize}
		\item $x_0$ is not Pareto-stationary w.r.t.\ $F_I$;
		\item the set $\mathcal{L}(x_0) = \bigcup_{j = 1}^{m} \{x \in \Omega: f_j(x) \le f_j(x_0)\}$ is compact.
	\end{itemize}
\end{assumption}

This assumption is stronger than the one required to prove convergence of the \texttt{MOSD} method (Lemma \ref{lem::conv-MOSD}). However, as also observed in \cite{cocchi2020convergence} for \texttt{FALS} (Algorithm \ref{alg::FALS}), this is reasonable since the second stopping criterion of \texttt{B-FALS} is weaker than the one used in \texttt{ALS} (Algorithm \ref{alg::ALS}).

\begin{proposition}
	\label{prop::conv-lim-po}
	Let Assumption \ref{ass::par-stat-bound} hold with $I \subseteq \{1,\ldots, m\}$, the set $X_0$ and the point $x_0$. Let $\{x_k\}$ be the sequence of points generated by \texttt{FMOPG}. Then $\{x_k\}$ admits limit points and every limit point is Pareto-stationary considering the objectives $f_j$, with $j \in I$.
\end{proposition}

\begin{proof}
	Firstly, we prove that the sequence $\{x_k\}$ admits limit points.  
	Since $x_0\in X_k$ for all $k$, Corollary \ref{cor::non-dom-new-po} guarantee that, for each $k$, $x_k \in \Omega$ and there exists an index $j(x_k) \in I$ such that 
	\begin{equation*}
		f_{j(x_k)}(x_k) \le f_{j(x_k)}(x_0).
	\end{equation*}
	So, 
	\begin{equation*}
		x_k \in \{x \in \Omega: f_{j(x_k)}(x) \le f_{j(x_k)}(x_0)\}
	\end{equation*}
	and, therefore, 
	\begin{equation*}
		x_k \in \mathcal{L}(x_0), \hspace{0.3cm} \forall k.
	\end{equation*}
	Assumption \ref{ass::par-stat-bound} assures that the sequence $\{x_k\}$ is bounded. Hence, this latter one admits limit points: we can consider a subsequence $K \subseteq \{1, 2,\ldots\}$ such that
	\begin{equation*}
		\lim_{\substack{k \rightarrow \infty \\ k \in K}} x_k = \bar{x}.
	\end{equation*}
	
	We recall that $\bar{x}$ is Pareto-stationary w.r.t.\ $F_I$ if and only if $\theta_\Omega^I(\bar{x}) = 0$. By contradiction, we assume that $\bar{x}$ is not Pareto-stationary w.r.t.\ $F_I$: there exists $\bar{\varepsilon} > 0$ such that
	\begin{equation}
		\label{eq::theta-not-0}
		\theta_\Omega^I(x_k) \le -\bar{\varepsilon} < 0, \hspace{0.3cm} \forall k \in K.
	\end{equation} 
	Next, we want to prove the following statement:
	\begin{equation}
		\label{eq::lim-alpha-to-0}
		\lim_{\substack{k \rightarrow \infty \\ k \in K}} \alpha_k\theta_\Omega^I(x_k) = 0.
	\end{equation}
	Again, by contradiction, we assume that the assertion is not true: there exist a subsequence $\bar{K} \subseteq K$ and $\bar{\eta} > 0$ such that
	\begin{equation}
		\label{eq::alpha-not-0}
		\alpha_k\theta_\Omega^I(x_k) \le -\bar{\eta} < 0, \hspace{0.3cm} \forall k \in \bar{K}.
	\end{equation}
	Recalling Proposition \ref{prop::fin-ter-B-FALS} and Corollary \ref{cor::non-dom-new-po}, for all $k \in \bar{K}$, \texttt{B-FALS} returns in a finite number of iterations a step size $\alpha_k$ such that $x_{k + 1} = x_k+\alpha_kd_{\Omega k}^I$ satisfies
	\begin{equation*}
		F_I(y_k) + \mathbf{1}\beta\alpha_k\theta_\Omega^I(x_k) \not < F_I(x_{k+1}),
	\end{equation*}
	for all $y_k \in X_k$.
	By using Equation \eqref{eq::alpha-not-0}, we obtain that, for all $k \in \bar{K}$ and for all $y_k \in X_k$, 
	\begin{equation}
		F_I(y_k) - \mathbf{1}\beta\bar{\eta} \not < F_I(x_{k+1}).
	\end{equation}
	Let $\bar{k}(k)$ be the smallest index such that $\bar{k}(k) \in \bar{K}$, $\bar{k}(k) > k$ and, for all $y_{\bar{k}(k) - 1} \in X_{\bar{k}(k) - 1}$,
	\begin{equation}
		\label{eq::non-dom-bar-k-y}
		F_I(y_{\bar{k}(k) - 1}) - \mathbf{1}\beta\bar{\eta} \not < F_I(x_{\bar{k}(k)}).
	\end{equation}
	Now, there are two cases: $k = \bar{k}(k) - 1$ or $k < \bar{k}(k) - 1$. In the first case, $x_k \in X_{\bar{k}(k) - 1}$ and it satisfies Equation \eqref{eq::non-dom-bar-k-y}. In the second case, the statement is still satisfied since 
	\begin{equation*}
		X_{\bar{k}(k) - 1} = X_0 \cup \{x_1\} \cup\ldots\cup\{x_k\}\cup\ldots\{x_{\bar{k}(k) - 1}\}.
	\end{equation*}
	So, we have that
	\begin{equation}
		\label{eq::non-dom-bar-k-x}
		F_I(x_k) - \mathbf{1}\beta\bar{\eta} \not < F_I(x_{\bar{k}(k)}).
	\end{equation}
	Equation \eqref{eq::non-dom-bar-k-x} means that there exists an index $j(x_k) \in I$ such that
	\begin{equation*}
		f_{j(x_k)}(x_k) - \beta\bar{\eta} \ge f_{j(x_k)}(x_{\bar{k}(k)}).
	\end{equation*}
	Since the set $I$ is finite, we can consider a subsequence $\tilde{K} \subseteq \bar{K}$ such that, for all $k \in \tilde{K}$, $j(x_k) = \hat{j}$. Employing the same procedure used for $\bar{k}(k)$, let $\tilde{k}(k)$ be the smallest index such that $\tilde{k}(k) \in \tilde{K}$, $\tilde{k}(k) > k$ and
	\begin{equation*}
		f_{\hat{j}}(x_k) - \beta\bar{\eta} \ge f_{\hat{j}}(x_{\tilde{k}(k)}).
	\end{equation*}
	The definitions of $\tilde{K}$ and $\tilde{k}(k)$ imply that
	\begin{equation*}
		\lim_{\substack{k \rightarrow \infty\\k \in \tilde{K}}}x_k = \bar{x}$$ and $$\lim_{\substack{k \rightarrow \infty\\k \in \tilde{K}}}x_{\tilde{k}(k)} = \bar{x}.
	\end{equation*}
	Then, by taking the limits for $k \rightarrow \infty$ and recalling the continuity of $F$, we have that
	\begin{equation*}
		\lim_{\substack{k \rightarrow \infty\\k \in \tilde{K}}} \left[f_{\hat{j}}(x_{\tilde{k}(k)}) - f_{\hat{j}}(x_k)\right] = 0
	\end{equation*}
	and, therefore, 
	\begin{equation*}
		-\beta\bar{\eta} \ge 0.
	\end{equation*}
	Since $\beta\bar{\eta} > 0$, we get the contradiction.
	
	Therefore, Equation \eqref{eq::lim-alpha-to-0} holds and, recalling Equation \eqref{eq::theta-not-0}, we obtain the following statement:
	\begin{equation*}
		\lim_{\substack{k \rightarrow \infty \\ k \in K}} \alpha_k = 0.
	\end{equation*}
	Given this limit, we can consider sufficiently large values for $k \in K$ such that
	\begin{equation}
		\label{eq::gre-k-sma-alpha}
		\alpha_k < \frac{\alpha_k}{\delta} \le 1.
	\end{equation}
	Since $\Omega$ is convex and $d_{\Omega k}^I$ is a feasible direction by construction, Equation \eqref{eq::gre-k-sma-alpha} implies that the point $x_k + (\alpha_k / \delta) d_{\Omega k}^I \in \Omega$. Therefore, for sufficiently large values for $k \in K$, the stopping conditions of \texttt{B-FALS} imply that there exists a point $y_k \in X_k$ such that 
	\begin{equation}
		\label{eq::dom-prev-alpha}
		F_I(y_k) + \mathbf{1}\beta\frac{\alpha_k}{\delta}\theta_\Omega^I(x_k) < F_I\left(x_k + \frac{\alpha_k}{\delta}d_{\Omega k}^I\right).
	\end{equation}
	Considering Equation \eqref{eq::dom-prev-alpha} and Corollary \ref{cor::non-dom-new-po}, we have that an index $j(x_k) \in I$ exists such that the following statement is true: 
	\begin{equation*}
		f_{j(x_k)}(x_k) + \beta\frac{\alpha_k}{\delta}\theta_\Omega^I(x_k) \le f_{j(x_k)}(y_k) + \beta\frac{\alpha_k}{\delta}\theta_\Omega^I(x_k)
	\end{equation*}
	and
	\begin{equation*}
		f_{j(x_k)}(y_k) + \beta\frac{\alpha_k}{\delta}\theta_\Omega^I(x_k) < f_{j(x_k)}\left(x_k + \frac{\alpha_k}{\delta}d_{\Omega k}^I\right).
	\end{equation*}
	Since the set $I$ is finite, we can consider a subsequence $\bar{K} \subseteq K$ such that, for sufficiently large values for $k \in \bar{K}$, $j(x_k) = \hat{j}$ and 
	\begin{equation*}
		f_{\hat{j}}\left(x_k + \frac{\alpha_k}{\delta}d_{\Omega k}^I\right) - f_{\hat{j}}(x_k) > \beta\frac{\alpha_k}{\delta}\theta_\Omega^I(x_k).
	\end{equation*}
	Using the Mean-value Theorem, we have that 
	\begin{equation*}
		f_{\hat{j}}\left(x_k + \frac{\alpha_k}{\delta}d_{\Omega k}^I\right) - f_{\hat{j}}(x_k) = \frac{\alpha_k}{\delta}\nabla f_{\hat{j}}(\xi_k)^Td_{\Omega k}^I,
	\end{equation*}
	with
	\begin{equation*}
		\xi_k = x_k + t_k \frac{\alpha_k}{\delta}d_{\Omega k}^I, \hspace{0.3cm} t_k \in (0, 1).
	\end{equation*}
	Then, we can write 
	\begin{equation*}
		\nabla f_{\hat{j}}(\xi_k)^Td_{\Omega k}^I > \beta\theta_\Omega^I(x_k),
	\end{equation*}
	from which we can state that 
	\begin{equation*}
		\nabla f_{\hat{j}}(x_k)^Td_{\Omega k}^I + \left[\nabla f_{\hat{j}}(\xi_k) - \nabla f_{\hat{j}}(x_k)\right]^Td_{\Omega k}^I > \beta\theta_\Omega^I(x_k).
	\end{equation*}
	Since $\hat{j} \in I$, we have that 
	\begin{equation*}
		\theta_\Omega^I(x_k) = \max_{j \in I}\nabla f_j(x_k)^Td_{\Omega k}^I \ge \nabla f_{\hat{j}}(x_k)^Td_{\Omega k}^I
	\end{equation*}
	and 
	\begin{equation*}
		(1 - \beta)\theta_\Omega^I(x_k) + \left[\nabla f_{\hat{j}}(\xi_k) - \nabla f_{\hat{j}}(x_k)\right]^Td_{\Omega k}^I > 0.
	\end{equation*}
	Using Equation \eqref{eq::theta-not-0}, we obtain 
	\begin{equation*}
		-(1 - \beta)\bar{\varepsilon} + \left[\nabla f_{\hat{j}}(\xi_k) - \nabla f_{\hat{j}}(x_k)\right]^Td_{\Omega k}^I > 0.
	\end{equation*}
	By taking the limit for $k \rightarrow \infty, k \in \bar{K}$, recalling the continuity of $J_F$, the boundedness of $d_{\Omega k}^I$ and that $\alpha_k \rightarrow 0$, we get that 
	\begin{equation*}
		-(1 - \beta)\bar{\varepsilon} > 0.
	\end{equation*}
	Since $1 - \beta > 0$ and $\bar{\varepsilon} > 0$, we get the contradiction. So, we prove that the limit point $\bar{x}$ of the sequence $\{x_k\}$ is Pareto-stationary w.r.t.\ $F_I$.
\end{proof}

Finally, we prove that, when a stopping criterion based on the $\varepsilon$-Pareto-stationarity is considered, \texttt{FMOPG} is well defined, i.e., it terminates in a finite number of iterations.

\begin{proposition}
	\label{prop::fin-ter-FMOPG}
	Let Assumption \ref{ass::par-stat-bound} hold with $I \subseteq \{1,\ldots, m\}$, the set $X_0$ and the point $x_0$. Let $\varepsilon > 0$. Then, the \texttt{FMOPG} algorithm finds in a finite number of steps a point $x_k$ which is $\varepsilon$-Pareto-stationary w.r.t.\ $F_I$.
\end{proposition}
\begin{proof}
	We assume, by contradiction, that \texttt{FMOPG} produces an infinite sequence of points $\{x_k\}$ such that, for all $k$, $x_k$ is not $\varepsilon$-Pareto-stationary w.r.t.\ $F_I$. 	
	Since Assumption \ref{ass::par-stat-bound} holds, Proposition \ref{prop::conv-lim-po} ensures that there exists a subsequence $K \subseteq \{1, 2,\ldots\}$ such that 
	\begin{equation*}
		\lim_{\substack{k \rightarrow \infty \\ k \in K}} x_k = \bar{x}
	\end{equation*}
	and $\bar{x}$ is Pareto-stationary w.r.t.\ $F_I$, i.e., recalling Definition \ref{def::par-stat}, for all $z \in \Omega$ such that $\left\| z - \bar{x} \right\| \le 1$ we have that 
	\begin{equation*}
		\max_{j \in I}\nabla f_j(\bar{x})^T(z - \bar{x})\ge 0.
	\end{equation*}
	Given the continuity of the max operator and $J_F$, we can state that 
	\begin{equation*}
		\lim_{\substack{k \rightarrow \infty \\ k \in K}} \max_{j \in I} \nabla f_j(x_k)^T(z - x_k) \ge 0 > -\varepsilon.
	\end{equation*}
	This last statement implies that, for sufficiently large values for $k \in K$, for all $z \in \Omega$ such that $\left\| z - x_k \right\| \le 1$, the following equation has to hold: 
	\begin{equation*}
		\max_{j \in I} \nabla f_j(x_k)^T(z - x_k) > -\varepsilon,
	\end{equation*}
	i.e.\ $x_k$ is $\varepsilon$-Pareto-stationary w.r.t.\ $F_I$. Therefore, we get the contradiction and the assertion is proved.
\end{proof}

\section{Computational experiments}
\label{sec::computational-experiments}

In this section, we provide the results of thorough computational experiments, focusing on the comparison of \texttt{NSMA} with the main state-of-the-art methods in diverse settings. The code of all the algorithms was written in Python3. In addition, all the tests were run on a computer with the following characteristics: Ubuntu 20.04, Intel Xeon Processor E5-2430 v2 6 cores 2.50 GHz, 16 GB RAM. We used the Gurobi Optimizer (Version 9) to solve instances of Problem \eqref{eq::proj-desc}.

\subsection{Settings}
\label{subsec::settings}

In this section, we report detailed information on the settings used for all the considered algorithms in our experiments, the metrics and the problems used to carry out the comparison.

\subsubsection{Metrics}
\label{subsubsec::metrics}

In this section, we provide a little description of the metrics and tools used to compare the algorithms.

The first three metrics are the ones introduced in \cite{custodio2011direct}: \textit{purity}, $\Gamma$\textit{--spread} and $\Delta$\textit{--spread}. These metrics are widely used to evaluate the performance of multi-objective optimization algorithms. 

We recall that the \textit{purity} metric measures the quality of the generated front, i.e., how effective a solver is at obtaining non-dominated points w.r.t.\ its competitors. 
In detail, the purity value indicates the ratio of the number of non-dominated points that a solver obtained over the number of the points produced by that solver. Clearly, a higher value is related to a better performance. In order to calculate the \textit{purity} metric, we need a reference front to establish whether a point is dominated or not. In our experiments, we considered as the reference front the one obtained by combining the fronts retrieved by all the considered algorithms and by discarding the dominated points. 

The \textit{spread} metrics are equally essential, since they measure the uniformity of the generated fronts in the objectives space. The $\Gamma$\textit{--spread} is defined as the maximum $\ell_\infty$ distance in the objectives space between adjacent points of the Pareto front, while the $\Delta$\textit{--spread} basically measures the standard deviation of the $\ell_\infty$ distance between adjacent Pareto front points. As opposed to the \textit{purity}, low values for the \textit{spread} metrics are associated with good performance.

In addition to the previous metrics, we used the \textit{ND-points} metric, introduced in \cite{COCCHI2021100008}. This score substantially indicates the number of non-dominated points obtained by a solver w.r.t.\ the reference front. We consider this metric as important as the \textit{purity} one: in particular, we think that these two metrics should be considered complementary.

Lastly, we employed the performance profiles introduced in \cite{dolan2002benchmarking} to carry out the comparison. Performance profiles are a useful tool to appreciate the relative performance and robustness of the considered algorithms. The performance profile of a solver w.r.t.\ a certain metric is the (cumulative) distribution function of the ratio of the score obtained by a solver over the best score among those obtained by all the considered solvers. 
In other words, it is the probability that the score achieved by a solver in a problem is within a factor $\tau \in \mathbb{R}$ of the best value obtained by any of the solvers in that problem. For a more technical explanation, we refer the reader to \cite{dolan2002benchmarking}. Note that performance profiles w.r.t.\ \textit{purity} and \textit{ND-points} were produced based on the inverse of the obtained values, since the metrics have increasing values for better solutions.

\subsubsection{Algorithms and hyper-parameters}
\label{subsubsec::algorithms_and_hyper_paramaters}

The first two algorithms we chose for the comparisons are, naturally, the \texttt{NSGA-II} \cite{deb2002fast} and the \texttt{FPGA} \cite{cocchi2020convergence} procedures, described in Section \ref{subsec::NSGA-II} and Appendix \ref{app::FPGA}, respectively. We consider these methods as representatives for EAs and descent methods and, thus, \texttt{NSMA} most direct competitors. The parameters values for both algorithms were chosen according to the reference papers. Like \texttt{NSMA}, in \texttt{NSGA-II} the number $N$ of solutions in the population was fixed to 100.

Then, the values for the parameters of \texttt{NSMA} were chosen based on some preliminary experiments on a subset of the tested problems, which we do not report here for the sake of brevity. The values are:

\begin{itemize}
	\item $N = 100$;
	\item $s_h = 10$;
	\item $q = 0.9$;
	\item $n_{opt} = 5$;
	\item in \texttt{B-FALS} $\alpha_0 = 1$, $\beta = 10^{-4}$, $\delta = 0.5$.
\end{itemize}

We also consider in the experiments the \texttt{DMS} algorithm \cite{custodio2011direct}, a multi-objective derivative-free method, inspired by the search/poll paradigm of direct-search methodologies of directional type. \texttt{DMS}  maintains a list of non-dominated points, from which the new iterates or poll centers are chosen. The parameters for this method were set according to the reference paper and the code available online (\url{http://www.mat.uc.pt/dms}).

In most of the computational experiments, for each algorithm and problem we ran the test for up to 2 minutes. A stopping criterion based on a time limit is the fairest way to compare such structurally different algorithms. Obviously, we also took into account other specific stopping criteria indicating that a certain algorithm cannot improve the solutions anymore. 

\texttt{NSMA} and \texttt{NSGA-II} are non-deterministic algorithm. Therefore, we decided to run them 5 times on every problem, with different seeds for the pseudo-random number generator. Every execution was characterized by the same time limit (2 minutes). The five generated fronts were compared based on the \textit{purity} metric and only the best one was chosen as the output of \texttt{NSMA}/\texttt{NSGA-II}. In this context, the reference front was the combination of the fronts of the 5 executions. Executing 5 runs lets \texttt{NSMA}/\texttt{NSGA-II} reduce its sensibility to the seed used for its random operations. On the other side, \texttt{FPGA} and \texttt{DMS} are deterministic and, then, they were executed once.

\subsubsection{Problems}
\label{subsubsec::problems}

The problems constituting the benchmark of the computational experiments are listed in Table \ref{tab::problems}. In this benchmark, we considered problems whose objective functions are at least continuously differentiable almost everywhere. If a problem is characterized by singularities, we counted these latter ones as Pareto-stationary points. All the constraints are defined by finite lower and upper bounds.

\begin{table}[t]
	\centering
	\renewcommand{\arraystretch}{1.15}
	\caption{Problems used in the computational experiments.}
	\label{tab::problems}
	\begin{tabular}{|c||c|c|}
		\hline
		\rule{0pt}{12pt}\texttt{PROBLEM}&$\texttt{n}$&$\texttt{m}$ \\
		\hline
		\hline
		CEC09\_1, CEC09\_2, CEC09\_3,  & 5, 10, 20, 30, & \multirow{2}{*}{2} \\
		CEC09\_4, CEC09\_5, CEC09\_6, CEC09\_7& 40, 50, 100, 200 & \\
		\hline
		\multirow{2}{*}{CEC09\_8, CEC09\_9, CEC09\_10}  & 5, 10, 20, 30, & \multirow{2}{*}{3} \\
		& 40, 50, 100, 200 & \\
		\hline
		\multirow{2}{*}{ZDT\_1, ZDT\_2, ZDT\_3, ZDT\_4	}  & 2, 5, 10, 20, 30, & \multirow{2}{*}{2} \\
		& 40, 50, 100, 200 & \\
		\hline
		MOP\_1 & 1 & 2 \\
		\hline
		\multirow{2}{*}{MOP\_2} & 2, 5, 10, 20, 30, & \multirow{2}{*}{2} \\
		& 40, 50, 100, 200 & \\
		\hline
		MOP\_3 & 2 & 2 \\
		\hline
		\multirow{2}{*}{MAN}  & 2, 5, 10, 20, 30, & \multirow{2}{*}{2} \\
		& 40, 50, 100, 200 & \\
		\hline
	\end{tabular}
\end{table}

The set is mainly composed by the CEC09 problems \cite{zhang_multiobjective_2009}, the ZDT problems \cite{ziztler2000} and the MOP problems \cite{Huband2006}. In particular, some of the CEC09 and the ZDT problems have particularly difficult objective functions. Hence, these problems are particularly interesting for the analysis of the behavior of the algorithms with hard tasks. 

We also defined a new test problem with convex objective functions: we refer to it as the MAN problem. Its formal definition is the following:

\begin{equation*}
	\begin{aligned}
		\min_{x\in \mathbb{R}^n}\quad& \begin{array}{l}
			f_1(x) = \sum_{i = 1}^{n} (x_i - i)^2 / n^2\\\\ f_2(x) =  \sum_{i=1}^{n} e^{-x_i} + x_i
		\end{array}\\\\\text{s.t.}\quad&x \in [-10^4, 10^4].
	\end{aligned}
\end{equation*}

Inspired by Cust\'odio \textit{et al.} \cite{custodio2011direct}, for each problem the initial points were uniformly selected from the hyper-diagonal defined by the bound constraints. Furthermore, the number of initial points is equal to the dimension $n$ of the problem. Since in the MOP\_1 problem $n=1$, only in this case we started the tests from one feasible point, namely, $x=0$.

\subsection{Experimental comparisons between \texttt{NSGA-II} and \texttt{FPGA}}
\label{subsec::NSGA-II-vs-FPGA}
Before turning to the evaluation of the \texttt{NSMA}, we carry out a preliminary study.

Evolutionary algorithms and descent methods have their own drawbacks. In particular, EAs do not have theoretical convergence properties. In addition, they can be very expensive in particular settings. On the other side, descent algorithms suffer on highly non-convex problems: in these cases, they often produce sub-optimal solutions, especially when the starting points are not chosen carefully.

In this section, we want to address two topics:
\begin{itemize}
	\item the impact of convexity of the objective functions on the performance of these algorithms;
	\item the  behavior of the methods as the problem dimension $n$ increases.
\end{itemize}

For the comparisons of this section, we only considered the \texttt{NSGA-II} and \texttt{FPGA} algorithms, which we respectively pick as representatives for the two classes of methods.

As benchmark, we picked four problems that are scalable w.r.t.\ the problem dimension $n$ and have the following features:
\begin{itemize}
	\item the MAN problem and the ZDT\_1 problem have convex objective functions;
	\item the CEC09\_4 problem and the ZDT\_3 problem have nonconvex objective functions.
\end{itemize}
For these comparisons, each problem was tested for values of $n\in \{5, 10,\allowbreak 20, 30, 40, 50, 100, 200\}$.

\begin{figure*}[!t]
	\centering
	\subfloat[]{\includegraphics[width=1.5in]{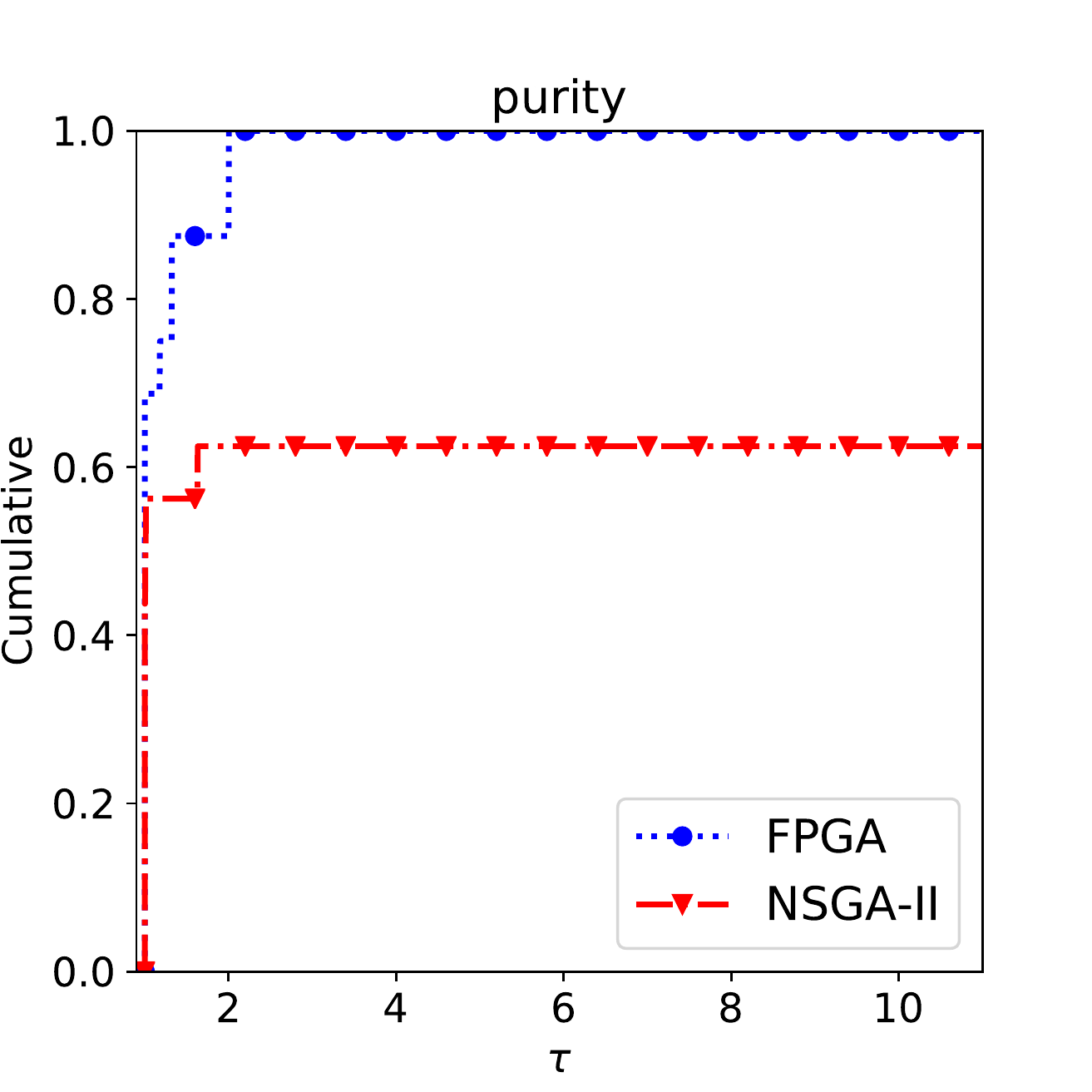}}
	\hfil
	\subfloat[]{\includegraphics[width=1.5in]{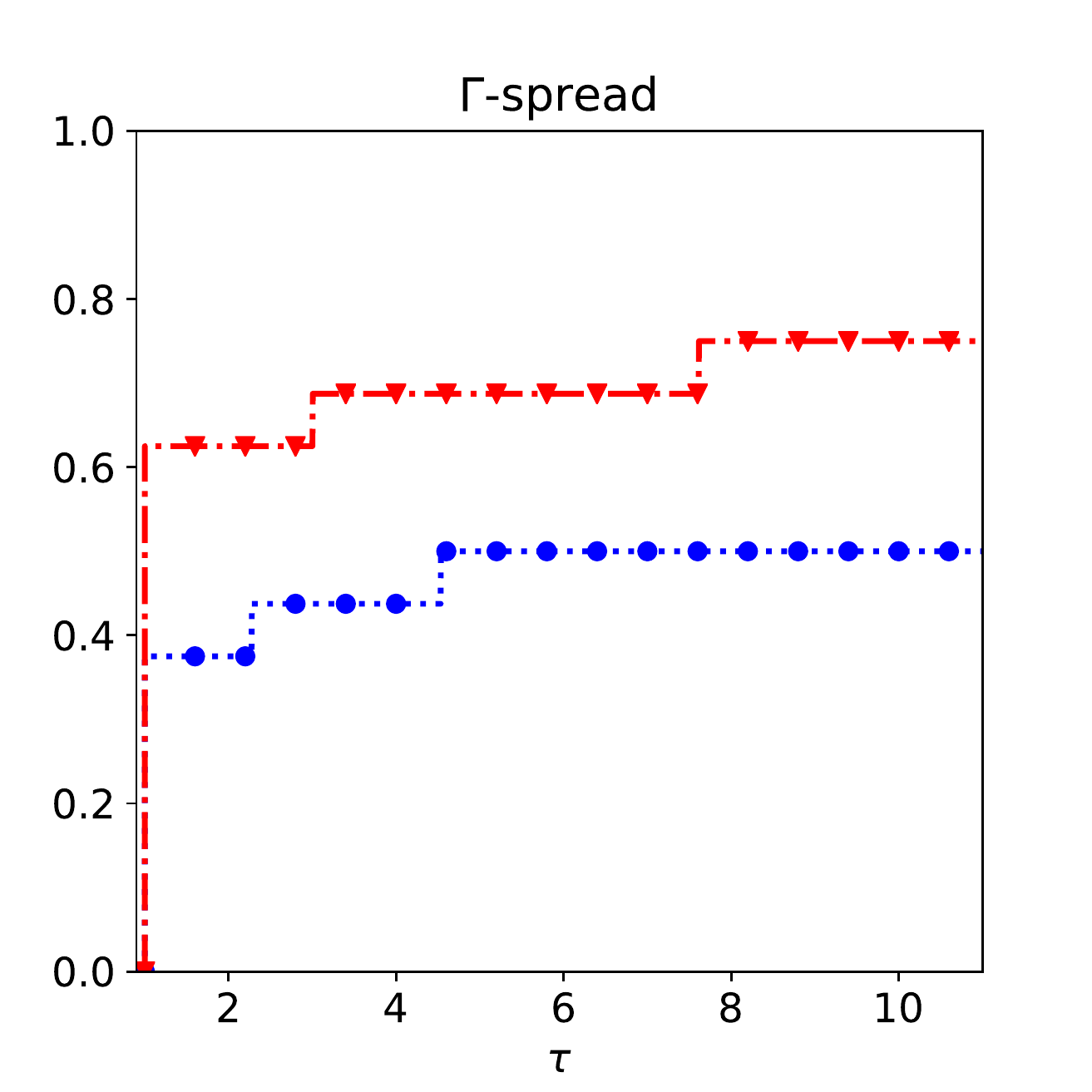}}
	\hfil
	\subfloat[]{\includegraphics[width=1.5in]{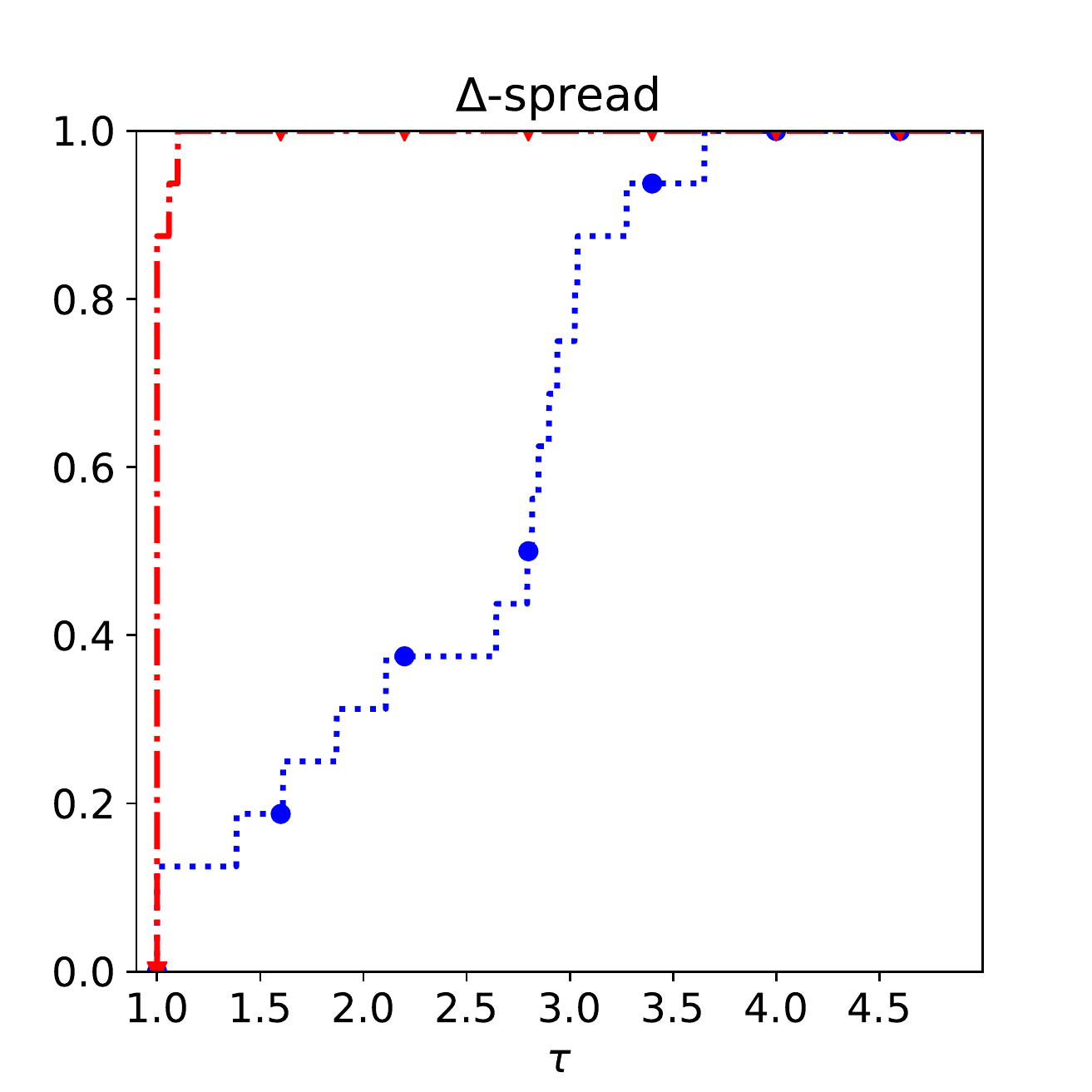}}
	\caption{Performance profiles for \texttt{FPGA} and \texttt{NSGA-II} on the \textit{convex} MAN and ZDT\_1 problems (for interpretation of the references to color in text, the reader is referred to the web version of the article). (a) \textit{purity.} (b) $\Gamma$\textit{--spread.} (c) $\Delta$\textit{--spread.}}
	\label{fig::conv-prob-perf-prof}
\end{figure*}

\begin{figure*}[!t]
	\centering
	\subfloat[]{\includegraphics[width=1.5in]{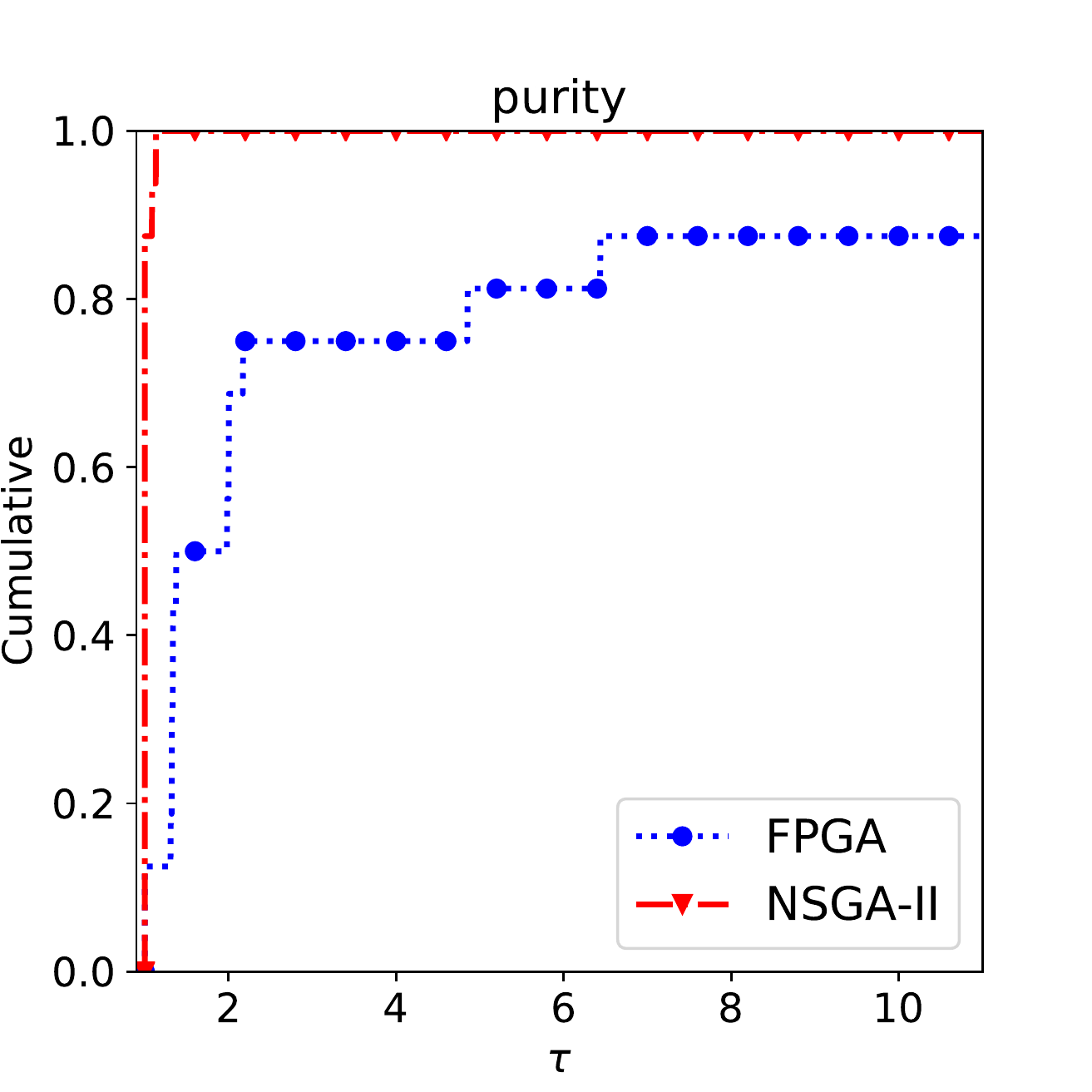}}
	\hfil
	\subfloat[]{\includegraphics[width=1.5in]{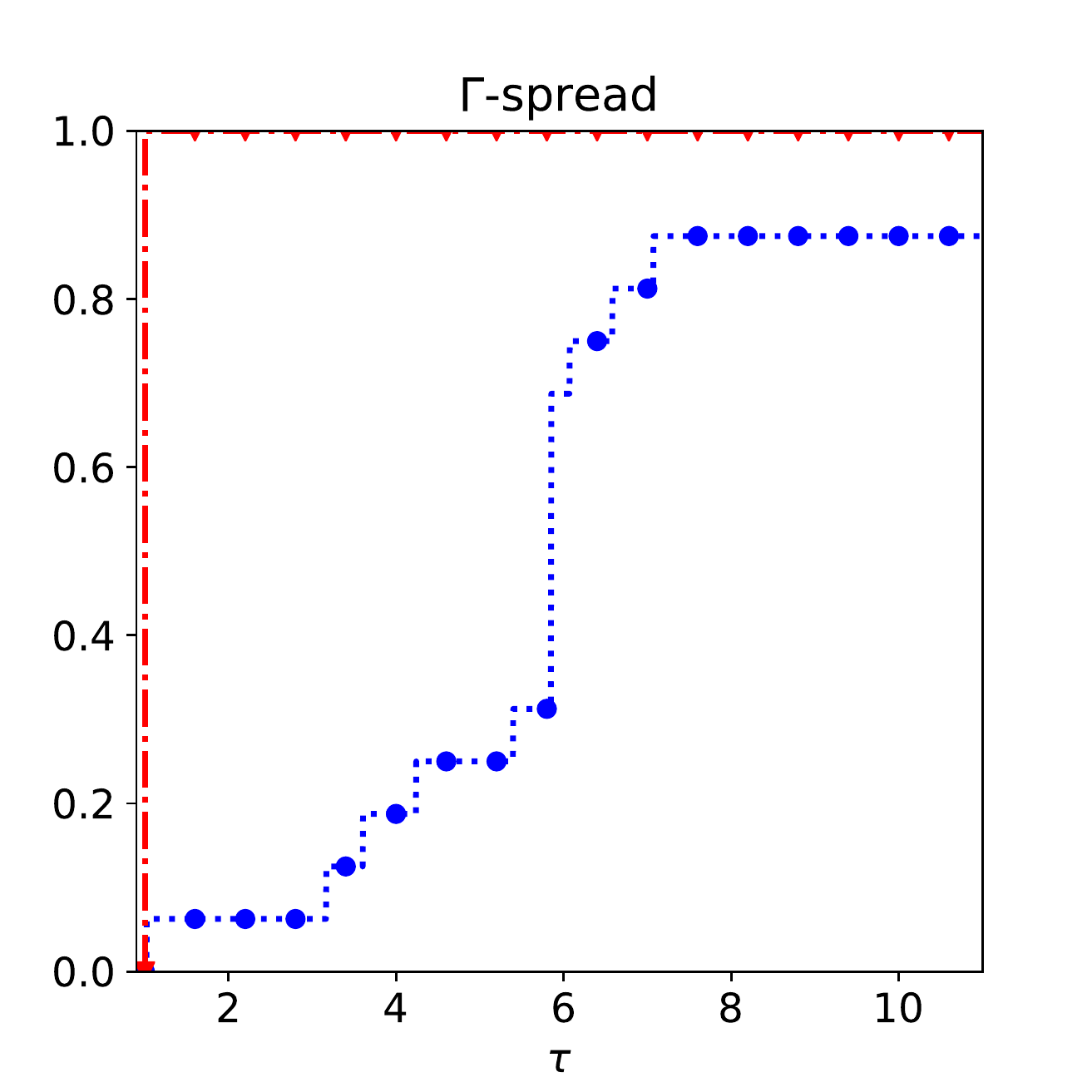}}
	\hfil
	\subfloat[]{\includegraphics[width=1.5in]{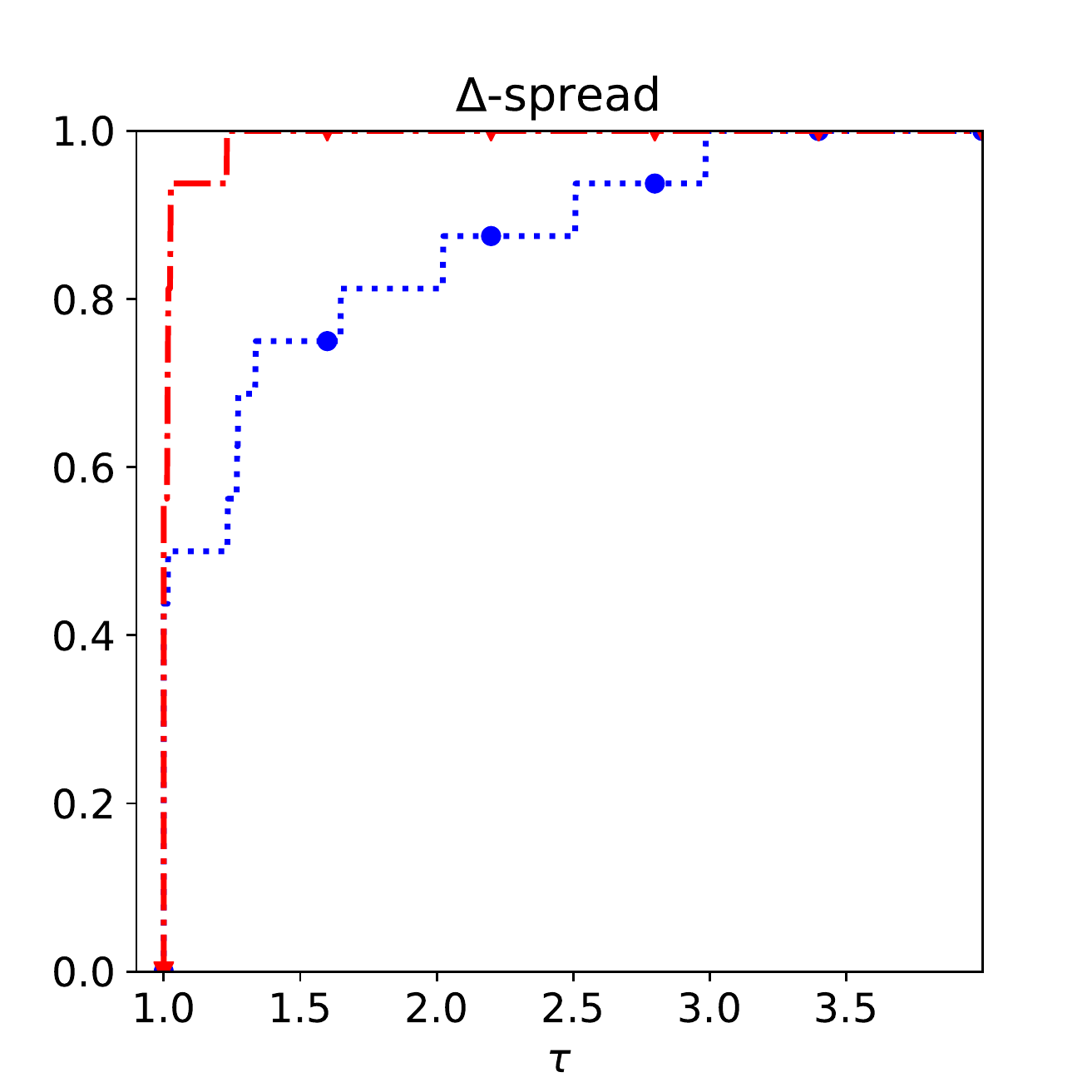}}
	\caption{Performance profiles for \texttt{FPGA} and \texttt{NSGA-II} on the \textit{nonconvex} CEC09\_4 and ZDT\_3 problems (for interpretation of the references to color in text, the reader is referred to the web version of the article). (a) \textit{purity.} (b) $\Gamma$\textit{--spread.} (c) $\Delta$\textit{--spread.}}
	\label{fig::non-conv-prob-perf-prof}
\end{figure*}

\begin{table*}
	\centering
	\renewcommand{\arraystretch}{1.15}
	\caption{Metrics values obtained by \texttt{FPGA} and \texttt{NSGA-II} in the MAN problem with $n = 5, 10, 20, 30, 40, 50, 100, 200$. The values marked in bold are the best obtained in a specific problem. Each of them is related to a specific score.}
	\label{tab::gen-perf-prof-MAN}
	\begin{tabular}{|c||cc||cc||cc|}%
		\hline%
		\multirow{3}{*}{$\texttt{n}$}&\multicolumn{6}{c|}{MAN ($F$ convex)}\\%
		\cline{2%
			-%
			7}%
		&\multicolumn{2}{c||}{\textit{purity}}&\multicolumn{2}{c||}{$\Gamma$\textit{--spread}}&\multicolumn{2}{c|}{$\Delta$\textit{--spread}}\\%
		\cline{2%
			-%
			7}%
		&\texttt{FPGA}&\texttt{NSGA{-}II}&\texttt{FPGA}&\texttt{NSGA{-}II}&\texttt{FPGA}&\texttt{NSGA{-}II}\\%
		\hline%
		\hline%
		5&\textbf{0.984}&0.98&1.839&\textbf{0.809}&1.928&\textbf{0.69}\\%
		\hline%
		10&\textbf{0.993}&0.61&6.241&\textbf{1.377}&1.791&\textbf{0.547}\\%
		\hline%
		20&\textbf{1.0}&0.05&\textbf{16.65}&49.983&1.59&\textbf{0.754}\\%
		\hline%
		30&\textbf{1.0}&0.0&\textbf{2.318}&89.577&1.353&\textbf{0.724}\\%
		\hline%
		40&\textbf{1.0}&0.0&\textbf{36.716}&279.445&1.458&\textbf{0.906}\\%
		\hline%
		50&\textbf{1.0}&0.0&\textbf{5.868}&412.791&1.231&\textbf{0.889}\\%
		\hline%
		100&\textbf{1.0}&0.0&\textbf{21.854}&3894.709&\textbf{0.95}&1.004\\%
		\hline%
		200&\textbf{1.0}&0.0&\textbf{62.971}&9283.624&\textbf{0.824}&0.906\\%
		\hline%
	\end{tabular}%
\end{table*}

\begin{table*}
	\centering
	\renewcommand{\arraystretch}{1.15}
	\caption{Metrics values obtained by \texttt{FPGA} and \texttt{NSGA-II} in the CEC09\_4 problem with $n = 5, 10, 20, 30, 40, 50, 100, 200$. The values marked in bold are the best obtained in a specific problem. Each of them is related to a specific score.}
	\label{tab::gen-perf-prof-CEC}
	\begin{tabular}{|c||cc||cc||cc|}%
		\hline%
		\multirow{3}{*}{$\texttt{n}$}&\multicolumn{6}{c|}{CEC09\_4 ($F$ non-convex)}\\%
		\cline{2%
			-%
			7}%
		&\multicolumn{2}{c||}{\textit{purity}}&\multicolumn{2}{c||}{$\Gamma$\textit{--spread}}&\multicolumn{2}{c|}{$\Delta$\textit{--spread}}\\%
		\cline{2%
			-%
			7}%
		&\texttt{FPGA}&\texttt{NSGA{-}II}&\texttt{FPGA}&\texttt{NSGA{-}II}&\texttt{FPGA}&\texttt{NSGA{-}II}\\%
		\hline%
		\hline%
		5&0.0&\textbf{1.0}&0.419&\textbf{0.132}&\textbf{0.579}&0.713\\%
		\hline%
		10&0.154&\textbf{0.99}&0.607&\textbf{0.041}&1.109&\textbf{0.549}\\%
		\hline%
		20&0.714&\textbf{0.98}&0.475&\textbf{0.037}&0.68&\textbf{0.537}\\%
		\hline%
		30&0.429&\textbf{0.93}&0.553&\textbf{0.078}&0.678&\textbf{0.55}\\%
		\hline%
		40&0.2&\textbf{0.97}&0.501&\textbf{0.093}&0.786&\textbf{0.618}\\%
		\hline%
		50&0.025&\textbf{0.95}&0.544&\textbf{0.128}&1.845&\textbf{0.618}\\%
		\hline%
		100&\textbf{0.965}&0.89&0.514&\textbf{0.143}&1.757&\textbf{0.701}\\%
		\hline%
		200&\textbf{0.915}&0.81&0.483&\textbf{0.474}&1.673&\textbf{1.016}\\%
		\hline%
	\end{tabular}%
\end{table*}

We show the performance profiles for the two algorithms on the convex problems in Figure \ref{fig::conv-prob-perf-prof} and on the nonconvex problems in Figure \ref{fig::non-conv-prob-perf-prof}.

We can observe that in the former case the \texttt{FPGA} turned out to be better than \texttt{NSGA-II} in terms of \textit{purity}. This result reasonably comes from the fact that, in problems characterized by convex objective functions, the use of first-order information and common descent directions lets \texttt{FPGA} find better solutions than \texttt{NSGA-II} for equal computational budget. 
On the contrary, the \texttt{FPGA} algorithm was outperformed by \texttt{NSGA-II} in terms of $\Gamma$\textit{--spread} and $\Delta$\textit{--spread}. In this perspective, the \texttt{crossover} and \texttt{mutation} operations of \texttt{NSGA-II} allow to consistently obtain spread Pareto front approximations, while the constrained steepest partial descent directions and the \texttt{B-FALS} employed by the \texttt{FPGA} are apparently not as effective.

As for the nonconvex case, we can observe from the \textit{purity} profile that now the \texttt{FPGA} obtained many points that are dominated by those produced by \texttt{NSGA-II}. The results with the \textit{spread} metrics are instead analogous to the convex case, with \texttt{NSGA-II} outperforming \texttt{FPGA}. However, the performance gap in terms of $\Gamma$\textit{--spread} is even larger, while it is less marked for the $\Delta$\textit{--spread}.

In order to assess the performance of the algorithms as the problem dimension $n$ increases, in Tables~\ref{tab::gen-perf-prof-MAN} and \ref{tab::gen-perf-prof-CEC} we show in detail the metrics values achieved by the two methods on a convex (MAN) and a nonconvex (CEC09\_4) problems. Again, the table shows the overall strength of \texttt{NSGA-II} w.r.t.\ the $\Delta$\textit{--spread} metric. 

As for the $\Gamma$\textit{--spread}, in the MAN problem we can observe the great results achieved by \texttt{FPGA}: it outperformed \texttt{NSGA-II} considering values of $n$ equal to or greater than 20. In these cases, the constrained steepest partial descent directions and the \texttt{B-FALS} algorithm turned out to be helpful in exploring the extreme regions of the objectives space and, then, in finding a spread approximation of the Pareto front. In the CEC09\_4 problem, it is the opposite: the genetic algorithm managed to obtain the best $\Gamma$\textit{--spread} values. 

The \textit{purity} values indicate another relevant feature of the two algorithms. In the nonconvex case, \texttt{FPGA} turned out not to be capable of obtaining better points than \texttt{NSGA-II} for low values of $n$. However, as the value of $n$ increased, the situation gradually changed and \texttt{FPGA} finally obtained better \textit{purity} values w.r.t.\ its competitor on the largest problems. These results remark one of the drawbacks of the EAs, i.e., the limited scalability. In this case, common descent directions can be very helpful for cheaply improving the quality of the solutions. 

In conclusion, both algorithms have features that make them very effective in specific situations: \texttt{FPGA} was better in convex and/or high dimensional problems, while \texttt{NSGA-II} was more effective in non-convex low dimensional ones. Furthermore, the genetic features of \texttt{NSGA-II} let this latter one perform better in finding spread and uniform Pareto fronts most of the times: this is also reflected in the \textit{spread} metrics values obtained by \texttt{NSGA-II}. All these facts remark once again how much trying to join these benefits in one algorithm might be appealing.

\subsection{Preliminary comparisons between \texttt{NSMA} and the state-of-the-art algorithms}
\label{subsec::first-impressions}

In this section, we provide the results on two problems along with some first comments about the behavior of the four algorithms. We analyzed the CEC09\_3 problem with $n = 10$ and the ZDT\_3 problem with $n = 20$. The first one has particularly difficult objective functions, while the second one is also characterized by a composite function and a disconnected front which is not convex everywhere. We consider these problems suitable to start an analysis about the performance of the considered algorithms.

\begin{figure*}[!t]
	\centering
	\subfloat[]{\includegraphics[width=2.25in]{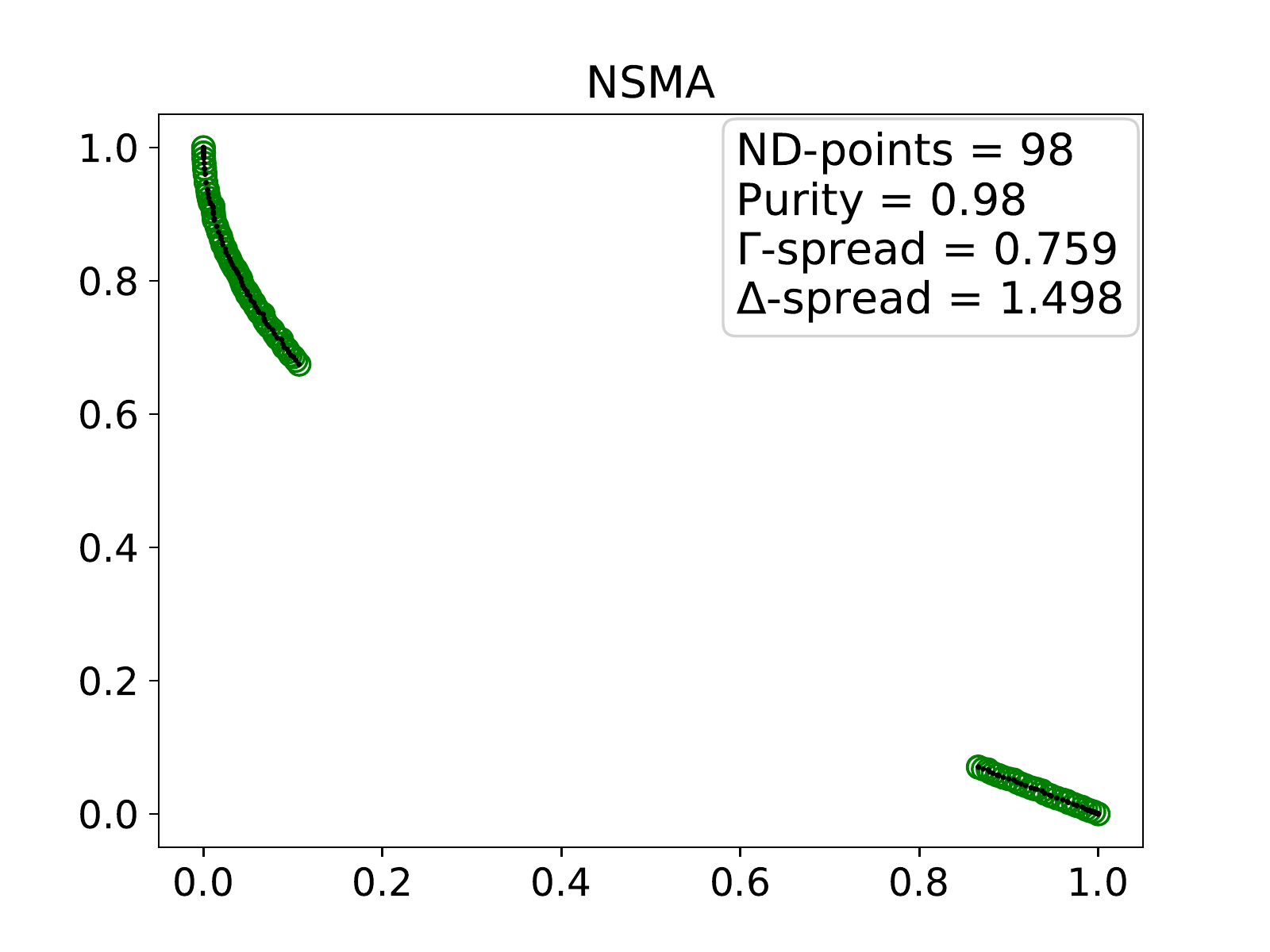}}
	\hfil
	\subfloat[]{\includegraphics[width=2.25in]{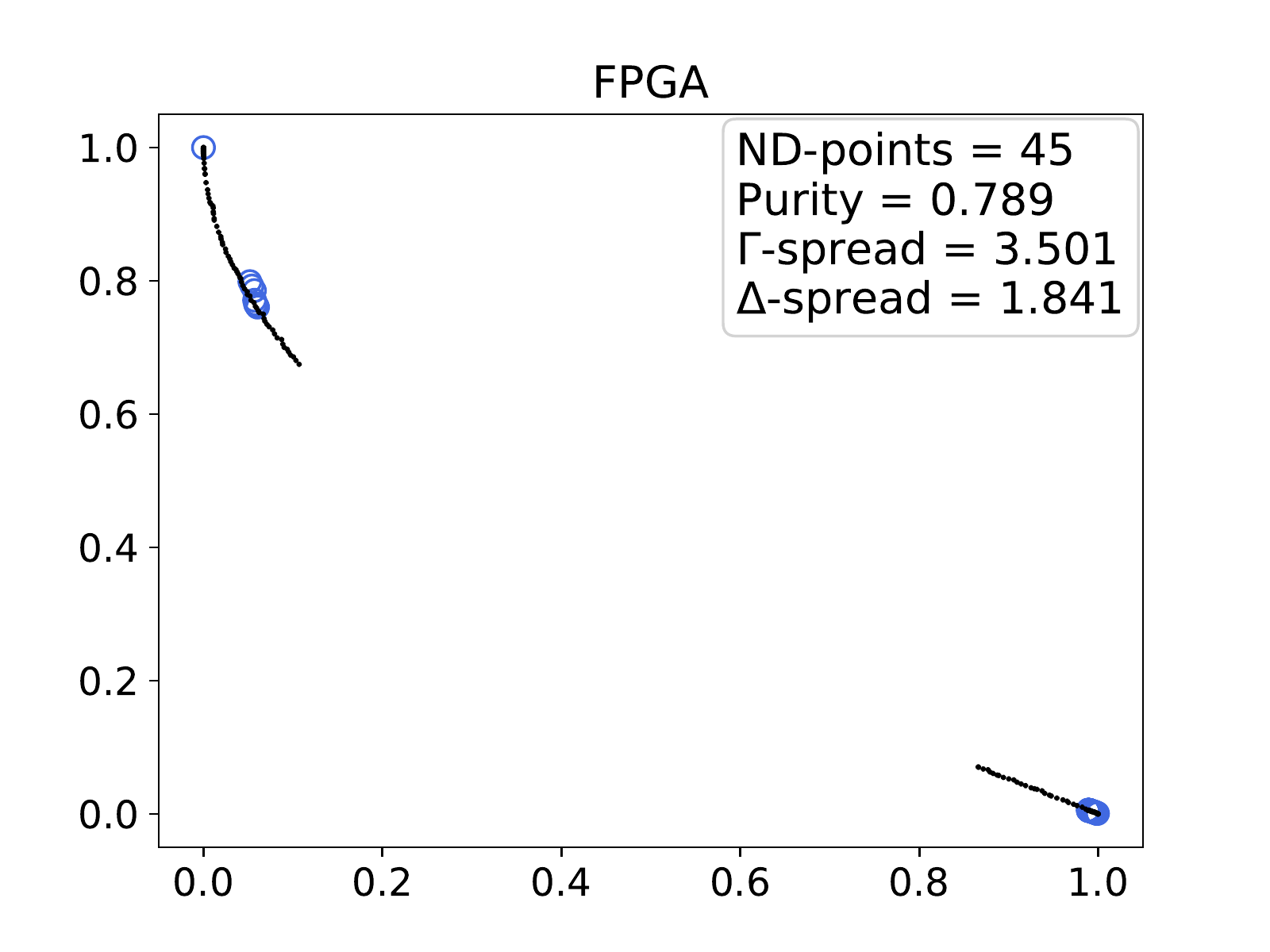}}
	\\
	\subfloat[]{\includegraphics[width=2.25in]{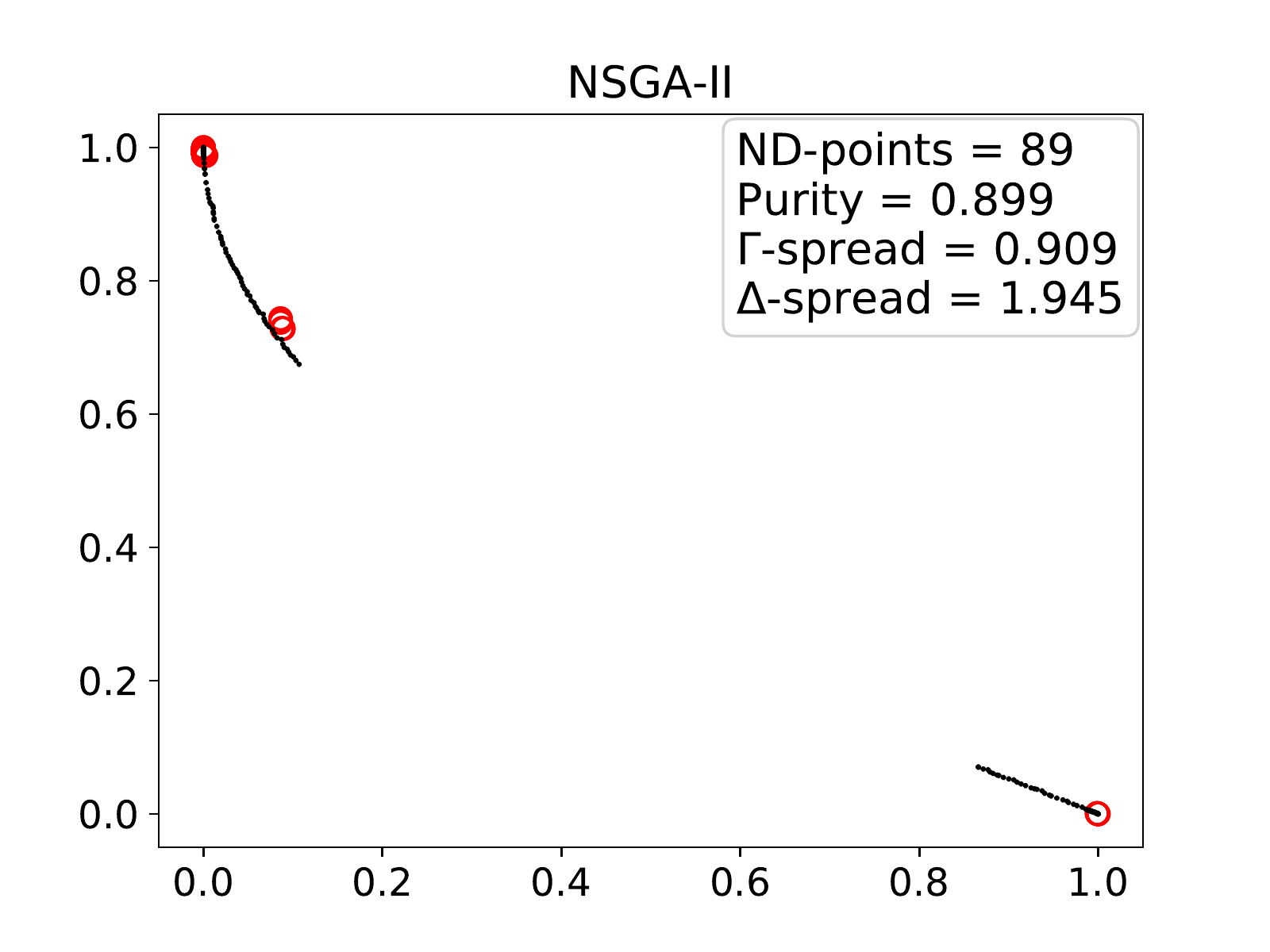}}
	\hfil
	\subfloat[]{\includegraphics[width=2.25in]{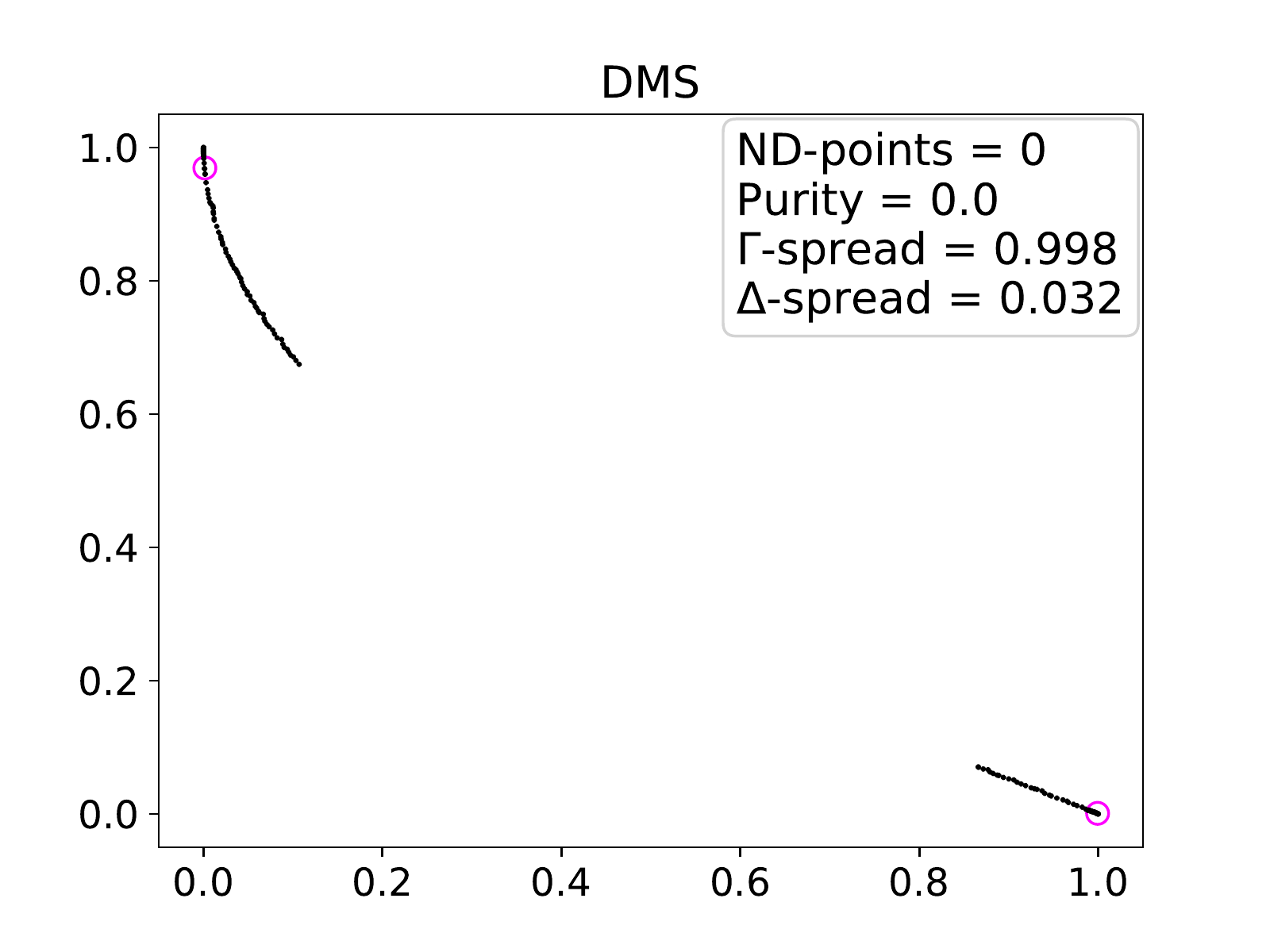}}
	\caption{Approximation of the Pareto front of the CEC09\_3 problem with $n = 10$ (for interpretation of the references to color in text, the reader is referred to the electronic version of the article). (a) \texttt{NSMA}. (b) \texttt{FPGA}. (c) \texttt{NSGA-II}. (d) \texttt{DMS}.}
	\label{fig::SingleProblem - CEC09_3}
\end{figure*}

From the results on the CEC09\_3 problem, shown in Figure \ref{fig::SingleProblem - CEC09_3}, we immediately observe the effectiveness of our approach. Indeed, \texttt{NSMA} outperformed the other algorithms in terms of \textit{ND-points}, \textit{purity} and $\Gamma$\textit{--spread}.

\texttt{NSGA-II} and \texttt{FPGA} turned out to be the second and the third best algorithms, respectively, with \texttt{FPGA} outperforming the genetic method only in terms of $\Delta$\textit{--spread}. Note that \texttt{FPGA} achieved a high value for the $\Gamma$\textit{--spread} metric since it produced a suboptimal point that is dominated and far from the reference front. This point is not shown in the figure for graphical reasons. 

\texttt{NSGA-II} and \texttt{FPGA} seem not to be capable of spreading the search in the objectives space. Indeed, they retrieved many points but most of them are concentrated in a small portion of the objectives space. In this regard, \texttt{NSMA} was better: this result arguably comes from the use of constrained steepest partial descent directions with points characterized by a high crowding distance. Indeed, using descent steps at such points lets \texttt{NSMA} obtain a more spread and uniform Pareto front approximation w.r.t.\ its competitors.

\begin{figure*}[!t]
	\centering
	\subfloat[]{\includegraphics[width=2.25in]{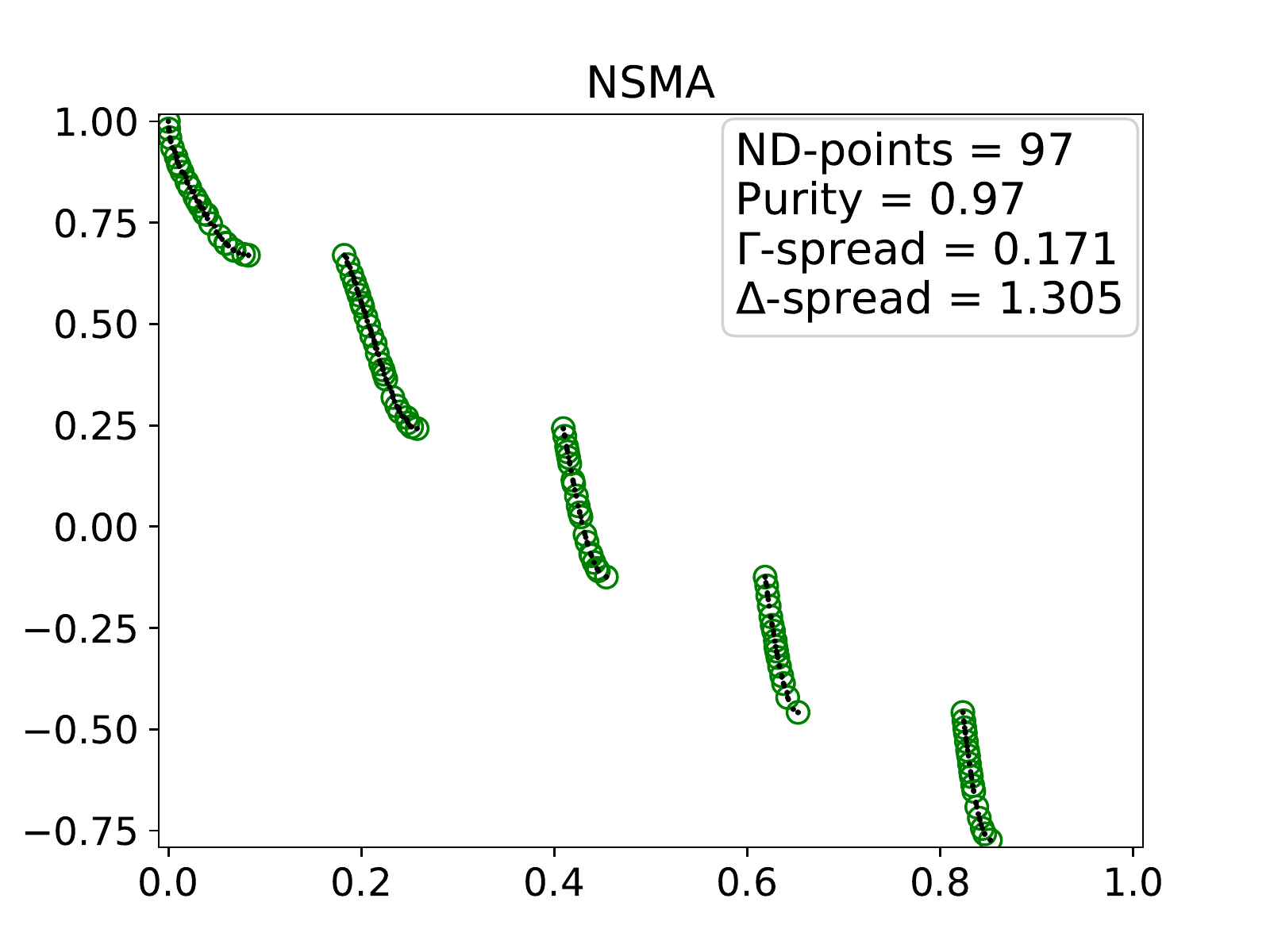}}
	\hfil
	\subfloat[]{\includegraphics[width=2.25in]{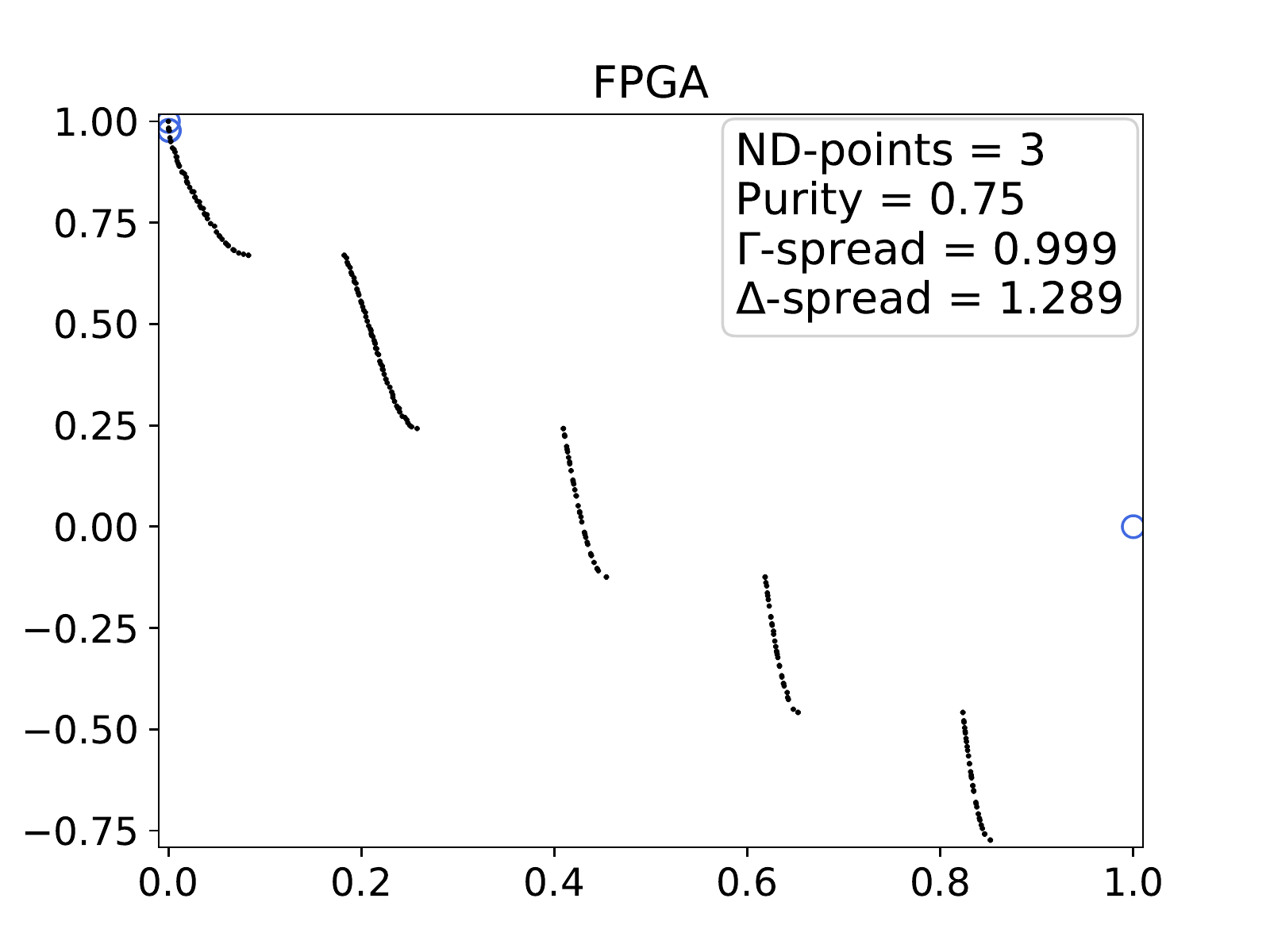}}
	\\
	\subfloat[]{\includegraphics[width=2.25in]{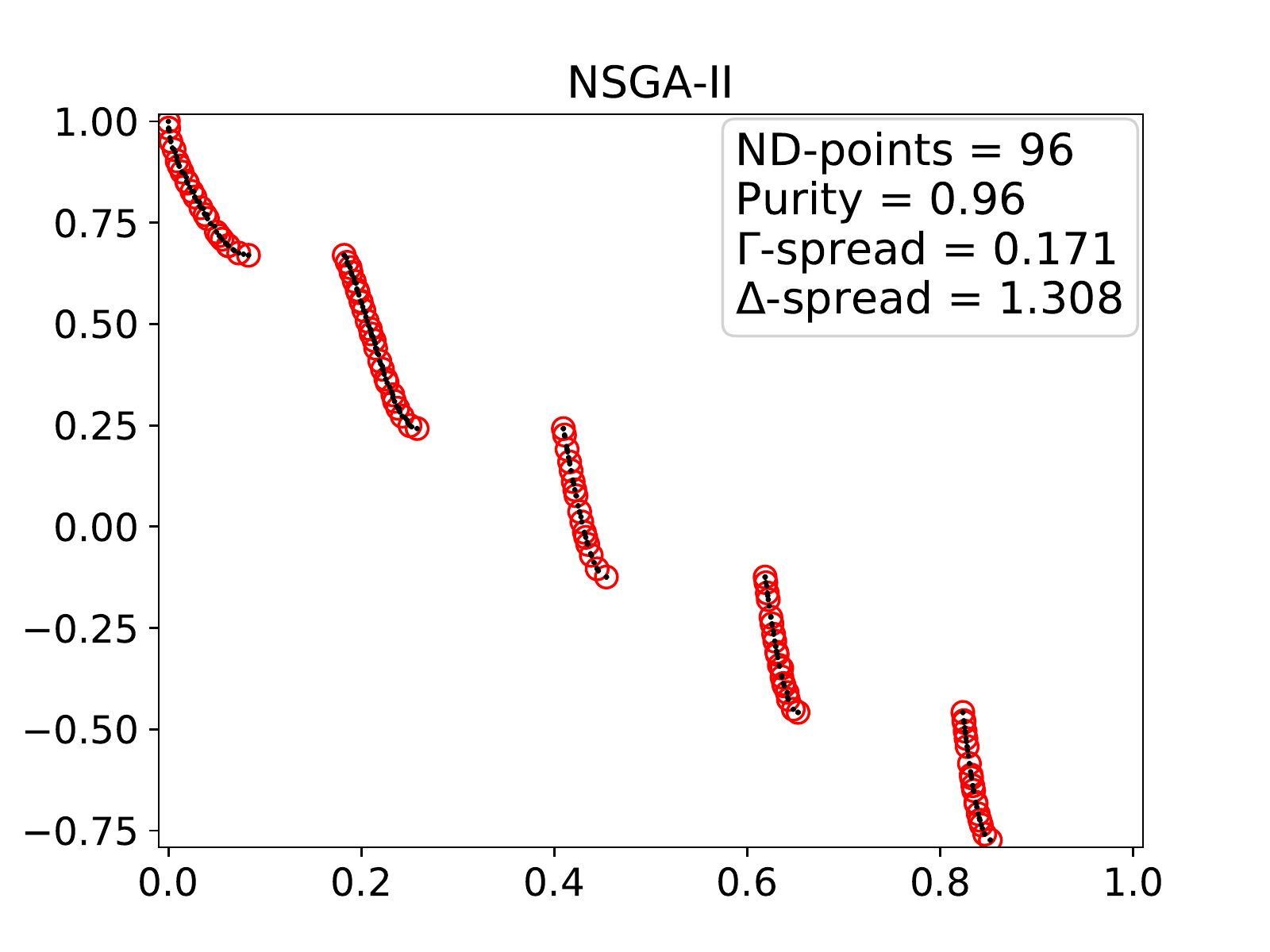}}
	\hfil
	\subfloat[]{\includegraphics[width=2.25in]{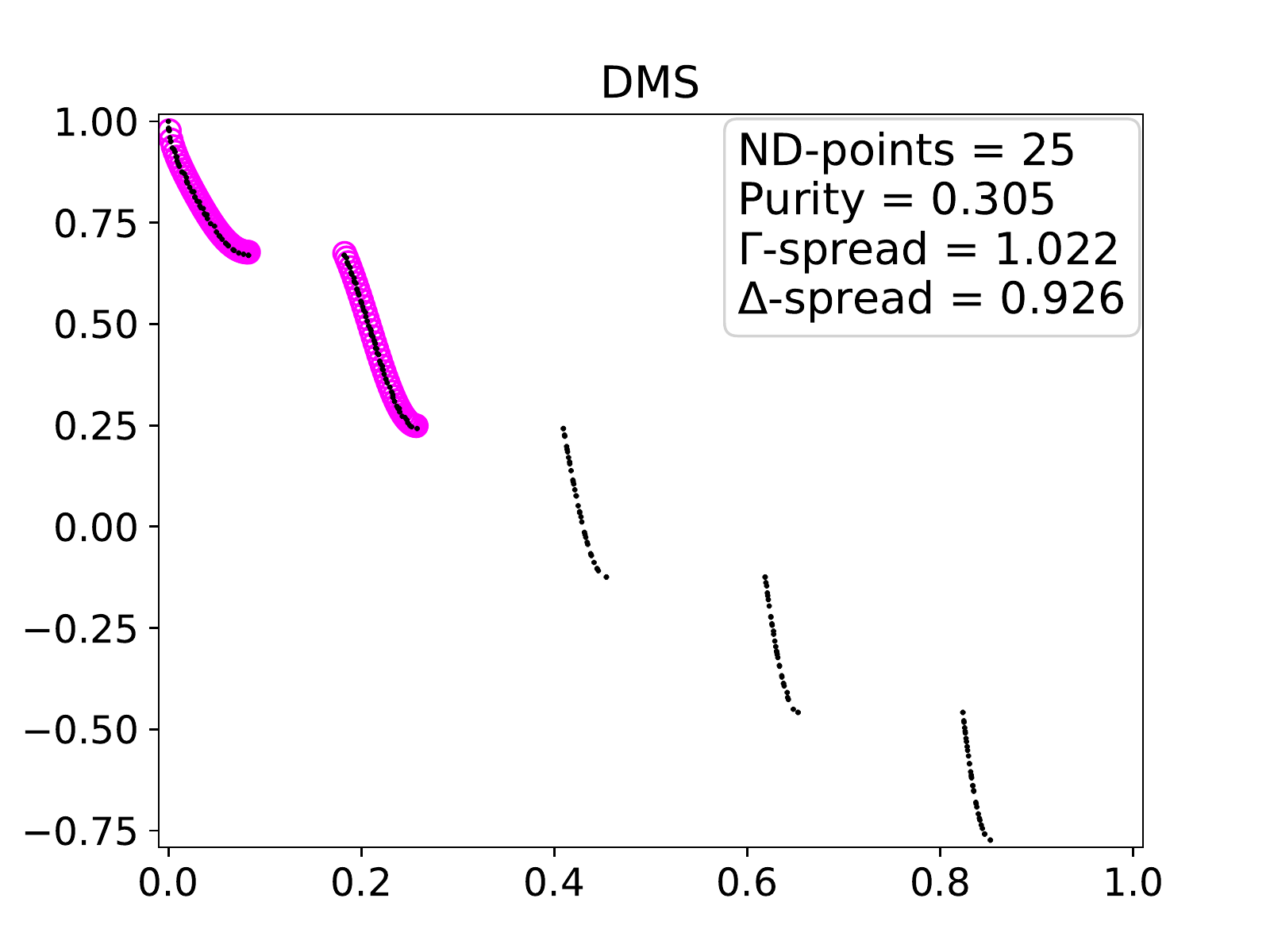}}
	\caption{Approximation of the Pareto front of the ZDT\_3 problem with $n = 20$ (for interpretation of the references to color in text, the reader is referred to the electronic version of the article). (a) \texttt{NSMA}. (b) \texttt{FPGA}. (c) \texttt{NSGA-II}. (d) \texttt{DMS}.}
	\label{fig::SingleProblem - ZDT_3}
\end{figure*}

\texttt{NSMA} and \texttt{NSGA-II} turned out to be the best algorithms on the ZDT\_3 problem, as we can observe in Figure \ref{fig::SingleProblem - ZDT_3}. Furthermore, they exhibited very similar performance. It is known that \texttt{NSGA-II} is one of the most effective algorithms to use with the ZDT problem class. Indeed, its genetic features allow it to escape from non-optimal Pareto-stationary solutions and to obtain good results with the most complex functions. \texttt{NSMA} seems to use these features as efficiently as \texttt{NSGA-II}. We also observe a little performance enhancement in terms of \textit{ND-points} and \textit{purity}. 

The lack of these characteristics did not allow \texttt{FPGA} to have the same performance. Indeed, although this algorithm obtained a good value for the \textit{purity} metric, it produced few points and it was not capable to obtain a spread and uniform Pareto front. \texttt{DMS} seems not to have the same issues, having been able to properly identify two blocks of the disconnected front. However, it performed worse than \texttt{FPGA} in terms of \textit{purity}.

Finally, we note that \texttt{NSMA} was better than all its competitors in terms of \textit{ND-points}. It was not obvious a priori to obtain such results, since, as opposed to \texttt{FPGA} and \texttt{DMS}, \texttt{NSMA} considers a fixed number of solutions in the population.   

\subsection{Performance analysis in variable settings}
\label{subsec::convexity-and-scalability}

In this section, we want to assess the robustness of the proposed algorithm in the specific settings where, as highlighted in Section \ref{subsec::NSGA-II-vs-FPGA}, genetic and descent methods exhibit particular struggles. In detail, we compare the performance of the four algorithms (\texttt{NSMA, FPGA, NSGA-II} and \texttt{DMS}) in two peculiar problems already addressed in Section \ref{subsec::NSGA-II-vs-FPGA}: MAN ($F$ convex) and CEC09\_4 ($F$ nonconvex). Moreover, we consider the following problem dimensionalities: $n = 5, 20, 50, 100$.

\begin{figure*}[!t]
	\centering
	\subfloat[]{\includegraphics[width=2.25in]{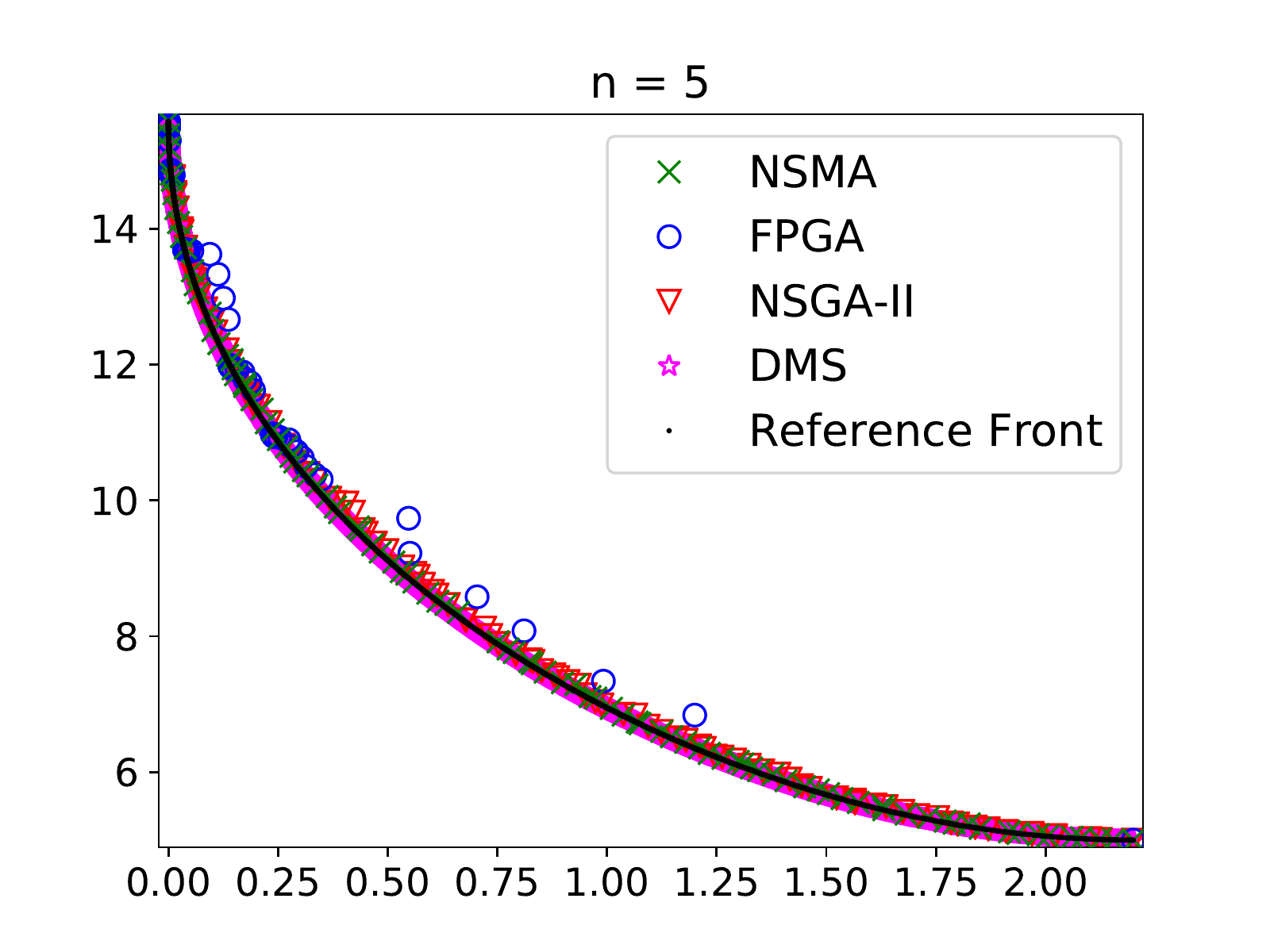}}
	\hfil
	\subfloat[]{\includegraphics[width=2.25in]{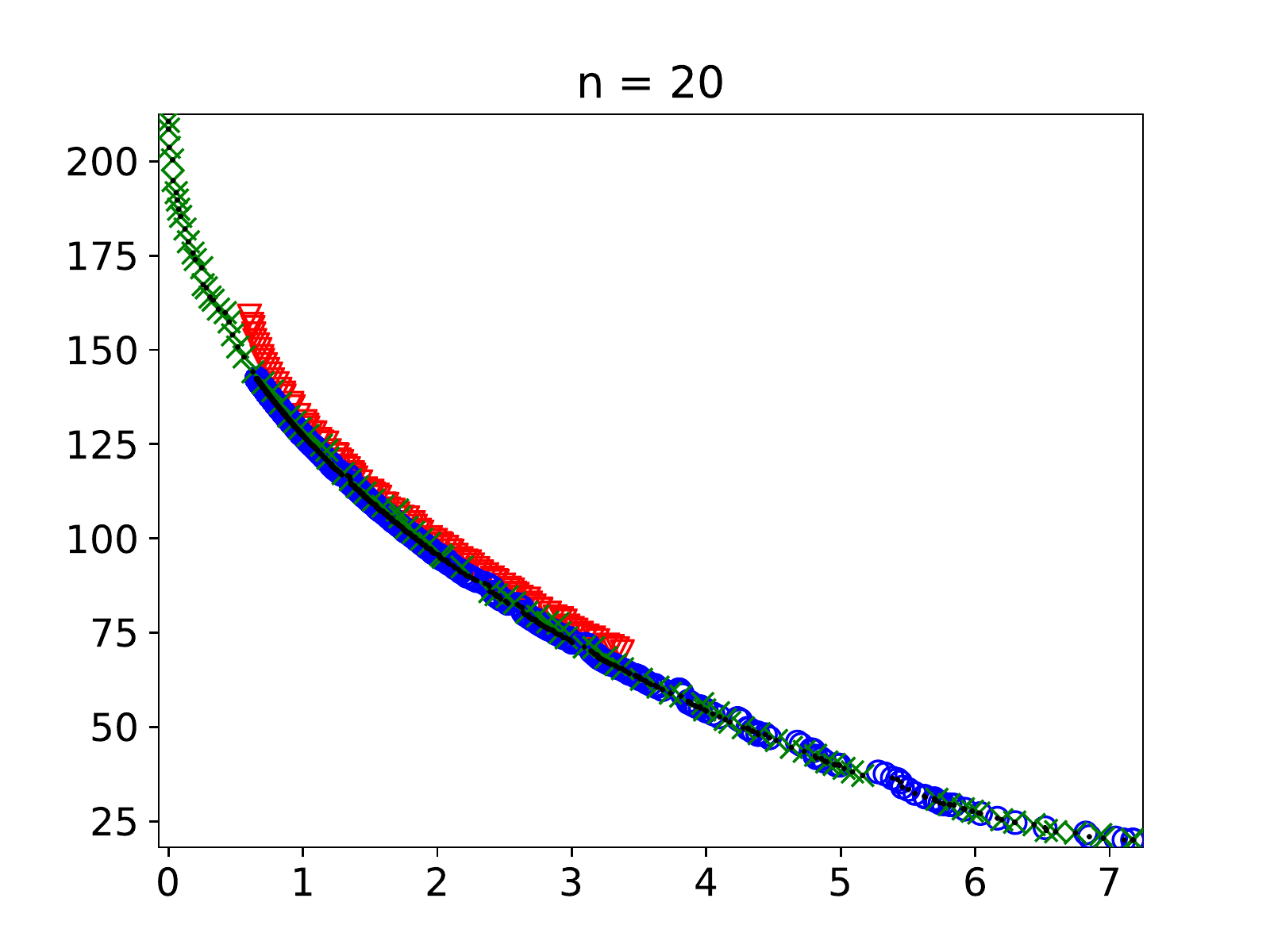}}
	\\
	\subfloat[]{\includegraphics[width=2.25in]{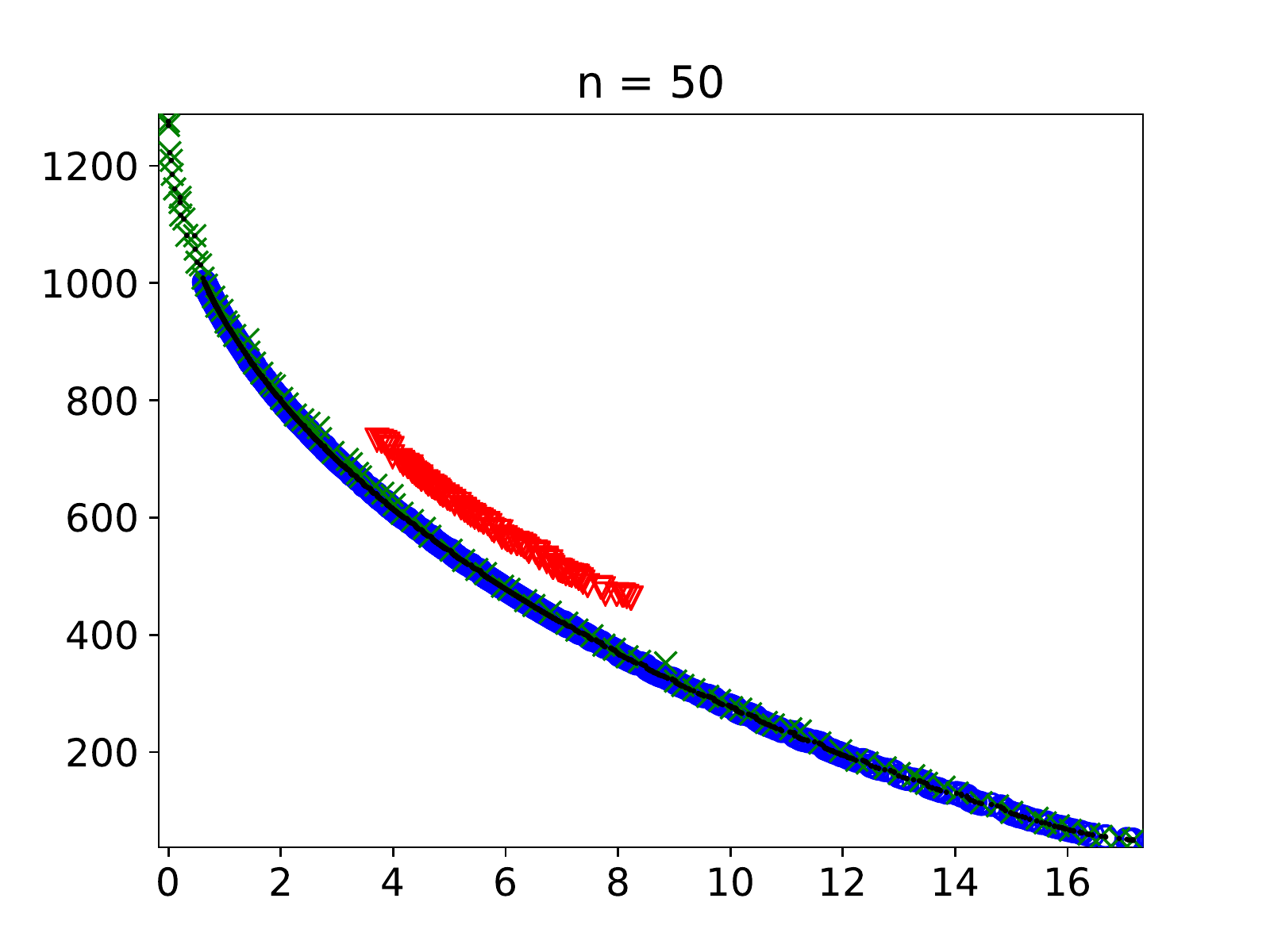}}
	\hfil
	\subfloat[]{\includegraphics[width=2.25in]{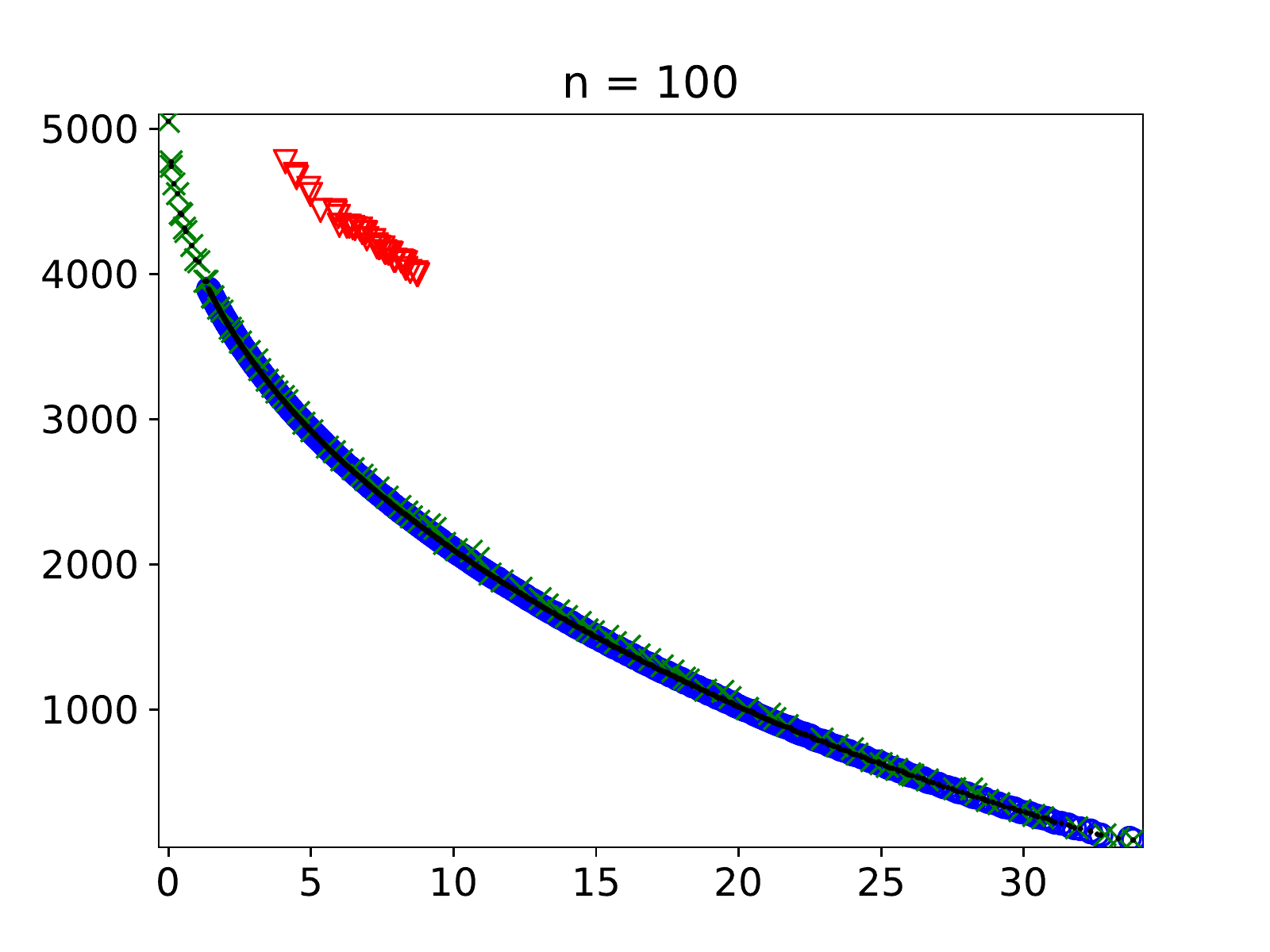}}
	\caption{Approximation of the Pareto front of the convex MAN problem at different dimensionalities, retrieved by \texttt{NSMA, FPGA, NSGA-II} and \texttt{DMS} (for interpretation of the references to color in text, the reader is referred to the electronic version of the article). (a) $n = 5$. (b) $n = 20$. (c) $n = 50$. (d) $n = 100$.}
	\label{fig::All - MAN}
\end{figure*}

\begin{table}
	\centering
	\renewcommand{\arraystretch}{1.15}
	\caption{Metrics values achieved by the four algorithms (\texttt{NSMA, FPGA, NSGA-II} and \texttt{DMS}) on the convex MAN problem for $n = 5, 20, 50, 100$. The values marked in bold are the best obtained on a specific problem.}
	\label{tab::All - MAN}
	\begin{tabular}{|c||c||c|c|c|c|}%
		\hline%
		$\texttt{n}$&\texttt{METRIC}&\texttt{NSMA}&\texttt{FPGA}&\texttt{NSGA{-}II}&\texttt{DMS}\\%
		\hline%
		\hline%
		\multirow{4}{*}{5}&\textit{ND{-}points}&9&1107&2&\textbf{5291}\\%
		&\textit{purity}&0.09&0.954&0.02&\textbf{0.994}\\%
		&$\Gamma$\textit{--spread}&0.433&1.839&0.81&\textbf{0.006}\\%
		&$\Delta$\textit{--spread}&\textbf{0.597}&1.928&0.69&1.003\\%
		\hline
		\multirow{4}{*}{20}&\textit{ND{-}points}&59&\textbf{2649}&0&0\\%
		&\textit{purity}&0.59&\textbf{0.99}&0.0&0.0\\%
		&$\Gamma$\textit{--spread}&\textbf{6.535}&68.252&51.601&9953.2\\%
		&$\Delta$\textit{--spread}&\textbf{0.558}&1.541&0.775&N/A\\%
		\hline
		\multirow{4}{*}{50}&\textit{ND{-}points}&30&\textbf{2776}&0&0\\%
		&\textit{purity}&0.3&\textbf{0.998}&0.0&0.0\\%
		&$\Gamma$\textit{--spread}&\textbf{46.41}&273.518&543.112&9790.081\\%
		&$\Delta$\textit{--spread}&\textbf{0.509}&1.222&0.893&N/A\\%
		\hline
		\multirow{4}{*}{100}&\textit{ND{-}points}&22&\textbf{2258}&0&0\\%
		&\textit{purity}&0.22&\textbf{0.999}&0.0&0.0\\%
		&$\Gamma$\textit{--spread}&\textbf{278.269}&1154.101&3894.709&9878.008\\%
		&$\Delta$\textit{--spread}&\textbf{0.521}&0.952&1.005&N/A\\%
		\hline%
	\end{tabular}%
\end{table}

The results for the MAN problem are shown in Figure \ref{fig::All - MAN} and Table~\ref{tab::All - MAN}. For $n = 5$, \texttt{DMS} turned out to be the best algorithm in all the metrics except for the $\Delta$\textit{--spread}. Only the \texttt{FPGA} algorithm obtained a similar \textit{purity}. However, observing the plot, the points produced by the \texttt{NSMA} algorithm seem to be near to those obtained by \texttt{DMS} and \texttt{FPGA}. We hence deduce that the latter algorithms produced only slightly better points. 

Furthermore, in this problem \texttt{NSMA} outperformed the competitors in terms of $\Delta$\textit{--spread}. Indeed, our method managed to achieve an uniform Pareto front, as opposed to \texttt{FPGA} that produced most of the points in restricted areas of the objectives space. 

As the value of $n$ increases, the \texttt{DMS} performance gets worse and \texttt{NSMA} outperforms it w.r.t.\ all the metrics. In particular, in these cases our method turned out to be the best in terms of the \textit{spread} metrics. For large values of $n$, \texttt{DMS} produced only a single point that is also dominated (it is not observable in the figure since it is too far from the reference front). The $\Delta$\textit{--spread} metric is not available for \texttt{DMS} in these cases, since it requires at least two points to be returned. 

The performance of \texttt{NSGA-II} is rather poor, regardless the value of $n$. Arguably, this result can be attributed to the aforementioned \texttt{NSGA-II} performance slowdown occurring on problems characterized by a particularly large feasible sets (Section \ref{subsubsec::getsurrogatebounds}). Furthermore, as also commented in Section \ref{subsec::NSGA-II-vs-FPGA}, in the MAN problem \texttt{NSGA-II} struggles to explore the extreme regions of the objectives space. In this context, \texttt{NSMA} particularly exploited the surrogate bounds, the constrained steepest descent directions and the optimization of the points with high crowding distance. The constrained steepest descent directions also allowed \texttt{FPGA} to be the best algorithm in terms of \textit{purity} overall. However, this method poorly performed regarding the \textit{spread} metrics. Finally, for great values of $n$, our approach and \texttt{FPGA} are the only algorithms whose \textit{purity} values are not equal to 0. Only \texttt{NSMA} managed to obtain points near to the ones of \texttt{FPGA}.

\begin{figure*}[!t]
	\centering
	\subfloat[]{\includegraphics[width=2.25in]{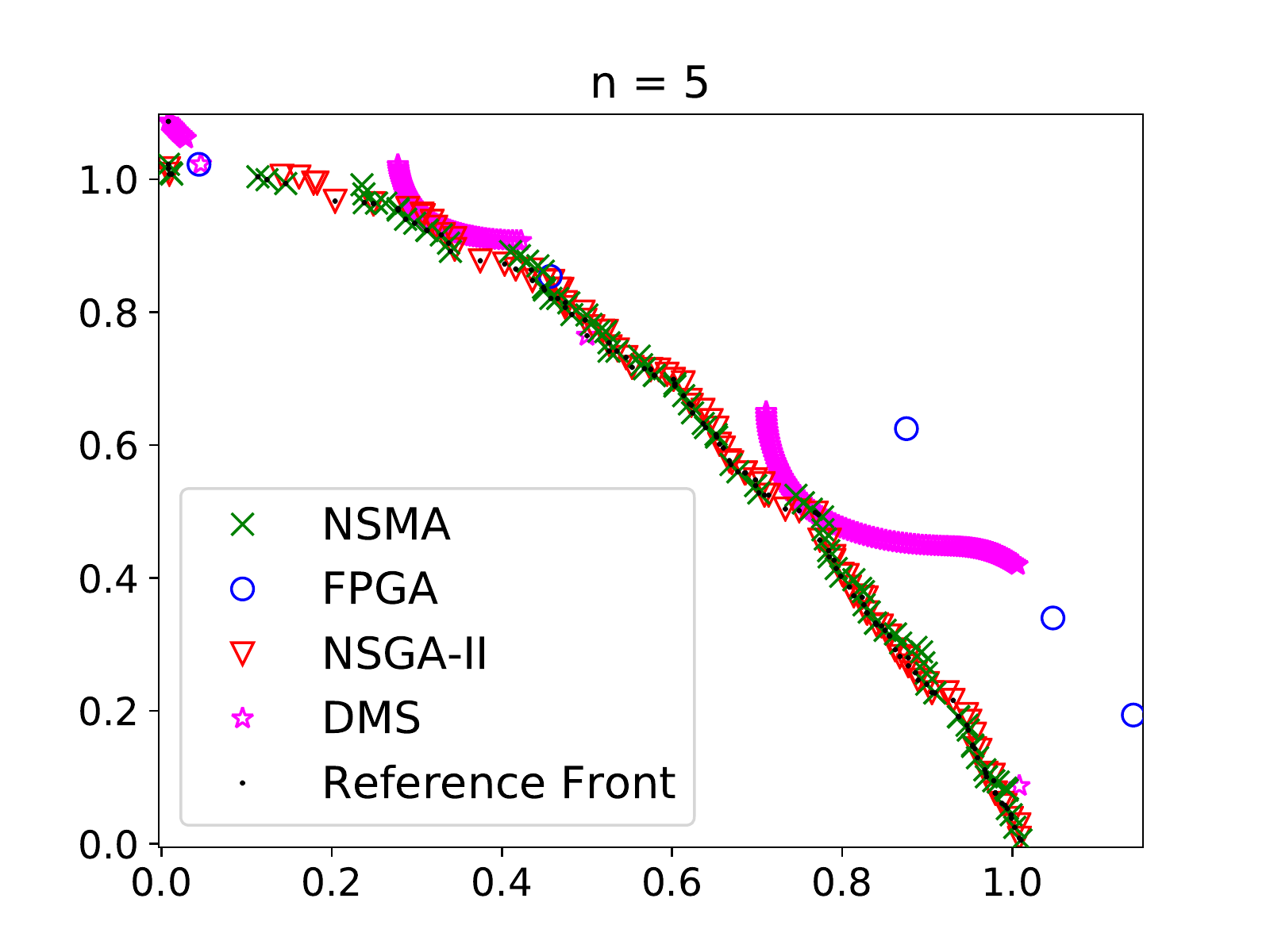}}
	\hfil
	\subfloat[]{\includegraphics[width=2.25in]{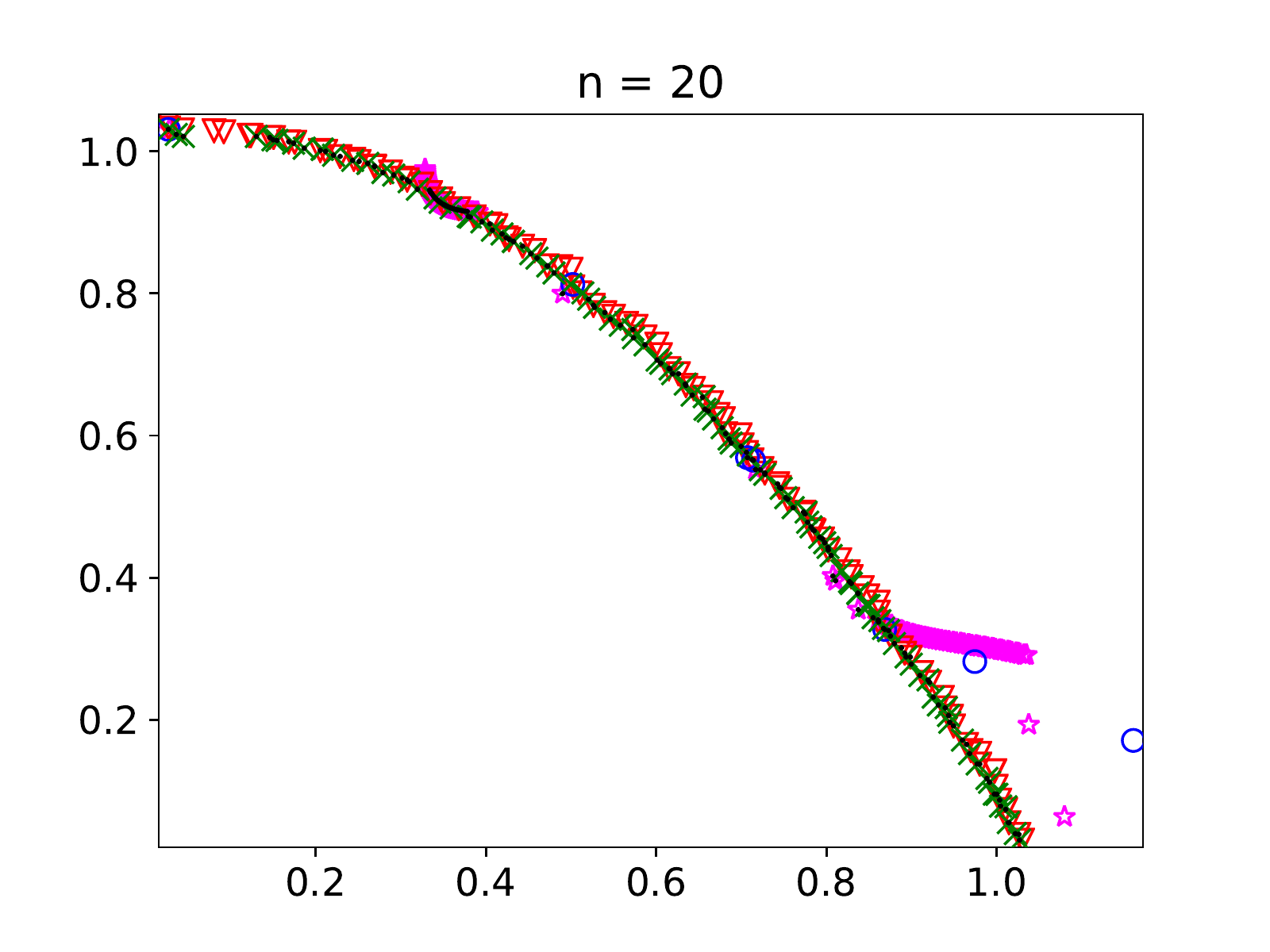}}
	\\
	\subfloat[]{\includegraphics[width=2.25in]{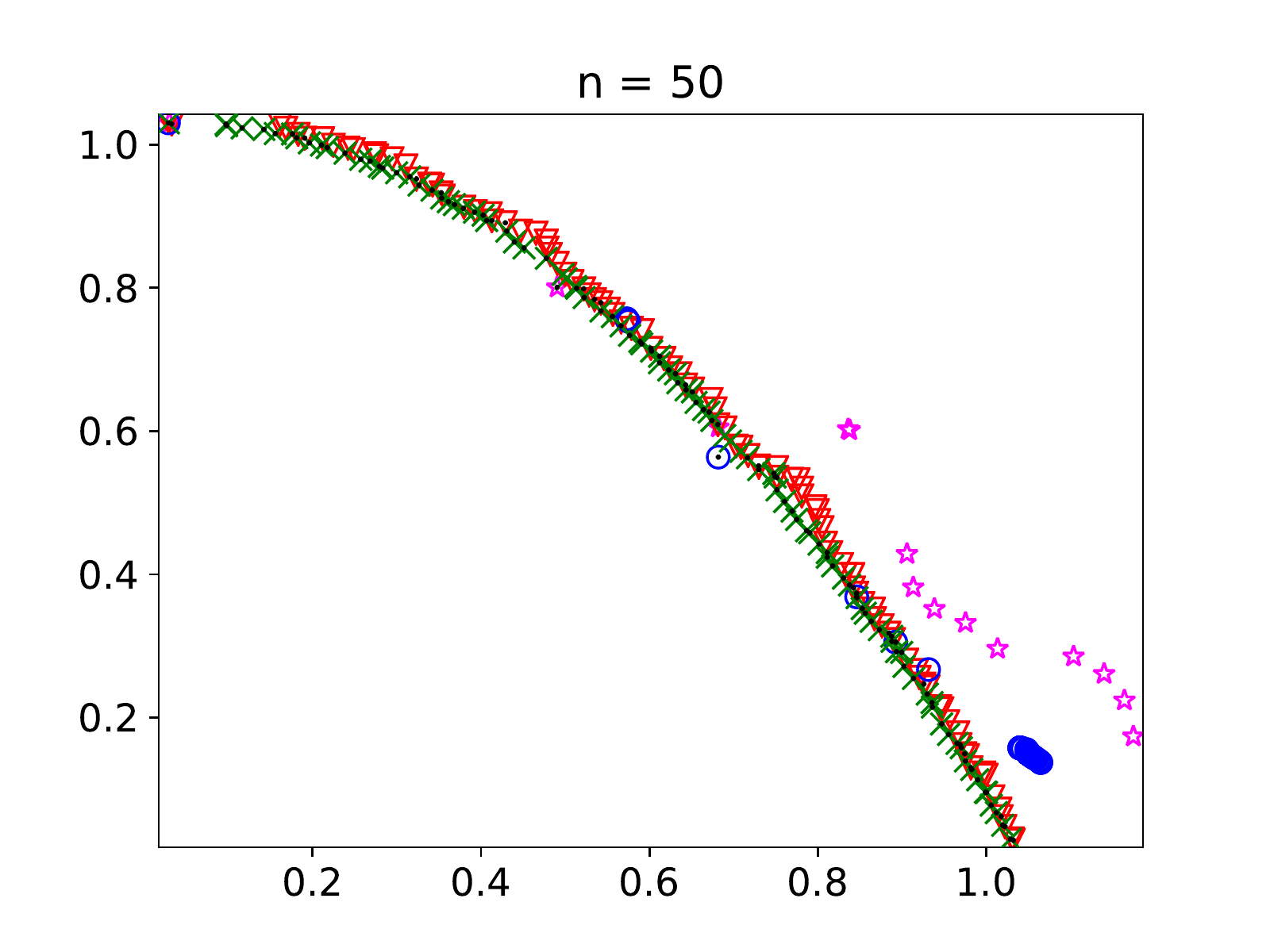}}
	\hfil
	\subfloat[]{\includegraphics[width=2.25in]{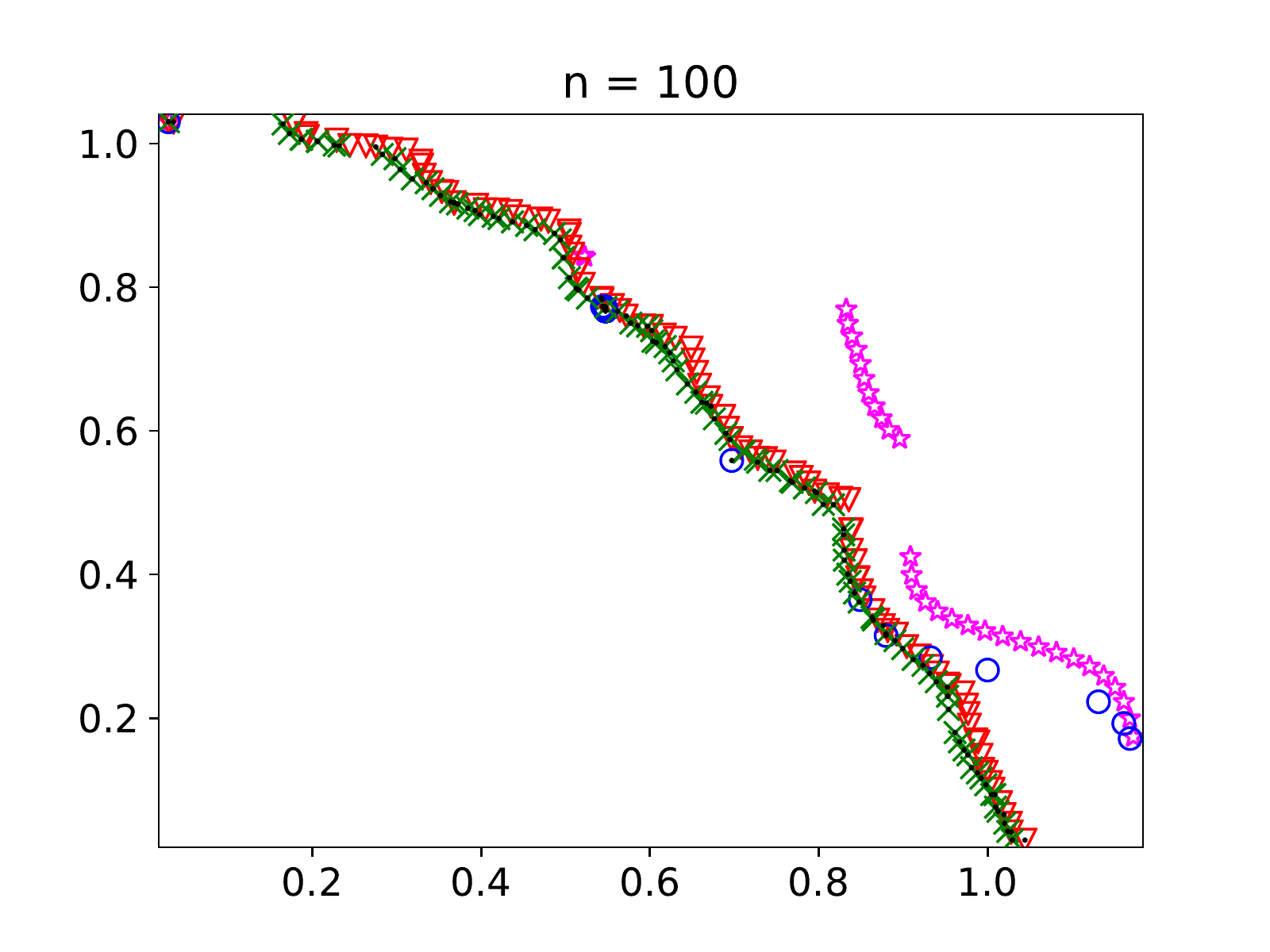}}
	\caption{Approximation of the Pareto front of the nonconvex CEC09\_4 problem at different dimensionalities, retrieved by \texttt{NSMA, FPGA, NSGA-II} and \texttt{DMS} (for interpretation of the references to color in text, the reader is referred to the electronic version of the article). (a) $n = 5$. (b) $n = 20$. (c) $n = 50$. (d) $n = 100$.}
	\label{fig::All - CEC09_4}
\end{figure*}

\begin{table}
	\centering
	\renewcommand{\arraystretch}{1.15}
	\caption{Metrics values achieved by the four algorithms (\texttt{NSMA, FPGA, NSGA-II} and \texttt{DMS}) on the nonconvex CEC09\_4 problem for $n = 5, 20, 50, 100$. The values marked in bold are the best obtained on a specific problem.}
	\label{tab::All - CEC09_4}
	\begin{tabular}{|c||c||c|c|c|c|}%
		\hline%
		$\texttt{n}$&\texttt{METRIC}&\texttt{NSMA}&\texttt{FPGA}&\texttt{NSGA{-}II}&\texttt{DMS}\\%
		\hline%
		\hline%
		\multirow{4}{*}{5}&\textit{ND{-}points}&\textbf{64}&0&56&4\\%
		&\textit{purity}&\textbf{0.64}&0.0&0.56&0.024\\%
		&$\Gamma$\textit{--spread}&\textbf{0.101}&0.419&0.132&0.333\\%
		&$\Delta$\textit{--spread}&0.801&\textbf{0.579}&0.714&1.193\\%
		\hline
		\multirow{4}{*}{20}&\textit{ND{-}points}&\textbf{86}&4&49&30\\%
		&\textit{purity}&\textbf{0.86}&0.571&0.49&0.234\\%
		&$\Gamma$\textit{--spread}&0.086&0.475&\textbf{0.037}&0.3\\%
		&$\Delta$\textit{--spread}&0.546&0.68&\textbf{0.537}&1.613\\%
		\hline
		\multirow{4}{*}{50}&\textit{ND{-}points}&\textbf{94}&4&24&1\\%
		&\textit{purity}&\textbf{0.94}&0.025&0.24&0.071\\%
		&$\Gamma$\textit{--spread}&\textbf{0.068}&0.544&0.128&0.461\\%
		&$\Delta$\textit{--spread}&\textbf{0.498}&1.845&0.618&0.956\\%
		\hline
		\multirow{4}{*}{100}&\textit{ND{-}points}&96&\textbf{109}&13&0\\%
		&\textit{purity}&\textbf{0.96}&0.948&0.13&0.0\\%
		&$\Gamma$\textit{--spread}&\textbf{0.136}&0.514&0.143&0.491\\%
		&$\Delta$\textit{--spread}&\textbf{0.641}&1.757&0.701&1.248\\%
		\hline%
	\end{tabular}%
\end{table}

Regarding the CEC09\_4 problem, whose results are reported in Figure \ref{fig::All - CEC09_4} and Table~\ref{tab::All - CEC09_4}, \texttt{NSMA} was the algorithm with the best overall performance: it generally obtained better metrics values than its most important competitors (\texttt{FPGA} and \texttt{NSGA-II}). Here, the combination of genetic operations and constrained steepest descent directions was greatly helpful to obtain remarkable results. Indeed, the independent use of only one of these two approaches did not lead to the same performance. In this problem, the \texttt{DMS} algorithm performed poorly regardless the value for $n$.

In conclusion, \texttt{NSMA} can be considered a viable option with convex problems, both in the low and the high dimensional cases. At the same time, our approach did not suffer with non-convex problems, as opposed to \texttt{FPGA}. On the contrary, it also outperformed \texttt{NSGA-II}, which is known to be a particularly suitable algorithm to use in these cases but struggles as the dimensionality of the problem grows.

\subsection{Overall comparison}
\label{subsec::performance-profiles}

In this last section of computational experiments, we provide the performance profiles for the four considered algorithms on the entire benchmark of problems, listed in Table~\ref{tab::problems}. The profiles are shown in Figure \ref{fig::PP_2m}.

\begin{figure*}[!t]
	\centering
	\subfloat[]{\includegraphics[width=2in]{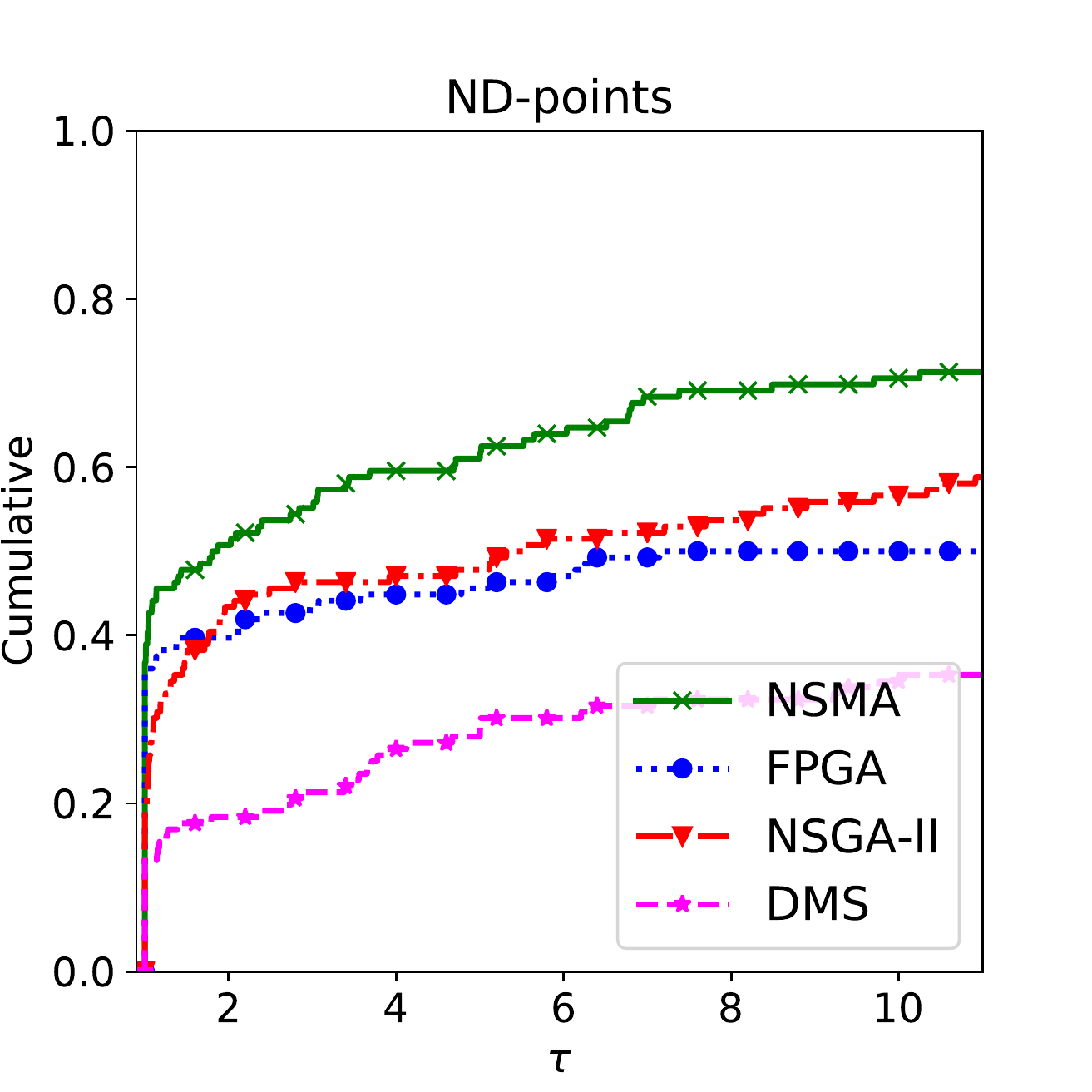}}
	\hfil
	\subfloat[]{\includegraphics[width=2in]{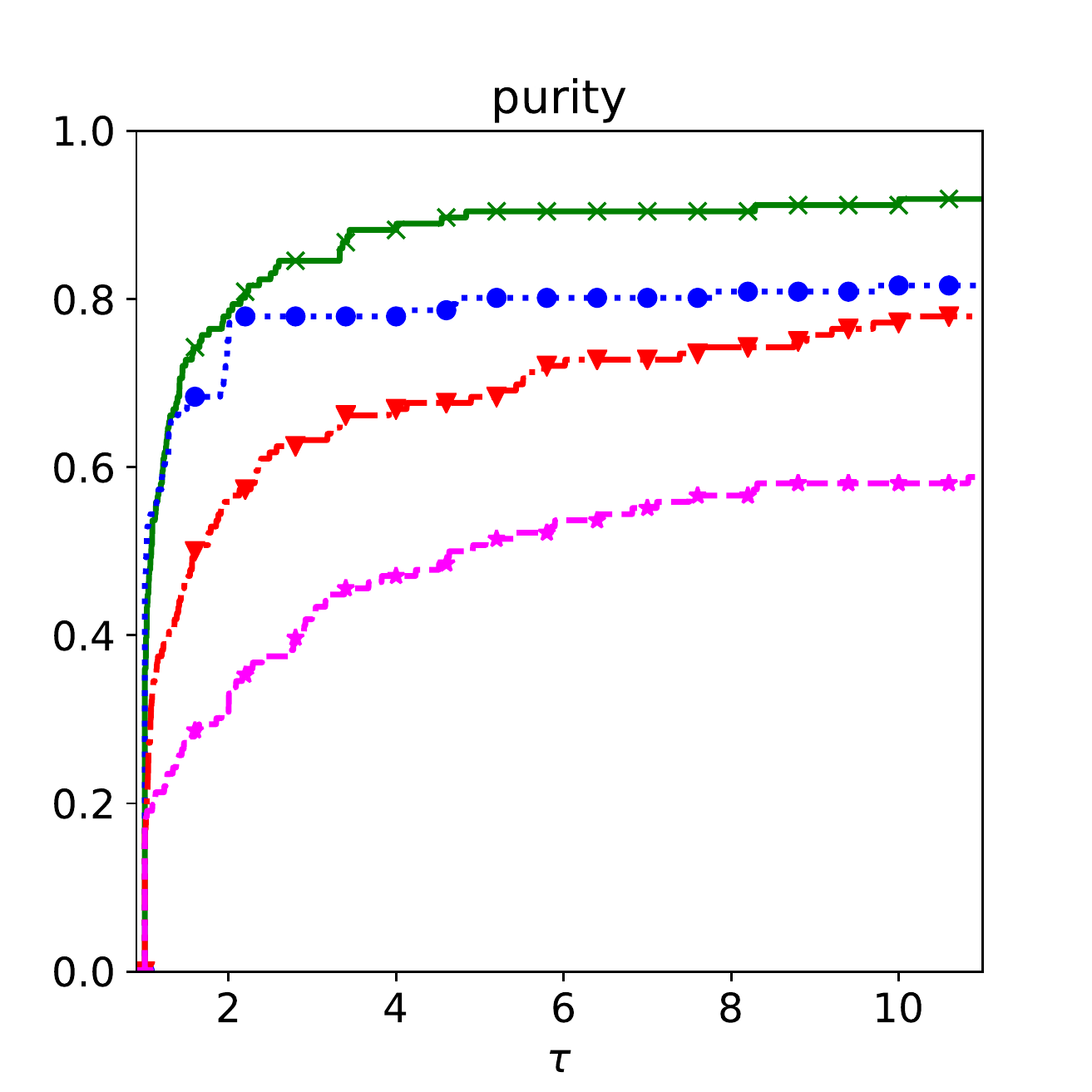}}
	\\
	\subfloat[]{\includegraphics[width=2in]{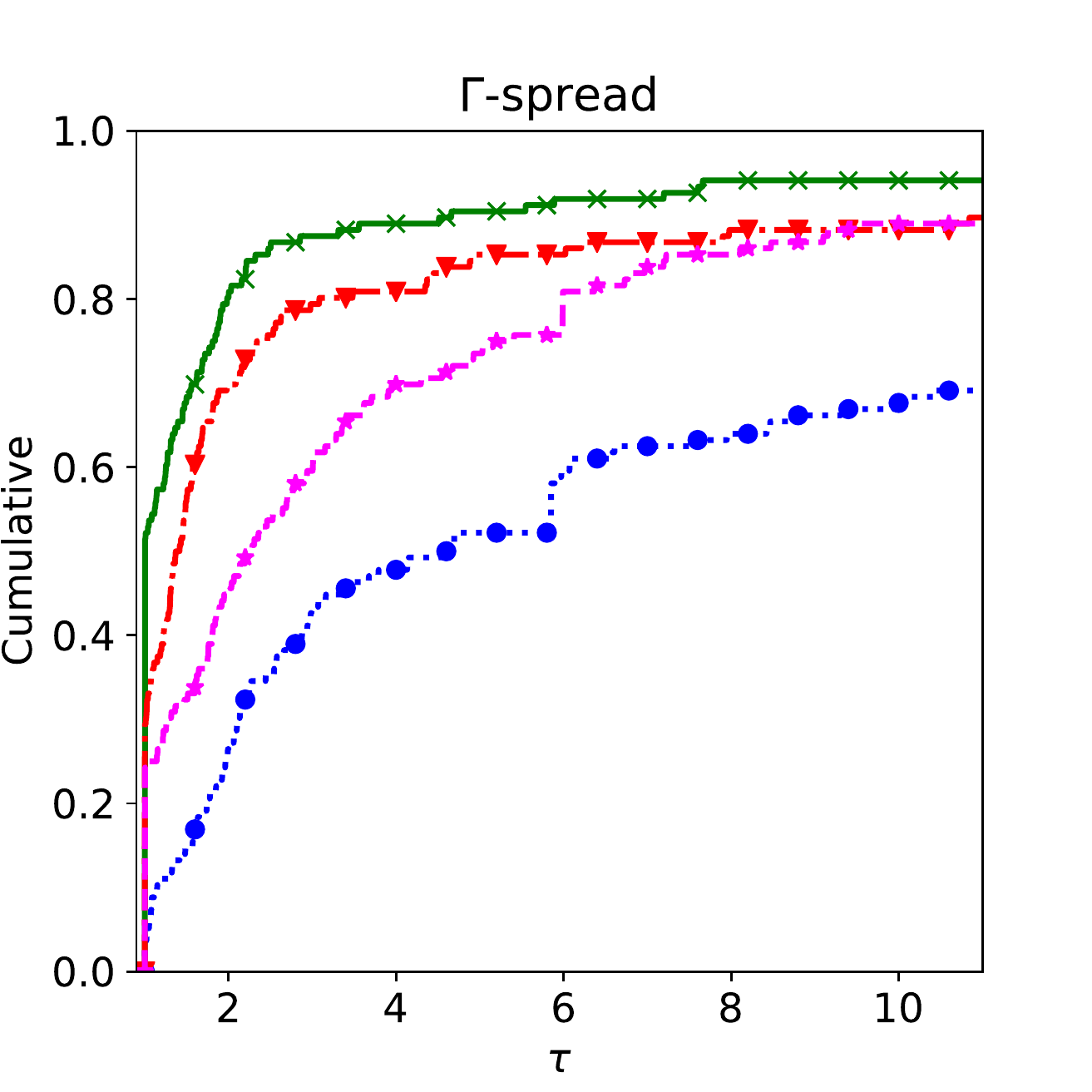}}
	\hfil
	\subfloat[]{\includegraphics[width=2in]{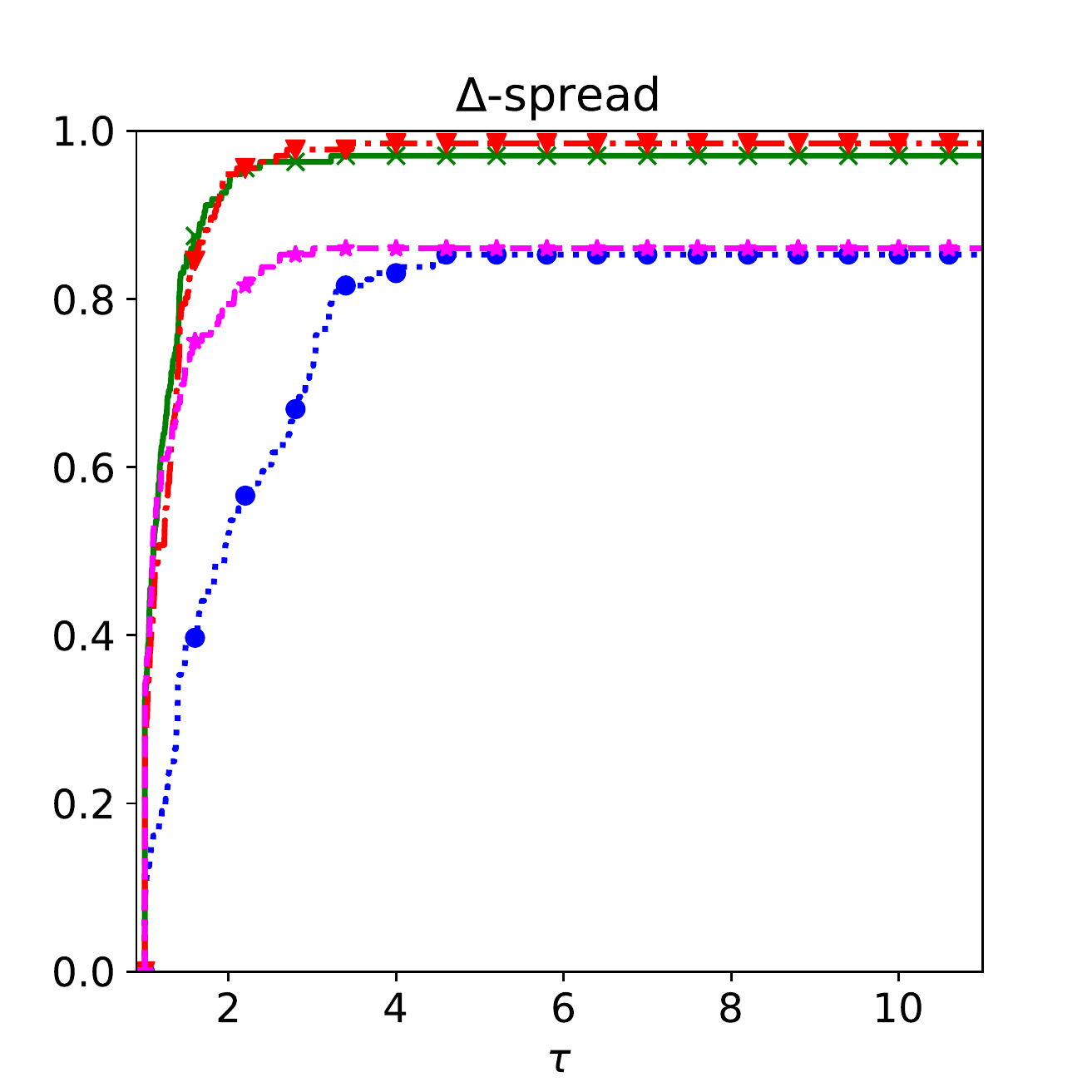}}
	\caption{Performance profiles for the \texttt{NSMA, FPGA, NSGA-II} and \texttt{DMS} algorithms on the CEC, ZDT, MOP and MAN problems, run with a time limit of 2 minutes (for interpretation of the references to color in text, the reader is referred to the electronic version of the article). (a) \textit{ND-points}. (b) \textit{purity}. (c) $\Gamma$\textit{--spread}. (d) $\Delta$\textit{--spread}.}
	\label{fig::PP_2m}
\end{figure*}

The performance profiles remark once again the benefits of using our proposed approach. Regarding the \textit{ND-points}, \texttt{NSMA} proved to be the most robust algorithm. This result was not obvious: we remind that our method, as opposed to \texttt{FPGA} and \texttt{DMS}, considers a fixed number of solutions in the population. 

Another interesting result is related to the \textit{purity} metric: \texttt{NSMA} is again the clear winner. In problems with complicated objective functions, local optimization of points in the \texttt{NSMA} mechanisms could result in a waste of computational time. From the results, however, we deduce that the converse is true: the combined use of constrained steepest descent directions and genetic operations allowed \texttt{NSMA} to achieve the best performance. 

The proposed method also outperformed the other ones in terms of $\Gamma$\textit{--spread}, while its performance is very similar to the one of \texttt{NSGA-II} in terms of $\Delta$\textit{--spread}. 
We can conclude that our approach is able to effectively obtain spread and uniform Pareto front approximations. At the same time, we deduce that the same cannot be said for the \texttt{FPGA}, which turned out to be the worst method w.r.t.\ the \textit{spread} metrics. However, the descent-based algorithm was the second best in terms of \textit{purity}, outperforming \texttt{NSGA-II}. 
In general, \texttt{DMS} was not effective overall on the considered benchmark.

\begin{figure*}[!t]
	\centering
	\subfloat[]{\includegraphics[width=2in]{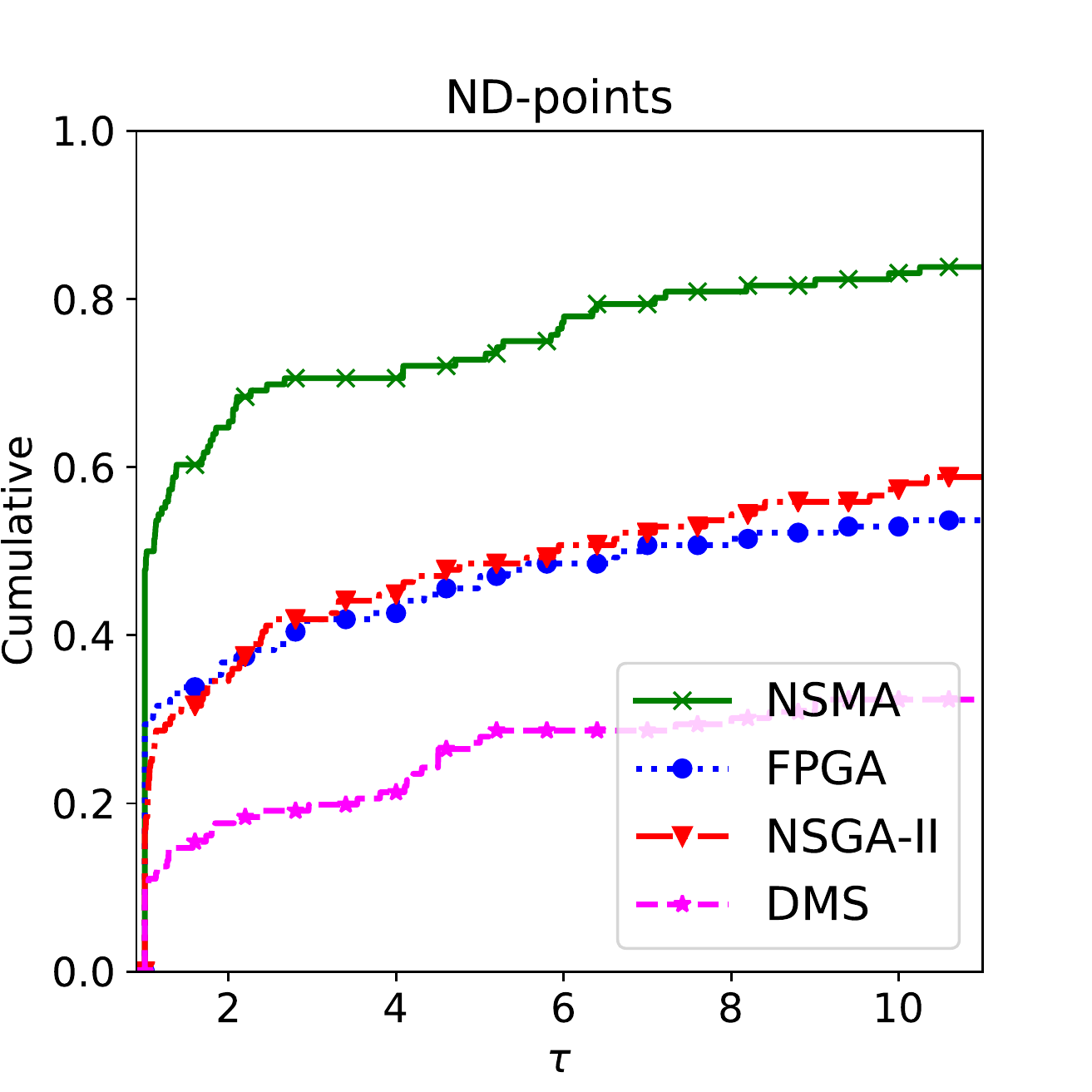}}
	\hfil
	\subfloat[]{\includegraphics[width=2in]{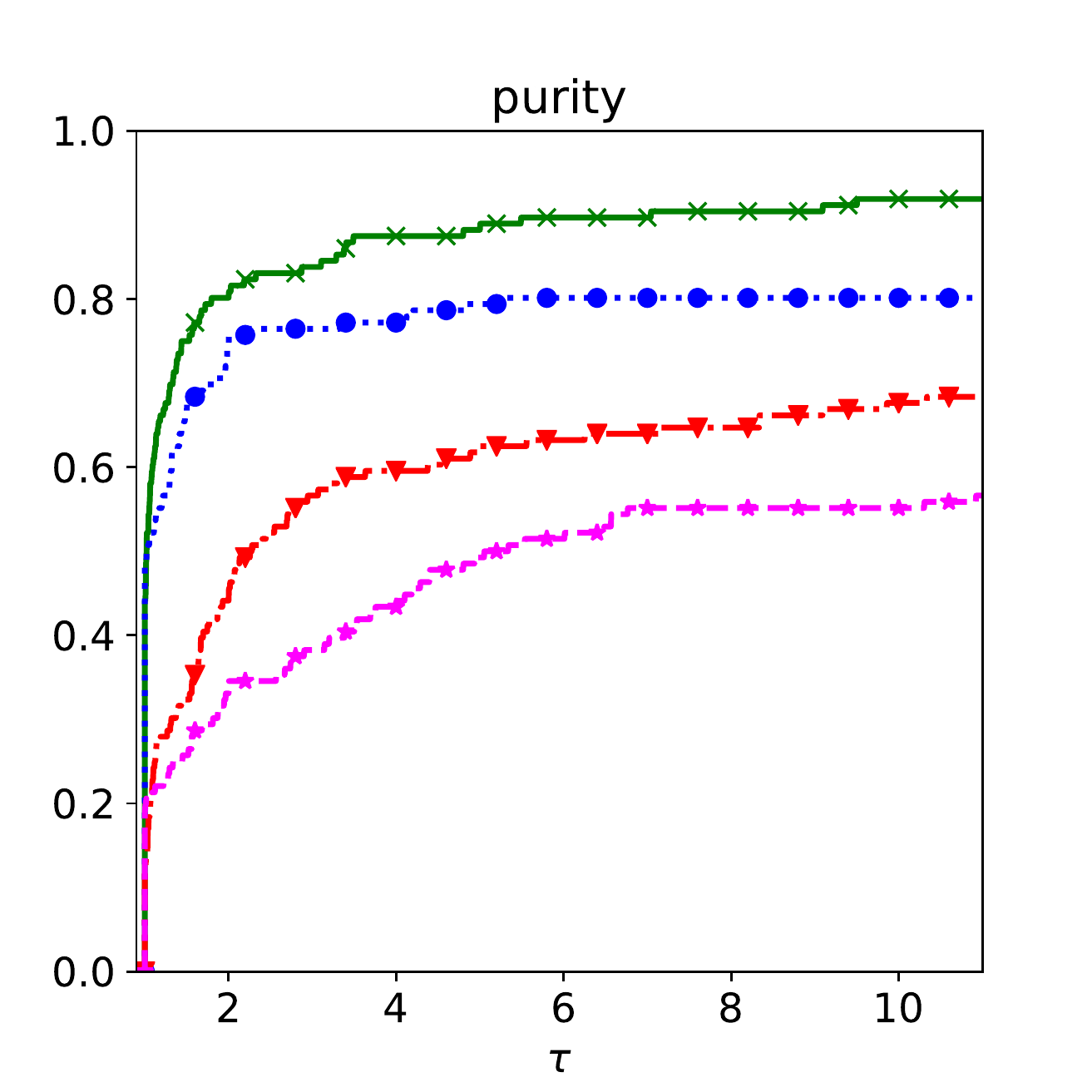}}
	\\
	\subfloat[]{\includegraphics[width=2in]{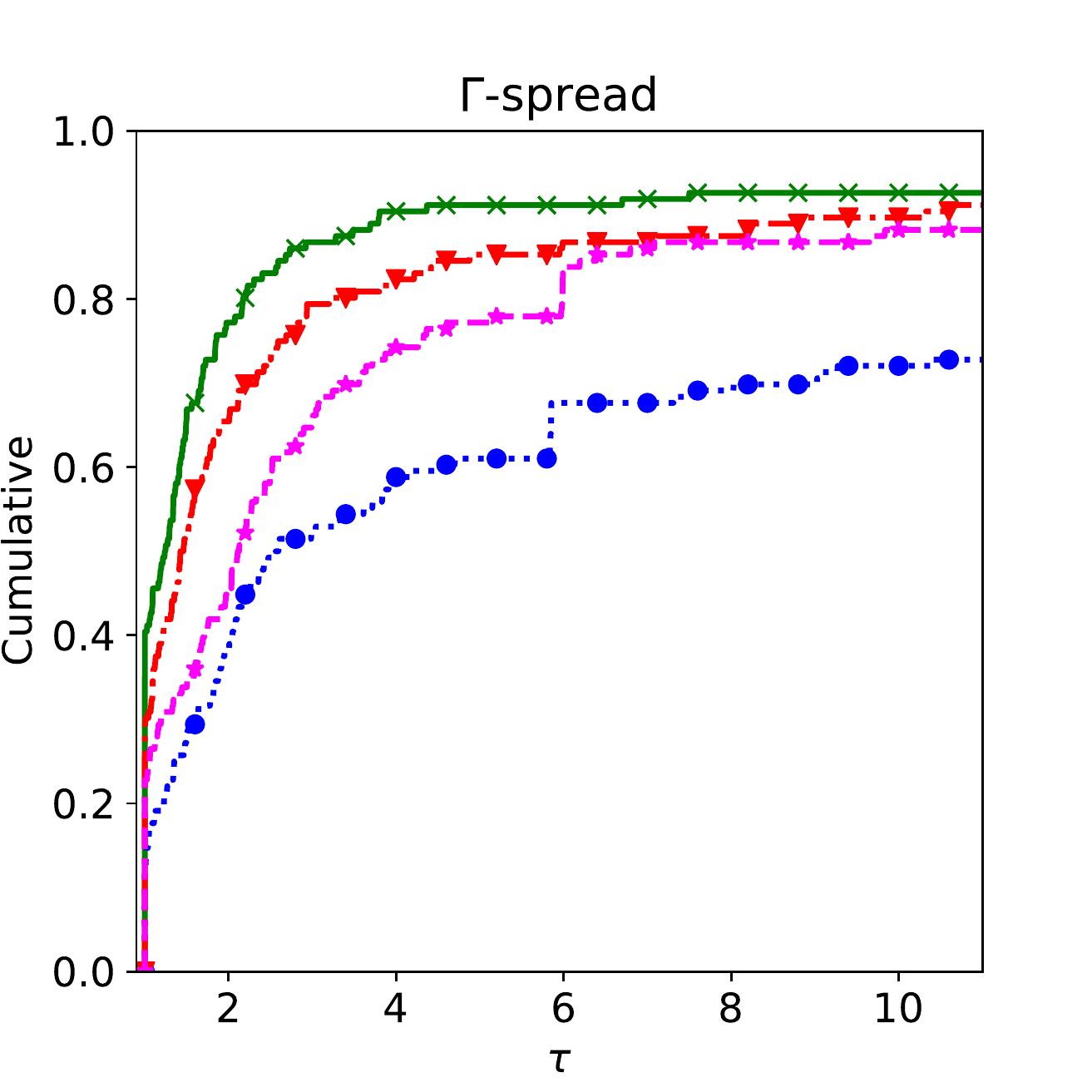}}
	\hfil
	\subfloat[]{\includegraphics[width=2in]{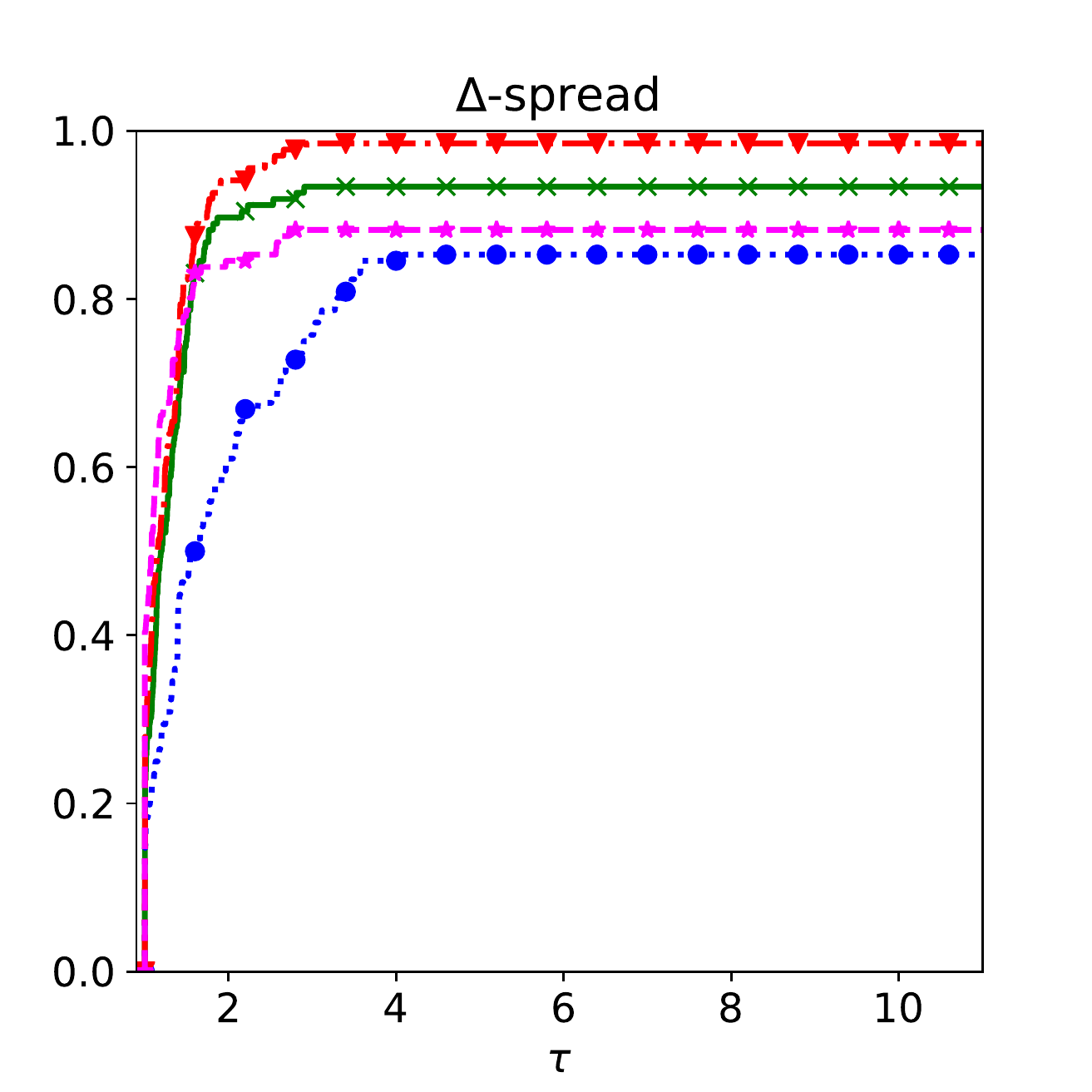}}
	\caption{Performance profiles for the \texttt{NSMA, FPGA, NSGA-II} and \texttt{DMS} algorithms on the CEC, ZDT, MOP and MAN problems, run with a time limit of 30 seconds (for interpretation of the references to color in text, the reader is referred to the electronic version of the article). (a) \textit{ND-points}. (b) \textit{purity}. (c) $\Gamma$\textit{--spread}. (d) $\Delta$\textit{--spread}.}
	\label{fig::PP_30s}
\end{figure*}

Lastly, we tested the four algorithms considering a time limit of 30 seconds for the experiments: the results can be seen in Figure \ref{fig::PP_30s}. Our aim is to observe the effectiveness of the methods at the first iterations. 

Considering the \textit{ND-points} and the \textit{purity} metrics, we observe that the differences between our approach and the other algorithms are now even clearer, while the situation is not changed in terms of $\Gamma$\textit{--spread}. Regarding the $\Delta$\textit{--spread} metric, \texttt{NSGA-II} was more effective than the other algorithms. However, our method was still competitive, as in terms of this metric it was the second most robust algorithm and it outperformed \texttt{FPGA} and \texttt{DMS}. We can conclude that \texttt{NSMA} turned out to be also effective considering a smaller time limit: from the very first iterations, our approach was capable to obtain good, wide and uniform Pareto front approximations. 
\section{Conclusions}
\label{sec::conclusions}

In this paper, we considered smooth multi-objective optimization problems subject to bound constraints. After a review of the existing literature, we listed and commented the main state-of-the-art approaches designed to approximate the Pareto front of such problems, along with their benefits and drawbacks. In particular, we focused on \texttt{NSGA-II} \cite{deb2002fast}, which is the most popular genetic algorithm, and on \texttt{FPGA}, which is a variant of the gradient-based descent method introduced in \cite{cocchi2020convergence}, capable of handling bound constraints. A detailed definition of \texttt{FPGA}, along with convergence properties, is provided in Appendix \ref{app::FPGA}. In a preliminary study, we compared these two algorithms trying to emphasize their strengths and weaknesses.

We then focused on the design of a memetic algorithm, whose aim is to combine the good features of both the aforementioned algorithms. We call this new method \textit{Non-dominated Sorting Memetic Algorithm} (\texttt{NSMA}). In this procedure, we exploit the genetic operations of \texttt{NSGA-II} and the tools typical of gradient-based descent methods, such as the steepest descent directions and line searches. In particular,
we employ a new descent method, called \textit{Front Multi-Objective Projected Gradient} (\texttt{FMOPG}), which is a front-based variant of the original \texttt{MOPG} firstly introduced in \cite{drummond2004projected}. For \texttt{FMOPG}, we proved properties of convergence to Pareto stationarity for the sequence of produced points. 

Moreover, results of thorough computational experiments in which we compared our method with main state-of-the-art algorithms, including \texttt{FPGA} and \texttt{NSGA-II}, are provided. These results show that \texttt{NSMA} can consistently outperform its competitors in terms of popular metrics for multi-objective optimization. Our approach turned out to be highly effective in any considered setting.

\bibliographystyle{abbrv}

\begin{thebibliography}{10}
	
	\bibitem{Bandyopadhyay2008}
	S.~{Bandyopadhyay}, S.~{Saha}, U.~{Maulik}, and K.~{Deb}.
	\newblock A simulated annealing-based multiobjective optimization algorithm:
	Amosa.
	\newblock {\em IEEE Transactions on Evolutionary Computation}, 12(3):269--283,
	2008.
	
	\bibitem{Bhuvana2015}
	J.~Bhuvana and C.~Aravindan.
	\newblock Memetic algorithm with preferential local search using adaptive
	weights for multi-objective optimization problems.
	\newblock {\em Soft Computing}, 20, 02 2015.
	
	\bibitem{brown2005directed}
	M.~Brown and R.~E. Smith.
	\newblock Directed multi-objective optimization.
	\newblock {\em International Journal of Computers, Systems, and Signals},
	6(1):3--17, 2005.
	
	\bibitem{CABASSIDE}
	F.~Cabassi and M.~Locatelli.
	\newblock Computational investigation of simple memetic approaches for
	continuous global optimization.
	\newblock {\em Computers \& Operations Research}, 72:50--70, 2016.
	
	\bibitem{campana2018multi}
	E.~F. Campana, M.~Diez, G.~Liuzzi, S.~Lucidi, R.~Pellegrini, V.~Piccialli,
	F.~Rinaldi, and A.~Serani.
	\newblock A multi-objective direct algorithm for ship hull optimization.
	\newblock {\em Computational Optimization and Applications}, 71(1):53--72,
	2018.
	
	\bibitem{carrizo2016trust}
	G.~A. Carrizo, P.~A. Lotito, and M.~C. Maciel.
	\newblock Trust region globalization strategy for the nonconvex unconstrained
	multiobjective optimization problem.
	\newblock {\em Mathematical Programming}, 159(1-2):339--369, 2016.
	
	\bibitem{carrizosa1998dominating}
	E.~Carrizosa and J.~B.~G. Frenk.
	\newblock Dominating sets for convex functions with some applications.
	\newblock {\em Journal of Optimization Theory and Applications},
	96(2):281--295, 1998.
	
	\bibitem{cocchi2020augmented}
	G.~Cocchi and M.~Lapucci.
	\newblock An augmented {L}agrangian algorithm for multi-objective optimization.
	\newblock {\em Computational Optimization and Applications}, 77(1):29--56,
	2020.
	
	\bibitem{COCCHI2021100008}
	G.~Cocchi, M.~Lapucci, and P.~Mansueto.
	\newblock Pareto front approximation through a multi-objective augmented
	lagrangian method.
	\newblock {\em EURO Journal on Computational Optimization}, page 100008, 2021.
	
	\bibitem{cocchi2020convergence}
	G.~Cocchi, G.~Liuzzi, S.~Lucidi, and M.~Sciandrone.
	\newblock On the convergence of steepest descent methods for multiobjective
	optimization.
	\newblock {\em Computational Optimization and Applications}, pages 1--27, 2020.
	
	\bibitem{cocchi2018implicit}
	G.~Cocchi, G.~Liuzzi, A.~Papini, and M.~Sciandrone.
	\newblock An implicit filtering algorithm for derivative-free multiobjective
	optimization with box constraints.
	\newblock {\em Computational Optimization and Applications}, 69(2):267--296,
	2018.
	
	\bibitem{custodio2011direct}
	A.~L. Cust{\'o}dio, J.~A. Madeira, A.~I.~F. Vaz, and L.~N. Vicente.
	\newblock Direct multisearch for multiobjective optimization.
	\newblock {\em SIAM Journal on Optimization}, 21(3):1109--1140, 2011.
	
	\bibitem{deb2002fast}
	K.~Deb, A.~Pratap, S.~Agarwal, and T.~Meyarivan.
	\newblock A fast and elitist multiobjective genetic algorithm: {NSGA-II}.
	\newblock {\em IEEE Transactions on Evolutionary Computation}, 6(2):182--197,
	2002.
	
	\bibitem{dolan2002benchmarking}
	E.~D. Dolan and J.~J. Mor{\'e}.
	\newblock Benchmarking optimization software with performance profiles.
	\newblock {\em Mathematical {P}rogramming}, 91(2):201--213, 2002.
	
	\bibitem{drugan2012}
	M.~Drugan and D.~Thierens.
	\newblock Stochastic pareto local search: Pareto neighbourhood exploration and
	perturbation strategies.
	\newblock {\em Journal of Heuristics}, 18, 10 2012.
	
	\bibitem{drummond2004projected}
	L.~G. Drummond and A.~N. Iusem.
	\newblock A projected gradient method for vector optimization problems.
	\newblock {\em Computational Optimization and applications}, 28(1):5--29, 2004.
	
	\bibitem{drummond2008choice}
	L.~G. Drummond, N.~Maculan, and B.~F. Svaiter.
	\newblock On the choice of parameters for the weighting method in vector
	optimization.
	\newblock {\em Mathematical Programming}, 111(1-2):201--216, 2008.
	
	\bibitem{eichfelder2009adaptive}
	G.~Eichfelder.
	\newblock An adaptive scalarization method in multiobjective optimization.
	\newblock {\em SIAM Journal on Optimization}, 19(4):1694--1718, 2009.
	
	\bibitem{Filatovas2017859}
	E.~Filatovas, A.~Lančinskas, O.~Kurasova, and J.~Žilinskas.
	\newblock A preference-based multi-objective evolutionary algorithm r-nsga-ii
	with stochastic local search.
	\newblock {\em Central European Journal of Operations Research},
	25(4):859--878, 2017.
	
	\bibitem{fliege2009newton}
	J.~Fliege, L.~G. Drummond, and B.~F. Svaiter.
	\newblock Newton's method for multiobjective optimization.
	\newblock {\em SIAM Journal on Optimization}, 20(2):602--626, 2009.
	
	\bibitem{fliege2000steepest}
	J.~Fliege and B.~F. Svaiter.
	\newblock Steepest descent methods for multicriteria optimization.
	\newblock {\em Mathematical Methods of Operations Research}, 51(3):479--494,
	2000.
	
	\bibitem{fliege2016method}
	J.~Fliege and A.~I.~F. Vaz.
	\newblock A method for constrained multiobjective optimization based on {SQP}
	techniques.
	\newblock {\em SIAM Journal on Optimization}, 26(4):2091--2119, 2016.
	
	\bibitem{Fukuda2014}
	E.~Fukuda and L.~Drummond.
	\newblock A survey on multiobjective descent methods.
	\newblock {\em Pesquisa Operacional}, 34:585--620, 09 2014.
	
	\bibitem{fukuda2011convergence}
	E.~H. Fukuda and L.~G. Drummond.
	\newblock On the convergence of the projected gradient method for vector
	optimization.
	\newblock {\em Optimization}, 60(8-9):1009--1021, 2011.
	
	\bibitem{fukuda2013inexact}
	E.~H. Fukuda and L.~G. Drummond.
	\newblock Inexact projected gradient method for vector optimization.
	\newblock {\em Computational Optimization and Applications}, 54(3):473--493,
	2013.
	
	\bibitem{fukuda2019barrier}
	E.~H. Fukuda, L.~G. Drummond, and F.~M. Raupp.
	\newblock A barrier-type method for multiobjective optimization.
	\newblock {\em Optimization}, pages 1--17, 2019.
	
	\bibitem{gonccalves2020globally}
	M.~Gon{\c{c}}alves, F.~Lima, and L.~Prudente.
	\newblock Globally convergent newton-type methods for multiobjective
	optimization.
	\newblock 2020.
	
	\bibitem{gravel1992multicriterion}
	M.~Gravel, J.~M. Martel, R.~Nadeau, W.~Price, and R.~Tremblay.
	\newblock A multicriterion view of optimal resource allocation in job-shop
	production.
	\newblock {\em European Journal of Operational Research}, 61(1-2):230--244,
	1992.
	
	\bibitem{GRIBELHG}
	D.~Gribel and T.~Vidal.
	\newblock Hg-means: A scalable hybrid genetic algorithm for minimum
	sum-of-squares clustering.
	\newblock {\em Pattern Recognition}, 88:569--583, 2019.
	
	\bibitem{grosso2007population}
	A.~Grosso, M.~Locatelli, and F.~Schoen.
	\newblock A population-based approach for hard global optimization problems
	based on dissimilarity measures.
	\newblock {\em Mathematical Programming}, 110(2):373--404, 2007.
	
	\bibitem{hu2003hybridization}
	X.~Hu, Z.~Huang, and Z.~Wang.
	\newblock Hybridization of the multi-objective evolutionary algorithms and the
	gradient-based algorithms.
	\newblock In {\em The 2003 Congress on Evolutionary Computation, 2003.
		CEC'03.}, volume~2, pages 870--877. IEEE, 2003.
	
	\bibitem{Huband2006}
	S.~Huband, P.~Hingston, L.~Barone, and L.~While.
	\newblock A review of multiobjective test problems and a scalable test problem
	toolkit.
	\newblock {\em IEEE Transactions on Evolutionary Computation}, 10(5):477--506,
	2006.
	
	\bibitem{Kim2014290}
	H.~Kim and M.-S. Liou.
	\newblock Adaptive directional local search strategy for hybrid evolutionary
	multiobjective optimization.
	\newblock {\em Applied Soft Computing Journal}, 19:290--311, 2014.
	
	\bibitem{Lara2010}
	A.~{Lara}, G.~{Sanchez}, C.~A.~C. {Coello}, and O.~{Schutze}.
	\newblock Hcs: A new local search strategy for memetic multiobjective
	evolutionary algorithms.
	\newblock {\em IEEE Transactions on Evolutionary Computation}, 14(1):112--132,
	2010.
	
	\bibitem{laumanns2002combining}
	M.~Laumanns, L.~Thiele, K.~Deb, and E.~Zitzler.
	\newblock Combining convergence and diversity in evolutionary multiobjective
	optimization.
	\newblock {\em Evolutionary Computation}, 10(3):263--282, 2002.
	
	\bibitem{liu2007multiobjective}
	D.~Liu, K.~C. Tan, C.~K. Goh, and W.~K. Ho.
	\newblock A multiobjective memetic algorithm based on particle swarm
	optimization.
	\newblock {\em IEEE Transactions on Systems, Man, and Cybernetics, Part B
		(Cybernetics)}, 37(1):42--50, 2007.
	
	\bibitem{liu2017improved}
	T.~Liu, X.~Gao, and Q.~Yuan.
	\newblock An improved gradient-based nsga-ii algorithm by a new chaotic map
	model.
	\newblock {\em Soft Computing}, 21(23):7235--7249, 2017.
	
	\bibitem{liuzzi2016derivative}
	G.~Liuzzi, S.~Lucidi, and F.~Rinaldi.
	\newblock A derivative-free approach to constrained multiobjective nonsmooth
	optimization.
	\newblock {\em SIAM Journal on Optimization}, 26(4):2744--2774, 2016.
	
	\bibitem{LOCATELLIDE}
	M.~Locatelli, M.~Maischberger, and F.~Schoen.
	\newblock Differential evolution methods based on local searches.
	\newblock {\em Computers \& Operations Research}, 43:169--180, 2014.
	
	\bibitem{locatelli2013global}
	M.~Locatelli and F.~Schoen.
	\newblock {\em Global optimization: theory, algorithms, and applications}.
	\newblock SIAM, 2013.
	
	\bibitem{mandal_memetic_2015}
	S.~K. Mandal, D.~Pacciarelli, A.~LØkketangen, and G.~Hasle.
	\newblock A memetic {NSGA}-{II} for the bi-objective mixed capacitated general
	routing problem.
	\newblock {\em Journal of Heuristics}, 21(3):359--390, June 2015.
	\newblock Number: 3.
	
	\bibitem{MANSUETO2021107849}
	P.~Mansueto and F.~Schoen.
	\newblock Memetic differential evolution methods for clustering problems.
	\newblock {\em Pattern Recognition}, 114:107849, 2021.
	
	\bibitem{mostaghim2007multi}
	S.~Mostaghim, J.~Branke, and H.~Schmeck.
	\newblock Multi-objective particle swarm optimization on computer grids.
	\newblock In {\em Proceedings of the 9th annual conference on Genetic and
		evolutionary computation}, pages 869--875. ACM, 2007.
	
	\bibitem{palermo2003system}
	G.~Palermo, C.~Silvano, S.~Valsecchi, and V.~Zaccaria.
	\newblock A system-level methodology for fast multi-objective design space
	exploration.
	\newblock In {\em Proceedings of the 13th ACM Great Lakes symposium on VLSI},
	pages 92--95. ACM, 2003.
	
	\bibitem{pascoletti1984scalarizing}
	A.~Pascoletti and P.~Serafini.
	\newblock Scalarizing vector optimization problems.
	\newblock {\em Journal of Optimization Theory and Applications},
	42(4):499--524, 1984.
	
	\bibitem{pellegrini2014application}
	R.~Pellegrini, E.~Campana, M.~Diez, A.~Serani, F.~Rinaldi, G.~Fasano, U.~Iemma,
	G.~Liuzzi, S.~Lucidi, and F.~Stern.
	\newblock Application of derivative-free multi-objective algorithms to
	reliability-based robust design optimization of a high-speed catamaran in
	real ocean environment1.
	\newblock {\em Engineering Optimization IV-Rodrigues et al.(Eds.)}, page~15,
	2014.
	
	\bibitem{shukla2007gradient}
	P.~K. Shukla.
	\newblock On gradient based local search methods in unconstrained evolutionary
	multi-objective optimization.
	\newblock In {\em International Conference on Evolutionary Multi-Criterion
		Optimization}, pages 96--110. Springer, 2007.
	
	\bibitem{Sindhya2013495}
	K.~Sindhya, K.~Miettinen, and K.~Deb.
	\newblock A hybrid framework for evolutionary multi-objective optimization.
	\newblock {\em IEEE Transactions on Evolutionary Computation}, 17(4):495--511,
	2013.
	
	\bibitem{sun2016multi}
	Y.~Sun, D.~W.~K. Ng, J.~Zhu, and R.~Schober.
	\newblock Multi-objective optimization for robust power efficient and secure
	full-duplex wireless communication systems.
	\newblock {\em IEEE Transactions on Wireless Communications}, 15(8):5511--5526,
	2016.
	
	\bibitem{tavana2004subjective}
	M.~Tavana.
	\newblock A subjective assessment of alternative mission architectures for the
	human exploration of {M}ars at {NASA} using multicriteria decision making.
	\newblock {\em Computers \& Operations Research}, 31(7):1147--1164, 2004.
	
	\bibitem{Tiwari2009}
	S.~{Tiwari}, G.~{Fadel}, P.~{Koch}, and K.~{Deb}.
	\newblock Performance assessment of the hybrid archive-based micro genetic
	algorithm (amga) on the cec09 test problems.
	\newblock In {\em 2009 IEEE Congress on Evolutionary Computation}, pages
	1935--1942, 2009.
	
	\bibitem{villalobos2018memetic}
	M.~Villalobos-Cid, M.~Dorn, R.~Ligabue-Braun, and M.~Inostroza-Ponta.
	\newblock A memetic algorithm based on an nsga-ii scheme for phylogenetic tree
	inference.
	\newblock {\em IEEE Transactions on Evolutionary Computation}, 23(5):776--787,
	2018.
	
	\bibitem{wang2011multi}
	X.~Wang, C.~Hirsch, S.~Kang, and C.~Lacor.
	\newblock Multi-objective optimization of turbomachinery using improved nsga-ii
	and approximation model.
	\newblock {\em Computer Methods in Applied Mechanics and Engineering},
	200(9-12):883--895, 2011.
	
	\bibitem{white1998epsilon}
	D.~White.
	\newblock Epsilon-dominating solutions in mean-variance portfolio analysis.
	\newblock {\em European Journal of Operational Research}, 105(3):457--466,
	1998.
	
	\bibitem{zhang_multiobjective_2009}
	Q.~Zhang, A.~Zhou, S.~Zhao, P.~Suganthan, W.~Liu, and S.~Tiwari.
	\newblock Multiobjective optimization test instances for the cec 2009 special
	session and competition.
	\newblock {\em Mechanical Engineering}, 01 2008.
	
	\bibitem{ziztler2000}
	E.~Zitzler, K.~Deb, and L.~Thiele.
	\newblock Comparison of multiobjective evolutionary algorithms: Empirical
	results.
	\newblock {\em Evol. Comput.}, 8(2):173–195, June 2000.
	
\end{thebibliography}

\renewcommand{\theequation}{A\arabic{equation}}
\setcounter{equation}{0}

\renewcommand{\thetable}{C\Roman{table}}
\setcounter{table}{0}
\renewcommand{\theHtable}{Supplement\thetable}

\renewcommand{\theproposition}{A\arabic{proposition}}
\setcounter{proposition}{0}
\renewcommand{\theassumption}{A\arabic{assumption}}
\setcounter{assumption}{0}

\renewcommand{\thealgocf}{A\arabic{algocf}}
\setcounter{algocf}{0}

{\appendices
	\section{The Front Projected Gradient Algorithm}
\label{app::FPGA}

In this appendix, we describe the adaptation of the \texttt{FSDA} algorithm \cite{cocchi2020convergence} to box-constrained optimization problems, which we call \textit{Front Projected Gradient Algorithm} (\texttt{FPGA}). We initially report the scheme of the new adaptation. Then, in the remainder of the appendix, we provide a rigorous theoretical analysis.

\subsection{Algorithmic scheme}
\label{subapp::algorithmic-scheme}

We report the scheme of \texttt{FPGA} in Algorithm \ref{alg::FPGA}. 

\begin{algorithm}[h]
	\caption{Front Projected Gradient Algorithm} \label{alg::FPGA}
	Input: $F:\mathbb{R}^n \rightarrow \mathbb{R}^m$, $\Omega$ feasible closed and convex set, $X_0$ set of feasible non-dominated points w.r.t.\ $F$. \\
	$k = 0$\\
	\While{a stopping criterion is not satisfied}{
		$\hat{X}_k = X_k$ \\
		\For{$c = 1,\ldots, |X_k|$}{
			\If{$x_c \in \hat{X}_k$ \textbf{and} $\exists I \subseteq \{1,\ldots, m\}$ such that 
				\begin{itemize}
					\item $x_c \in \hat{X}_k^I$
					\item $\theta_\Omega^I(x_c) < 0$
			\end{itemize}}{
				Let $d_{\Omega c}^I \in v_\Omega^I(x_c)$ be the direction associated with $\theta_\Omega^I(x_c)$ \\
				$\alpha$ = \texttt{B-FALS}($F(\cdot), \Omega, I, \hat{X}_k^I, x_c, d_{\Omega c}^I, \theta_\Omega^I(x_c)$) \\
				$\hat{X}_k = \hat{X}_k \setminus \{y \in \hat{X}_k | F(x_c + \alpha d_{\Omega c}^I) \lneqq F(y)\} \cup \{x_c + \alpha d_{\Omega c}^I\}$\label{line::point-insert}
			}
			\Else{
				\If{$x_c \in \hat{X}_k$ (i.e.\ $x_c$ has not been filtered out by previous inner iterations)}{
					$x_c$ is a Pareto-stationary point w.r.t.\ $F$ for Problem \eqref{eq::mo-prob}.
				}
			}
		}
		$X_{k + 1} = \hat{X}_k$ \\
		$k = k + 1$
	}
	\Return $X_k$
\end{algorithm}

\noindent With respect to the \texttt{FSDA} algorithm, there are two major differences.

\begin{itemize}
	\item For the descent directions, we solve instances of Problem \eqref{eq::proj-desc}, while in the original implementation Problem \eqref{eq::ste-com-desc} is considered.
	\item We employ \texttt{B-FALS} (Algorithm \ref{alg::B-FALS}) instead of \texttt{FALS} (Algorithm \ref{alg::FALS}) in order to handle bound constraints.
\end{itemize}

In the remainder of the appendix, we provide some \texttt{FPGA} properties.

\subsection{Algorithm analysis}
\label{subapp::algorithm-analysis}

In this subsection, we provide a formal analysis of the \texttt{FPGA} algorithm from a theoretical perspective.

We first prove the feasibility of the points produced by the algorithm.

\begin{proposition}
	Let $\{X_k\}$ be the sequence of sets of points generated by \texttt{FPGA}. Then, for all $k$, every point $x_c$ in the set $X_k$ is feasible for Problem \eqref{eq::mo-prob}.
\end{proposition}
\begin{proof}
	The proof is straightforward. First of all, the initial set $X_0$ is composed by feasible points. New solutions are only added through Line \ref{line::point-insert}. Considering that $\Omega$ is convex and $d_{\Omega c}^I$ is a feasible direction by construction for any $I \subseteq \{1,\ldots, m\}$, and reminding the stopping criteria of \texttt{B-FALS}, these new points are contained in $\Omega$ and, therefore, they are feasible for Problem \eqref{eq::mo-prob}.
\end{proof}

In order to prove convergence properties, we need the concept of linked sequence (Definition \ref{def::link-seq}) and the following assumption.

\begin{assumption}
	\label{ass::FPGA-boundedness}
	Let $X_0$ be a set of feasible non-dominated points w.r.t.\ $F$. A point $x_0 \in X_0$ exists such that:
	\begin{itemize}
		\item $x_0$ is not Pareto-stationary w.r.t.\ $F$;
		\item the set $\mathcal{L}(x_0) = \bigcup_{j=1}^m\{x \in \Omega: f_j(x) \le f_j(x_0)\}$ is compact.
	\end{itemize}
\end{assumption}

This latter one is similar to Assumption 1 in \cite{cocchi2020convergence}. The difference is that, in this case, bound constraints must be also taken into account.

\begin{proposition}
	Let us assume that Assumption \ref{ass::FPGA-boundedness} holds. Let $\{X_k\}$ be the sequence of sets of non-dominated points w.r.t.\ $F$ produced by \texttt{FPGA}. Let $\{x_k\}$ be a linked sequence, then it admits limit points and every limit point is Pareto-stationary w.r.t.\ $F$.
\end{proposition}
\begin{proof}
	The proof is almost identical to the one of Proposition 5 in \cite{cocchi2020convergence}. There is only one difference. After proving that 
	\begin{equation}
		\label{eq::alpha-lim}
		\lim_{\substack{k\to\infty \\ k \in K}}\alpha_{k + 1} = 0
	\end{equation}
	\cite[Equation 22]{cocchi2020convergence}, where $K$ indicates a subsequence, the \texttt{FSDA} authors consider sufficiently large values of $k$ such that $\alpha_{k + 1} < \alpha_0$. In this case, the steps of \texttt{FALS} and the definition of $X_k$, $k \in K$, imply that there exists $y_k \in X_k$ such that 
	\begin{equation*}
		F_I(y_k) + \mathbf{1}\beta\frac{\alpha_{k + 1}}{\delta}\theta^I(x_k) < F_I\left(x_k + \frac{\alpha_{k + 1}}{\delta}d_k^I\right).
	\end{equation*}

	With respect to \texttt{FALS}, \texttt{B-FALS} has an additional stopping criterion: the step size must lead to a point that is feasible for Problem \eqref{eq::mo-prob}. In this context, $\alpha_{k + 1} / \delta \le \alpha_0$ might not have been selected because the point $x_k + (\alpha_{k + 1} / \delta)d_{\Omega k}^I \not \in \Omega$. However, through a little modification, we can handle this additional stopping criterion. 
	
	First of all, Equation \eqref{eq::alpha-lim} still holds: the proof of this statement is the same provided in \cite{cocchi2020convergence}. Then, we can consider sufficiently large values of $k$ such that 
	\begin{equation}
		\alpha_{k + 1} < \frac{\alpha_{k + 1}}{\delta} \le 1.
	\end{equation}
	In this way, since $\Omega$ is convex and $d_{\Omega k}^I$ is a feasible direction by construction, the points produced by the two step sizes are feasible, i.e., the \texttt{B-FALS} feasibility stopping criterion is satisfied. Then, the steps of \texttt{B-FALS} and the definition of $X_k$, $k \in K$, imply that there exists $y_k \in X_k$ such that 
	\begin{equation*}
		F_I(y_k) + \mathbf{1}\beta\frac{\alpha_{k + 1}}{\delta}\theta_\Omega^I(x_k) < F_I\left(x_k + \frac{\alpha_{k + 1}}{\delta}d_{\Omega k}^I\right).
	\end{equation*}
	From this point forward, we can follow the remainder of the proof of Proposition 5 in \cite{cocchi2020convergence} in order to prove the thesis.
\end{proof}
}

\end{document}